\documentclass[a4paper]{siamltex}

\newcommand{\tabincell}[2]{\begin{tabular}{@{}#1@{}}#2\end{tabular}}

\usepackage{amsmath,latexsym,amssymb}
\usepackage{epsf}
\usepackage{graphicx}
\usepackage{mathrsfs}
\usepackage{makeidx}
\makeindex
\usepackage{color}
\usepackage[linesnumbered,ruled]{algorithm2e}
\usepackage{float}

\usepackage{setspace}
\usepackage{fixltx2e}

\usepackage{stmaryrd}

\usepackage[colorlinks,linkcolor=black,anchorcolor=black,citecolor=black]{hyperref}
\usepackage{hyperref}

\title{Gradient flow based discretized Kohn-Sham density functional theory
\thanks{
This work was partially supported by the National Natural Science
Foundation of China under grant 91730302 and 11671389 and the Key Research Program of Frontier Sciences of the Chinese Academy of Sciences under grant QYZDJ-SSW-SYS010.
}}
\author{Xiaoying Dai\footnotemark[2]
\and
Qiao Wang\footnotemark[2]
\and
Aihui Zhou\thanks{LSEC, Institute of Computational Mathematics and Scientific/Engineering Computing, Academy of Mathematics and Systems Science, Chinese Academy of Sciences, Beijing 100190, China, and School of Mathematical Sciences, University of Chinese Academy of Sciences, Beijing 100049, China (daixy@lsec.cc.ac.cn, qwang@lsec.cc.ac.cn, azhou@lsec.cc.ac.cn).}
}

\date{}

\newtheorem{prop}[theorem]{Proposition}

\newtheorem{rema}[theorem]{Remark}

\begin{document}

\maketitle

\begin{abstract}
    In this paper, we propose and analyze a gradient flow based Kohn-Sham density functional theory. First, we prove that the critical point of the gradient flow based model can be a local minimizer of the Kohn-Sham total energy. Then we apply a midpoint scheme to carry out the temporal discretization. It is shown that the critical point of the Kohn-Sham energy can be well-approximated by the scheme. In particular, based on the midpoint scheme, we design an orthogonality preserving iteration scheme to minimize the Kohn-Sham energy and show that the orthogonality preserving iteration scheme produces  approximations that are orthogonal and convergent to a local minimizer under reasonable assumptions. Finally, we report numerical experiments that support our theory.
\end{abstract}\vskip 0.2cm

\vskip 0.2cm


\section{Introduction}
    Kohn-Sham density functional theory (DFT) is the most wi\-dely used model in electronic structure calculations \cite{Kohn_Sham3}.
    We see that to solve the Kohn-Sham equation, which is a nonlinear eigenvalue problem, some self-consistent field (SCF) iterations are demanded \cite{dai2011finite, dai2008three,Ele_Str3}.
    However, the convergence of SCF iterations is not guaranteed, especially for large scale systems with small band gaps, for which the performance of the SCF iterations is unpredictable \cite{dai2017conjugate,SCF2}.  It has been shown by numerical experiments that the SCF iterations usually converge for systems with larger gap between the occupied orbitals and the remainder \cite{SCF1}. We understand that there are a number of works trying to illustrate this phenomenon and see that SCF iterations do converge if the gap is uniformly large enough locally or globally \cite{cai2017eigenvector, liu2014convergence, liu2015analysis, SCF1}.

In order to obtain approximations of  the Kohn-Sham DFT that are convergent,   in recent two decades, people pay much attention to study the direct energy minimization model. Instead of solving the Kohn-Sham equation, people minimize the Kohn-Sham total energy under an orthogonality constraint \cite{dai2017conjugate, Opt1, liu2014convergence, liu2015analysis, Opt3, Dyn2, Opt4, Opt6, Str_Pre2, SCF2, BaiZhengjian}. It is shown in \cite{Opt7} that a monotonic optimization approach may produce a locally convergent approximations. We observe that the iterations based on the optimization should be carefully carried out due to the orthogonality constraint, for which the existing methods are indeed either retraction (see, e.g., \cite{dai2017conjugate, SCF2}) or manifold path optimization approaches (see, e.g., \cite{dai2017conjugate, Str_Pre2, SCF2}). We see that some backtracking should be applied in a monotonic optimization method,  due to not only theory but also practice.

    In this paper, we introduce and analyze a gradient flow based discretized Kohn-Sham DFT for electronic structure calculations. First, we prove that our gradient flow based discretized Kohn-Sham DFT preserves orthogonality and models the ground state well. We then propose a midpoint scheme to carry out the temporal discretization, which is of orthogonality preserving, too. We mention that our numerical scheme avoids a retraction process and does not need any backtracking. Based on the midpoint scheme, finally, we design and analyze an orthogonality preserving iteration scheme for solving the discretized Kohn-Sham energy. It is shown by theory and numerics that our orthogonality preserving iteration scheme is convergent provided some reasonable assumption.

    For illustration, we provide Figure \ref{Table: Differences} to show the differences among the three approaches. In the midpoint scheme of the gradient flow based model (blue dashed line with square symbol endpoint), the auxiliary point of midpoint scheme of the dynamical system is inside the manifold. In the manifold path method (black solid line with circle symbol endpoint), the path is on the manifold and the energy is decreasing when the iteration is moving along the path. In the retraction method (red solid line with triangle symbol endpoint), the auxiliary point is in the tangent space and outside the manifold.

    \begin{table}
      \begin{tabular}{c|c|c|c}
                                & \tabincell{c}{Gradient Flow\\ Based Model}   & \tabincell{c}{Minimization\\Model} &   \tabincell{c}{Eigenvalue\\Model}\\
        \hline
        \tabincell{c}{Orthogona\\-lization\\requirement}   & No                     & \tabincell{c}{Yes(retraction method),\\no(manifold path method)}    &Yes\\
        \hline
        \tabincell{c}{Auxiliary\\points\\ location}& \tabincell{c}{Inside\\manifold\\(midpoint\\scheme)}    &\tabincell{c}{Outside manifold\\(retraction method),\\on manifold \\(manifold path method)}&\tabincell{c}{On\\manifold}\\
        \hline
        \tabincell{c}{Energy \\decreasing}&Yes  &   Yes &May not\\
        \hline
        \tabincell{c}{Convergence\\result\\assumptions}& \tabincell{c}{Local\\uniqueness\\of\\minimizer}&
        \tabincell{c}{Local\\uniqueness\\of\\minimizer}&\tabincell{c}{Eigenvalue gap\\ large enough\\(depending on\\discretization)}
      \end{tabular}
      \caption{Comparison of three models for Kohn-Sham DFT}\label{Table: Differences}
    \end{table}

    \begin{figure}
    \centering
      \includegraphics[width = 0.49\textwidth]{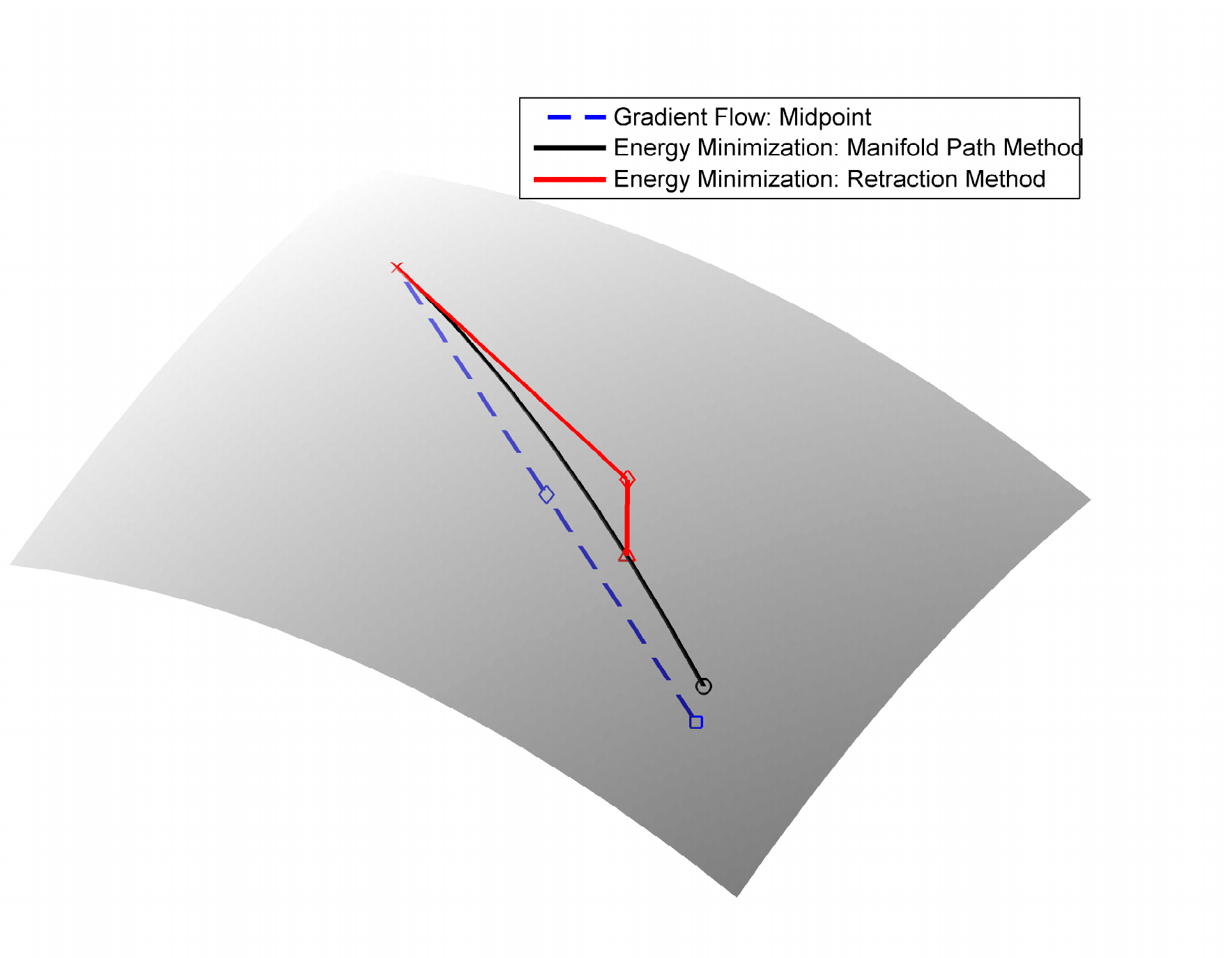}
      \includegraphics[width = 0.49\textwidth]{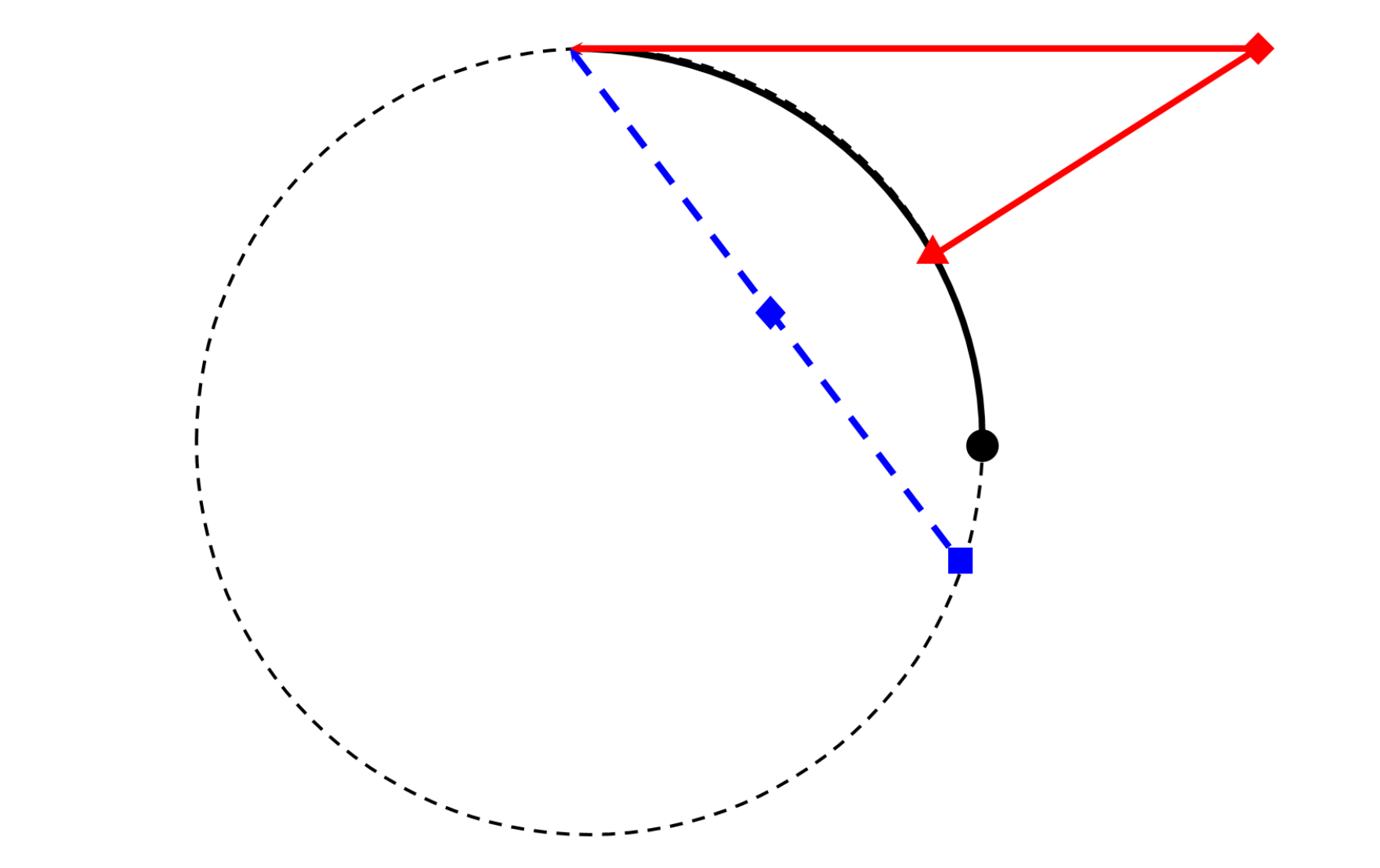}
      \caption{Comparison among gradient flow scheme, manifold path method and retraction method: blue dashed line with square symbol endpoint -- midpoint scheme of gradient flow model, black solid line with circle symbol endpoint -- manifold path method, red solid line with triangle symbol endpoint-- retraction method. Diamond symbol -- auxiliary point of each method.}\label{Fig: Cmp}
    \end{figure}

    We observe that there are some existing works on the gradient flow methods of eigenvalue problems. We refer to \cite{Oja1, Dyn1, Oja} and references cited therein for linear eigenvalue problems and \cite{Dyn3} for the ground state of Bose-Einstein condensate (which requires the smallest eigenvalue and its associated eigenfunction only). We point out that our gradient flow based model is different from the gradient flow model proposed in \cite{Opt7} for the Kohn-Sham DFT, in which the numerical scheme is either the retraction approach or the manifold path approach.

    We organize the rest of the paper as follows. In section 2, we introduce some necessary notation and the Kohn-Sham DFT models. Then we come up with our gradient flow based model and prove its local convergence and convergence rate of the asymptotic behaviours in section 3. In section 4, we propose a midpoint scheme to realize temporal discretization of the gradient flow based model and investigate the relevant properties including preserving orthogonality automatically, updating inside the manifold as well as the local convergence. Based on the midpoint scheme, in section 5, we design and analyze an orthogonality preserving iteration scheme for solving the discretized Kohn-Sham energy. In section 6, we provide numerical experiments that support our theory. Finally, we present some concluding remarks.

    \section{Preliminaries}
    In this section, we introduce some basic notation and the Kohn-Sham models.
    \subsection{Basic notation}
    We apply the standard $\textup{L}^2$-inner product $(\cdot, \cdot)_{\textup{L}^2(\mathbb{R}^3)}$, which is defined as
    \begin{equation}
      (u, v)_{\textup{L}^2(\mathbb{R}^3)} = \displaystyle\int_{\mathbb{R}^3}u(x)v(x)\textup{d}x,
    \end{equation}
   denote $\textup{L}^2$-norm $\Vert \cdot\Vert_{\textup{L}^2(\mathbb{R}^3)}$ by $\Vert u\Vert_{\textup{L}^2(\mathbb{R}^3)} = (u, u)_{\textup{L}^2(\mathbb{R}^3)}^{\frac12}$, and $\textup{L}^1$-norm $\Vert \cdot\Vert_{\textup{L}^1(\mathbb{R}^3)}$ by
    \begin{equation}
      \Vert u\Vert_{\textup{L}^1(\mathbb{R}^3)} =  \displaystyle\int_{\mathbb{R}^3}|u(x)|\textup{d}x.
    \end{equation}
    We define $\textup{H}^1$-norm $\Vert \cdot\Vert_{\textup{H}^1(\mathbb{R}^3)}$ as
    $$\Vert u\Vert^2_{\textup{H}^1(\mathbb{R}^3)} = \|u\|^2_{\textup{L}^2(\mathbb{R}^3)}+\|\nabla u\|^2_{\textup{L}^2(\mathbb{R}^3)}$$
    and use Sobolev space $\textup{H}^1(\mathbb{R}^3)$
    \begin{equation*}
      \textup{H}^1(\mathbb{R}^3) = \{u\in\textup{L}^2(\mathbb{R}^3):\Vert u \Vert_{\textup{H}^1(\mathbb{R}^3)} < +\infty\}.
    \end{equation*}

    Let ~$\Psi = (\psi_1, \psi_2, \ldots, \psi_N)\in\big(\textup{H}^1(\mathbb{R}^3)\big)^N$~ and ~$\Phi = (\varphi_1, \varphi_2, \ldots, \varphi_N)$ $\in \big(\textup{H}^1(\mathbb{R}^3)\big)^N$~. Define product matrix
    \begin{equation*}
      \Psi\odot\Phi = (\psi_i\varphi_j)_{i,j = 1}^N\in \big(\textup{L}^1(\mathbb{R}^3)\big)^{N\times N}
    \end{equation*}
    and inner product matrix
    \begin{equation*}
       \langle\Psi^\top\Phi\rangle = \big( (\psi_i, \varphi_j)_{\textup{L}^2(\mathbb{R}^3)} \big)_{i,j = 1}^N\in\mathbb{R}^{N\times N}.
    \end{equation*}
    For $\mathcal{F} = (\mathcal{F}_1, \mathcal{F}_2, \ldots, \mathcal{F}_N)\in \big( (\textup{H}^{1}(\mathbb{R}^3))^N \big)' = \big( \textup{H}^{-1}(\mathbb{R}^3) \big)^N$, we set
    \begin{equation}
      \langle \mathcal{F}, \Psi\rangle = \big( \langle \mathcal{F}_i, \psi_j\rangle \big)_{i,j = 1}^N\in\mathbb{R}^{N\times N}.
    \end{equation}
    %Furthermore, we define the dual pairing
    %\begin{equation}
     % \langle\langle \mathcal{F}, \Psi\rangle\rangle = \textup{tr} \langle \mathcal{F}, \Psi\rangle,
    %\end{equation}
    %where $\textup{tr}\langle \mathcal{F}, \Psi\rangle$ is the trace of the matrix $\langle \mathcal{F}, \Psi\rangle$.
We then introduce the Stiefel manifold defined as
    \begin{equation*}
      \mathcal{M}^N = \big\{U \in \big(\textup{H}^1(\mathbb{R}^3) \big)^N:\langle U^\top U\rangle = I_N \big\}.
    \end{equation*}
    For $U \in \big(\textup{H}^1(\mathbb{R}^3) \big)^N$ and any matrix $P \in \mathbb{R}^{N\times N}$, we denote
    \begin{equation*}
      UP = \Big( \sum\limits_{j = 1}^N p_{j1}u_j, \sum\limits_{j = 1}^N p_{j2}u_j, \ldots, \sum\limits_{j = 1}^N p_{jN}u_j \Big).
    \end{equation*}
    We see that
    \begin{equation*}
        U\in\mathcal{M}^N \Leftrightarrow UP\in\mathcal{M}^N, \quad \forall P \in\mathcal{O}^{N},
    \end{equation*}
    where
    \begin{equation*}
      \mathcal{O}^{N}=\{P\in\mathbb{R}^{N\times N}:P^\top P = I_N\}.
    \end{equation*}

    We define an equivalent relation ``$\sim$'' on $\mathcal{M}^N$ as
    \begin{equation*}
      U \sim \hat U\Leftrightarrow \exists P\in \mathcal{O}^{N},~\hat U = UP,
    \end{equation*}
    and get a Grassmann manifold, which is a quotient of $\mathcal{M}^N$
    \begin{equation*}
      \mathcal{G}^N = \mathcal{M}^N\slash\sim.
    \end{equation*}
    We introduce an equivalent class of $U\in \mathcal{M}^N$ by
    \begin{equation*}
      [U] = \{UP: P\in\mathcal{O}^{N}\},
    \end{equation*}
    an inner product as
    \begin{equation*}
      (U, \hat U) = \textup{tr}\big(\langle U^\top  \hat U \rangle \big)
    \end{equation*}
    together with an associated norm
    \begin{equation*}
      ||| U ||| = (U, U)^{1\slash2}
    \end{equation*}
    on $\big(\textup{H}^1(\mathbb{R}^3) \big)^N$.

    Give a finite-dimensional space $V_{N_g}\subset \textup{H}^1(\mathbb{R}^3)$ spanned by $\phi_1, \phi_2, \ldots, \phi_{N_g}$. We denote $\Phi = (\phi_1, \phi_2, \ldots, \phi_{N_g})$. We see that for any $U\in (V_{N_g})^N$, there exists $C\in\mathbb{R}^{N_g\times N}$ such that
    \begin{equation}\label{Equ: spacialDiscrete}
      U = \Phi C = \Big(\sum\limits_{j = 1}^{N_g} c_{j1}\phi_j, \sum\limits_{j = 1}^{N_g} c_{j2}\phi_j, \ldots, \sum\limits_{j = 1}^{N_g} c_{jN}\phi_j\Big).
    \end{equation}
    We define a closed $\delta$-neighborhood of $U$ by
    \begin{equation*}
      B(U, \delta) = \{\hat U\in (V_{N_g})^N :\textup{dist}(U, \hat U)\leqslant \delta\}
    \end{equation*}
    where
    \begin{equation*}
      \textup{dist}(U, \hat U) = ||| U - \hat U |||,
    \end{equation*}
    and for $U\in(V_{N_g})^N\bigcap\mathcal{M}^N$ introduce a closed $\delta$-neighborhood of $[U]$ on $\mathcal{G}^N$ by
    \begin{equation*}
      B([U], \delta) = \{[\hat U]\in \mathcal{G}^N: \hat U \in (V_{N_g})^N\bigcap\mathcal{M}^N, \textup{dist}([U],[\hat U])\leqslant \delta\},
    \end{equation*}
    where
    \begin{equation*}
      \textup{dist}([U],[\hat U]) = \inf\limits_{P\in\mathcal{O}^{N}} ||| U - \hat U P |||.
    \end{equation*}

    For simplicity, we use notation
    \begin{equation}
      \{U, W\} = UW^\top - WU^\top, \:\forall U, W \in(V_{N_g})^N
    \end{equation}
    where $UW^\top$ and $WU^\top$ denote operators on $(V_{N_g})^N$:
      \begin{equation*}
      \begin{split}
        (UW^\top) V &= U \langle W^\top V\rangle,\\
        (WU^\top) V &= W \langle U^\top V\rangle,
      \end{split}
      \end{equation*}
    for any $V\in(V_{N_g})^N$.

    Obviously
    \begin{equation}
      \{U, W\} + \{W, U\} = 0, \:\forall U, W \in (V_{N_g})^N.
    \end{equation}
    Namely $\{U, W\}$ is skew-symmetric.

    \subsection{Kohn-Sham models}
    The energy based Kohn-Sham DFT model for a system of $N$ electron orbitals with external potential contributed by M nuclei of charges is the following constrained optimization problem on the Stiefel manifold
    \begin{equation}\label{Opt: Kohn-Sham Energy}
      \begin{array}{cc}
        \inf\limits_{U\in (\textup{H}^1(\mathbb{R}^3))^N} &E(U)\\
        \textup{s.t.} &~U\in\mathcal{M}^N
      \end{array}
    \end{equation}
    where $E(U)$ is the Kohn-Sham energy
    \begin{equation}\label{Equ: Kohn-Sham Energy}
      \begin{split}
        E(U) =& \frac12\sum\limits_{i = 1}^N f_i \int_{\mathbb{R}^3}|\nabla u_i(r)|^2 \textup{d}r + \int_{\mathbb{R}^3}V_{ext}(r)\rho_{_U}(r)\textup{d}r\\
            &+ \frac12\int_{\mathbb{R}^3} \int_{\mathbb{R}^3} \frac{\rho_{_U}(r)\rho_{_U}(r')}{|r - r'|}\textup{d}r \textup{d}r' + E_{xc}\big(\rho_{_U}\big).
      \end{split}
    \end{equation}
    Here $u_i\in\textup{H}^1(\mathbb{R}^3)(i = 1, 2, \ldots, N)$ are Kohn-Sham orbitals,
    \begin{equation}
      \rho_{_U}(r) = \sum\limits_{i = 1}^N f_i |u_i(r)|^2 = \textup{tr}(U\odot U F)
    \end{equation}
    is the associated electron density with $f_i$ being the occupation number of the $i$-th orbital and $F = \textup{diag}(f_1, f_2, \ldots, f_N)$. $V_{ext}(r)$ is the external potential generated by the nuclei: for full potential calculations,
    \[
    V_{ext}(r) = -\sum\limits_{I = 1}^{M}\frac{Z_I}{|r - R_I|},\]
    $Z_I$ and $R_I$ are the nuclei charge and position of the $I$-th nuclei respectively; while for pseudo potential
    approximations, the formula for the energy is still (\ref{Equ: Kohn-Sham Energy}) (see, e.g., \cite{dai2017conjugate}).
   % but $V_{ext}(r)u_i(r)$ is replaced by
%    \begin{equation*}
%        \sum \limits_{I=1}^M(V^I_{loc}u_i)(r)+(V^I_{nloc}u_i)(r),
%    \end{equation*}
%    where
%    $(V^I_{loc}u_i)(r)$ is the local part and $(V_{nloc}u_i)(r)$ is the nonlocal part, which usually have the following form
%    \begin{equation*}
%        (V^I_{nloc} u_i)(r)= \sum_l\int_{\mathbb{R}^3}\xi_l^{I}(r') u_i(r') dr'\xi_l^{I}(r),
%    \end{equation*}
%    with $\xi_l^{I}\in L^2(\mathbb{R}^3)$\cite{Ele_Str2}.
The fourth term in \eqref{Equ: Kohn-Sham Energy} is the exchange-correlation energy, to which some approximations, such as  LDA(Local Density Approximation), GGA(General Gradient Approximation) and so on\cite{Ele_Str2, LDA}, should be applied. We assume that $E(U)$ is bounded from below with orthogonality constraint of $U$, which is of physics. For simplicity, we consider the case of $F = 2I_N$.

    We see that for any $ U\in\mathcal{M}^N$ and all $P\in\mathcal{O}^N$, there hold
    \begin{equation*}
      \rho_{_{UP}} = \textup{tr}((UP)\odot UP F) = 2\textup{tr}(U\odot UP P^\top) = 2\textup{tr}(U\odot U) = \rho_{_U}
    \end{equation*}
    and
    \begin{equation}
    E(UP) = E(U).
    \end{equation}
Instead we consider an optimization problem on $\mathcal{G}^N$
    \begin{equation}\label{Equ: EnergyGrassmann}
      \begin{array}{cc}
        \inf\limits_{U\in (\textup{H}^1(\mathbb{R}^3))^N} &E(U)\\
        \textup{s.t.} &~[U]\in\mathcal{G}^N
      \end{array}
    \end{equation}
and define  level set
    \begin{equation*}
      \mathcal{L}_{E} = \{[U]\in \mathcal{G}^N : E(U)\leqslant E\}.
    \end{equation*}
 %   where $E(U)$ is defined in equation \eqref{Equ: Kohn-Sham Energy}.

    To introduce the gradient on $\mathcal{G}^N$, we suppose
    %that $E_{xc}\big(\rho_{_U}\big)$ can be approximated as
    \begin{equation*}
      E_{xc}\big(\rho_{_U}\big) = \int_{\mathbb{R}^3} \varepsilon_{xc}\big(\rho_{_U}\big)(r)\rho_{_U}(r)\textup{d}r
    \end{equation*}
and assume that the ex\-change-corre\-lation energy is differentiable and the exchange-correlation potential
    \begin{equation*}
      v_{xc}(\rho) = \frac{\delta\big(\rho \varepsilon_{xc}(\rho) \big)}{\delta\rho}.
    \end{equation*}
    We may write the gradient of $E(U)$ as
    \begin{equation*}
      \nabla E(U) = (E_{u_1}, E_{u_2}, \ldots, E_{u_N}) \in \big(\textup{H}^{-1}(\mathbb{R}^3)\big)^N,
    \end{equation*}
    where $E_{u_i}\in\textup{H}^{-1}(\mathbb{R}^3)$ is defined by
    \begin{equation}\label{Equ: Derivative}
    \begin{split}
      \langle E_{u_i}, v\rangle =& 4\bigg( \frac12(\nabla u_i, \nabla v)_{\textup{L}_2} + (V_{ext}~u_i, v)_{\textup{L}_2}\\
       &+ \Big(\int_{\mathbb{R}^3}\frac{\rho_{_U}(r')}{|r - r'|}\textup{d}r~u_i, v\Big)_{\textup{L}_2} + \Big(v_{xc}\big(\rho_{_U}\big)~u_i, v\Big)_{\textup{L}_2} \bigg), \forall v\in\textup{H}^1(\mathbb{R}^3).
    \end{split}
    \end{equation}
    Obviously
    \begin{equation}\label{Equ: GradientSym}
      \big\langle\nabla E(U), U\big\rangle = \big\langle\nabla E(U), U\big\rangle^\top,\forall U\in\big(\textup{H}^1(\mathbb{R}^3)\big)^N.
    \end{equation}
  %  and from now on we denote
%    \begin{equation}
%      \big\langle\nabla E(U)^\top U\big\rangle = \big\langle\nabla E(U), U\big\rangle,~\forall U\in\big(\textup{H}^1(\mathbb{R}^3)\big)^N.
%    \end{equation}

    We see from \cite{Grassmann1} that the gradient on Grassmann manifold $\mathcal{G}^N$of $E(U)$ at $[U]$ is
% $\nabla_{G} E(U)\in \big(\textup{H}^{-1}(\mathbb{R}^3)\big)^N$:
%    \begin{equation}\label{Equ: GrassmannGradientPrm}
%          \nabla_{G}E(U) = \nabla E(U) - U \big\langle U^\top\nabla E(U)\big\rangle, ~\forall  U\in \mathcal{M}^N,
%    \end{equation}
%    or equivalently
    \begin{equation}\label{Equ: GrassmannGradientPrm}
          \nabla_{G}E(U) = \nabla E(U) - U\big\langle\nabla E(U), U\big\rangle^\top, ~\forall U\in \mathcal{M}^N.
    \end{equation}
   % which is an operator from $\big(\textup{H}^{1}(\mathbb{R}^3)\big)^N$ to $\big(\textup{H}^{-1}(\mathbb{R}^3)\big)^N$ defined as
%    \begin{equation}
%      \langle \nabla_{G}E(U)^\top \Psi \rangle = \langle \nabla E(U)^\top \Psi \rangle - \langle\nabla E(U)^\top U \rangle \langle U^\top \Psi\rangle, \quad\forall \Psi\in \big(\textup{H}^{1}(\mathbb{R}^3)\big)^N.
%    \end{equation}
%    \begin{equation}\label{Equ: GrassmannGradientPre}
%      \nabla_{G}E(U) = \nabla E(U) - U \big\langle \nabla E(U), U \big\rangle^\top, \nabla E(U) \in \big(\textup{H}^{-1}(\mathbb{R}^3)\big)^N.
%    \end{equation}
    To propose a gradient flow based model preserving orthogonality, we need to extend the domain of $\nabla_{G}E(U)$ from $\mathcal{M}^N$ to $\big(\textup{H}^{1}(\mathbb{R}^3)\big)^N$. We then define extended gradient $\nabla_{G}E(U):\big(\textup{H}^{1}(\mathbb{R}^3)\big)^N \longrightarrow \big(\textup{H}^{-1}(\mathbb{R}^3)\big)^N$ as follows
    \begin{equation}\label{Equ:GrassmannGradient}
      \nabla_{G}E(U) = \nabla E(U)\langle U^\top U\rangle - U \big\langle\nabla E(U), U\big\rangle^\top, ~\forall U\in \big(\textup{H}^{1}(\mathbb{R}^3)\big)^N.
    \end{equation}
    Note that \eqref{Equ:GrassmannGradient} is consistent with \eqref{Equ: GrassmannGradientPrm} for $[U]\in\mathcal{G}^N$ since $\langle U^\top U\rangle = I_N$.

   We see from \cite{dai2017conjugate,Grassmann1} that the tangent space on $\mathcal{G}^N$ is
    \begin{equation}
      \mathcal{T}_{[U]}\mathcal{G}^N = \big\{ W\in \big(\textup{H}^1(\mathbb{R}^3)\big)^N:\langle W^\top U \rangle= 0 \big\}
    \end{equation}
    and the Hessian of $E(U)$ on $\mathcal{G}^N$  is
    \begin{equation}
      \textup{Hess}_{G}E(U)[V, W] = \textup{tr}\big(\langle V^\top \nabla^2E(U)W\rangle\big) - \textup{tr}\big(\langle V^\top W\rangle \langle U^\top \nabla E(U)\rangle\big),\:\forall V, W \in \mathcal{T}_{[U]}\mathcal{G}^N.
    \end{equation}

%    Additionally, we define the level set of discretized energy by
%    \begin{equation*}
%      \mathcal{L}_{E_0} = \{[U]\in \mathcal{G}^N : E(U)\leqslant E_0\}.
%    \end{equation*}

If $U\in (V_{N_g})^N$, then we may view $\nabla E(U)\in (V_{N_g})^N$ in the sense of isomorphism and
\begin{equation}
\big\langle\nabla E(U), V\big\rangle = \big\langle\big(\nabla E(U)\big)^\top V\big\rangle,~\forall V\in(V_{N_g})^N.
\end{equation}
As a result, $\nabla_{G}E(U)\in (V_{N_g})^N$ and we may write
    \begin{equation}\label{Equ: GrassmannGradientDis}
%    \begin{split}
      \nabla_{G}E(U)
%&= \nabla E(U)\langle U^\top U\rangle - U \big\langle U^\top \nabla E(U)\big\rangle\\
                     %= \big(\{\nabla E(U), U\}\big)U
                     = \mathcal{A}_{U} U, ~\forall U\in(V_{N_g})^N,
 %   \end{split}
    \end{equation}
    where
    \begin{equation*}
      \mathcal{A}_{U} = \{\nabla E(U), U\}.
    \end{equation*}

    \section{Gradient flow based model}\label{Sec: GrdFlw}
    In this section, we propose and analyze a gradient flow based model.
    \subsection{The model}
    Different from the Kohn-Sham equation and the Kohn-Sham energy minimization model, we propose a gradient flow based model of Kohn-Sham DFT as follows:
    \begin{equation}\label{Equ: GradientFlowOp}
    \left\{
    \begin{array}{l}
      \displaystyle\frac{\textup{d}U}{\textup{d}t} = -\nabla_G E(U), \quad 0< t <\infty\\
      U(0) = U_{0},
    \end{array}
    \right.
    \end{equation}
    where $U(t)\in (V_{N_g})^N$ and $U_{0}\in\mathcal{M}^N$. We see that \eqref{Equ: GradientFlowOp} is different from the standard gradient flow model presented in \cite{Opt7},  which applies the $\nabla_G E(U)$ in \eqref{Equ: GrassmannGradientPrm} rather than \eqref{Equ: GrassmannGradientDis}. We point out that whether the solution of \eqref{Equ: GrassmannGradientPrm} keeps on the Stiefel manifold is unclear. However, we see from Proposition \ref{Prop: GradientFlowOp} that our new $\nabla_G E(U)$ defined by \eqref{Equ: GrassmannGradientDis} guarantees that the solution keeps on the Stiefel manifold. Namely, \eqref{Equ: GradientFlowOp} is an orthogonality preserving model whenever the initial is orthogonal.

    \begin{lemma}\label{Lemma: ZeroTrace}
      If $A, B\in \mathbb{R}^{N\times N}$ and
      \begin{equation*}
        A^\top = A, \quad B^\top = -B,
      \end{equation*}
      then
      \begin{equation*}
        \textup{tr}(AB) = 0.
      \end{equation*}
    \end{lemma}

    \begin{proof}
      We see that
      \begin{equation*}
        \textup{tr}(AB) = \textup{tr}(AB)^\top = \textup{tr}(B^\top A^\top) = -\textup{tr}(BA) = -\textup{tr}(AB),
      \end{equation*}
      which indicates
      \begin{equation*}
        \textup{tr}(AB) = 0.
      \end{equation*}
    \end{proof}

    \begin{prop}\label{Prop: GradientFlowOp}
      The solution of \eqref{Equ: GradientFlowOp} satisfies ${U(t)}\in\mathcal{M}^N$. Moreover, there holds
        \begin{equation}
        \frac{\textup{d}E\big( U(t) \big)}{\textup{d}t} = -\Big|\Big|\Big| \nabla_{G} E\big(U(t)\big) \Big|\Big|\Big|^2\leqslant 0, \quad 0 < t < \infty.
        \end{equation}
    \end{prop}
    \begin{proof}
    A direct calculation shows that
      \begin{equation*}
      \begin{split}
        \frac{\textup{d}}{\textup{d}t}\big\langle U(t)^\top U(t)\big\rangle &= \bigg\langle\Big(\frac{\textup{d}}{\textup{d}t}U(t)\Big)^\top U(t)\bigg\rangle + \Big\langle U(t)^\top \frac{\textup{d}}{\textup{d}t}U(t)\Big\rangle\\
        &= \Big(\big\langle U(t)^\top \mathcal{A}_{U(t)} U(t)\big\rangle \Big) - \Big(\big\langle U(t)^\top \mathcal{A}_{U(t)} U(t)\big\rangle \Big) = 0,
      \end{split}
      \end{equation*}
      which indicates
      \begin{equation*}
        \big\langle{U(t)}^\top U(t)\big\rangle = I_N
      \end{equation*}
      due to $\langle{U_0}^\top U_0\rangle = I_N$. Consequently, we see from Lemma \ref{Lemma: ZeroTrace} that
      \begin{equation}
      \begin{split}
            &\Big|\Big|\Big| \nabla_{G} E\big(U(t)\big) \Big|\Big|\Big|^2 - \textup{tr}\Big\langle\nabla E\big(U(t)\big)^\top \nabla_{G} E\big(U(t)\big)\Big\rangle\\
        =&\textup{tr}\bigg(\Big\langle \nabla E\big(U(t)\big)^\top U(t)\Big\rangle \Big\langle U(t)^\top \nabla_{G} E\big(U(t)\big)\Big\rangle\bigg) = 0.
      \end{split}
      \end{equation}
      As a result,
      \begin{equation}
        \frac{\textup{d}E\big( U(t) \big)}{\textup{d}t} = \frac{\delta E}{\delta U}\cdot \frac{\textup{d}U}{\textup{d}t}= -\textup{tr}\Big\langle\nabla E\big(U(t)\big)^\top \nabla_{G} E\big(U(t)\big)\Big\rangle = -\Big|\Big|\Big| \nabla_{G} E\big(U(t)\big) \Big|\Big|\Big|^2\leqslant 0.
      \end{equation}
    \end{proof}

    \subsection{Critical points}

    We denote Lagrange function of \eqref{Equ: EnergyGrassmann}
    \begin{equation}
      \mathcal{L}(U, \Lambda) = E(U) - \frac12\big(\langle U^\top U\rangle - I_N \big)\Lambda
    \end{equation}
    for $U \in (V_{N_g})^N$ and $\Lambda\in\mathbb{R}^{N\times N}$, then the corresponding first-order necessary condition is as follows
    \begin{eqnarray}
	\nabla_U \mathcal{L}(U, \Lambda)			&\equiv& \nabla E(U) - U\Lambda = 0,\label{Equ: EigenGrad}\\
	\nabla_{\Lambda} \mathcal{L}(U, \Lambda)	&\equiv& \frac12\big(I_N - \langle U^\top U\rangle \big) = 0.\label{Equ: EigenOrtho}
    \end{eqnarray}

    We call $[U]$ a critical point of \eqref{Equ: EnergyGrassmann} if
    \[\nabla_{G} E(U) = 0.\]
    Obviously, for such a critical point, we have
    \[
    \nabla E(U) = U \langle U^\top \nabla E(U)\rangle,
    \]
    which suggests
    \begin{eqnarray*}
	\nabla_U \mathcal{L}\big(U, \langle U^\top \nabla E(U)\rangle\big)			= 0,\\
	\nabla_{\Lambda} \mathcal{L}\big(U, \langle U^\top \nabla E(U)\rangle\big)	= 0.
    \end{eqnarray*}
    Thus we see that such a critical point may be a local minimizer.

    As $t\to\infty$, we know that energy $E\big( U(t) \big)$ decreases monotonically, thus $\lim\limits_{t\rightarrow\infty}$ $E\big( U(t) \big)$ exists provided that $E\big( U(t) \big)$ is bounded from below. The following statement tells us the asymptotical behavior of the extended gradient flow(c.f. \cite{Lyapunov}).
    \begin{theorem}\label{Theo: Inf ConvergeOp}
      If $U(t)$ is a solution of \eqref{Equ: GradientFlowOp}, then
      \begin{equation}
      \begin{split}
        &\liminf\limits_{t\rightarrow\infty} \Big|\Big|\Big| \nabla_{G} E\big(U(t)\big) \Big|\Big|\Big| = 0.
      \end{split}
      \end{equation}
    \end{theorem}
    \begin{proof}
     We see from Proposition \ref{Prop: GradientFlowOp} that
      \begin{equation*}
      \begin{split}
        &\int_{0}^{+\infty} \Big|\Big|\Big| \nabla_{G} E\big(U(t)\big) \Big|\Big|\Big|^2 \textup{d} t = -\int_{0}^{+\infty} \frac{\textup{d} E(U)}{\textup{d}t} \textup{d}t\\
        =& E\big(U(0)\big) - \lim\limits_{t\rightarrow \infty} E\big(U(t)\big) < +\infty.
      \end{split}
      \end{equation*}
      Since $\Big|\Big|\Big| \nabla_{G} E\big(U(t)\big) \Big|\Big|\Big|^2$ is nonnegative function, we have
      \begin{equation*}
      \liminf\limits_{t\rightarrow\infty} \Big|\Big|\Big| \nabla_{G} E\big(U(t)\big) \Big|\Big|\Big| = 0.
      \end{equation*}
    \end{proof}

    Suppose that the local minimizer $[U^*]$ is the unique critical point of \eqref{Equ: EnergyGrassmann} in $B([U^*], \delta_1)$. For a fixed constant $\delta_2 \in(0, \delta_1]$, we define
      \begin{equation*}
        E_0 = \min \{ E([\tilde U])~|~[\tilde U] \in \overline{B([U^*], \delta_1)\backslash B([U^*], \delta_2)}\}.
      \end{equation*}
    Here and hereafter, we assume that as an operator from $\big(V_{N_g}\big)^N$ to $\big(V_{N_g}\big)^N$, $\nabla E$ is continuous in $B(U^*, \delta_1)$.

    \begin{theorem}\label{Theo: ContConv}
    If the initial value satisfies \[E(U_0)\leqslant \frac{E_0 + E(U^*)}{2}\equiv E_1,\]
      then
      \begin{eqnarray*}
      \lim\limits_{t\rightarrow\infty} \Big|\Big|\Big| \nabla_{G} E\big(U(t)\big) \Big|\Big|\Big| = 0,\\
      \lim\limits_{t\rightarrow\infty} E\big(U(t)\big) = E(U^*),\\
      \lim\limits_{t\rightarrow\infty}\textup{dist}([U(t)], [U^*]) = 0.
      \end{eqnarray*}
    \end{theorem}

    \begin{proof}
    We obtain from Theorem \ref{Theo: Inf ConvergeOp} that there exists a sequence $\{\tau_k\}_{i = 1}^{\infty}$ so that $\lim\limits_{k\rightarrow\infty}\tau_k = +\infty$ and $\lim\limits_{k\rightarrow\infty} \big|\big|\big| \nabla_{G} E\big(U(\tau_k)\big) \big|\big|\big| = 0$. The uniqueness of critical point in $B([U^*], \delta_1)$ implies $E_0> E([U^*])$. Due to $E([U(t)]) \leqslant E_1, \forall t\geqslant 0, $ we have
        \begin{equation*}
          [U(\tau_k)]\in B([U^*], \delta_2)\bigcap \mathcal{L}_{E_1},
        \end{equation*}
    where $\mathcal{L}_{E_1}=\{[U]\in\mathcal{G}^N: E(U)\le E_1\}$ is the level set.
    Since set
    \begin{equation*}
    \mathcal{S}= \{\tilde U\in (V_{N_g})^N:[\tilde U]\in B([U^*], \delta)\bigcap \mathcal{L}_{E_1}\}
    \end{equation*}
    is compact, there exist a subsequence $\{U(\tau_{k_l})\}$ and $\hat U\in\mathcal{S}$ that $\lim\limits_{l\rightarrow\infty}U(\tau_{k_l}) = \hat U$. Since $\nabla E$ is continuous, then $\nabla_{G} E$ is also continuous, so $\nabla_{G} E(\hat U) = 0$. By the uniqueness of critical point in $B([U^*], \delta_1)$ again, we get $[\hat U] = [U^*]$ and
    \begin{equation*}
      \lim\limits_{t\rightarrow\infty} E\big(U(t)\big) = \lim\limits_{l\rightarrow\infty} E\big(U(\tau_{k_l})\big) = E(U^*).
    \end{equation*}
    We claim that $\lim\limits_{t\rightarrow\infty} \textup{dist}([U(t)], [U^*]) = 0$. Otherwise, there exists a subsequence $\{U(\tau_p)\}$ that for some fixed $\hat\delta > 0$, $\textup{dist}([U(\tau_p)], [U^*]) \geqslant \hat\delta$. Since $\mathcal{S}$ is compact, there exist a subsequence $\{U(\tau_{p_q})\}$ and $\bar U\in\mathcal{S}$ that $\lim\limits_{q\rightarrow\infty}U(\tau_{p_q}) = \bar U$. Therefore
    \[E(\bar U) = \lim\limits_{q\rightarrow\infty} E\big(U(\tau_{p_q})\big) = E([U^*]),\]
    and $[\bar U] = [U^*]$, which contradicts the assumption $\textup{dist}([U(\tau_p)], [U^*]) \geqslant \hat\delta$.

    Clearly, there exists $P(t)\in\mathcal{O}^N$ that
    \begin{equation}
      |||U(t)P(t) - U^*|||= \textup{dist}([U(t)], [U^*]),
    \end{equation}
    then
    \begin{equation*}
       \lim\limits_{t\rightarrow\infty} \Big|\Big|\Big| \nabla_{G} E\big(U(t)\big) \Big|\Big|\Big| =  \lim\limits_{t\rightarrow\infty} \Big|\Big|\Big| \nabla_{G} E\big(U(t)P(t)\big) \Big|\Big|\Big| = \Big|\Big|\Big| \nabla_{G} E\big(U^*\big) \Big|\Big|\Big| = 0.
    \end{equation*}
    \end{proof}

    Indeed, we may have some convergence rate.
    \begin{theorem}
    If $E(U_0)\leqslant E_1$ and
      \begin{equation}\label{Equ: Hess}
        \textup{Hess}_G E(U)[D, D] \geqslant \sigma ||| D |||^2 \quad \forall [U]\in B([U^*], \delta_3),\: \forall D \in \mathcal{T}_{[U]}\mathcal{G}^N\bigcap(V_{N_g})^N
      \end{equation}
      for some $\delta_3\in(0,\delta_1]$ and $\sigma > 0$,
      then there exists $\hat{T} > 0$ such that
      \begin{eqnarray*}
        &\Big|\Big|\Big|\nabla_{G} E\big(U(t)\big)\Big|\Big|\Big|\leqslant e^{-\sigma (t - \hat T)},\\
        &E\big(U(t)\big) - E(U^*) \leqslant \frac{1}{2\sigma} e^{-2\sigma (t - \hat T)}
      \end{eqnarray*}
      for all $t \geqslant {\hat T}$.
    \end{theorem}

    \begin{proof}
    We see that
    \begin{equation*}
    \begin{split}
      &\frac12\frac{\textup{d}}{\textup{d}t}\Big|\Big|\Big|\nabla_{G} E\big(U(t)\big)\Big|\Big|\Big|^2 = \textup{tr} \bigg( \Big\langle\nabla_{G} E\big(U(t)\big)^\top \frac{\textup{d}}{\textup{d}t} \nabla_{G} E\big(U(t)\big) \Big\rangle\bigg)\\
      =&\textup{tr} \bigg( \Big\langle\nabla_{G} E\big(U(t)\big)^\top \nabla^2 E\big(U(t)\big) \frac{\textup{d}}{\textup{d}t} U(t) \Big\rangle \bigg) \\
      &- \textup{tr} \bigg( \Big\langle \nabla_{G} E\big(U(t)\big)^\top \frac{\textup{d}}{\textup{d}t}U(t)\Big\rangle \Big\langle U(t)^\top\nabla E\big(U(t)\big) \Big\rangle\bigg)\\
      &-\textup{tr} \bigg( \Big\langle \nabla_{G} E\big(U(t)\big)^\top U(t)\Big\rangle \frac{\textup{d}}{\textup{d}t}\bigg(\Big\langle U(t)^\top\nabla E\big(U(t)\big) \Big\rangle\bigg)\bigg),
    \end{split}
    \end{equation*}
    which together with Lemma \ref{Lemma: ZeroTrace} leads to
     \begin{equation*}
      \frac12\frac{\textup{d}}{\textup{d}t}\Big|\Big|\Big|\nabla_{G} E\big(U(t)\big)\Big|\Big|\Big|^2
      = -\textup{Hess}_G E\big(U(t)\big)\big[\nabla_{G} E\big(U(t)\big), \nabla_{G} E\big(U(t)\big)\big].
    \end{equation*}
    Note that Theorem \ref{Theo: ContConv} implies that there exists $\hat T > 0$ such that
    \begin{equation}
      U(t)\in B([U^*], \delta_3), \quad \forall t\geqslant \hat T.
    \end{equation}
    Hence, we obtain from (\ref{Equ: Hess}) that
    \begin{equation}
      \frac{\textup{d}}{\textup{d}t}\Big|\Big|\Big|\nabla_{G} E\big(U(t)\big)\Big|\Big|\Big|^2 \leqslant - 2\sigma \Big|\Big|\Big|\nabla_{G} E\big(U(t)\big)\Big|\Big|\Big|^2,\quad t\geqslant \hat T.
    \end{equation}
    Using Gr\"{o}nwall's inequality we arrive at
    \begin{equation}
      \Big|\Big|\Big|\nabla_{G} E\big(U(t)\big)\Big|\Big|\Big|^2 \leqslant e^{-2\sigma(t-\hat T)},\quad t\geqslant \hat T.
    \end{equation}
    Therefore, for all $t\ge \hat T$, there hold
    \begin{equation}
      \Big|\Big|\Big|\nabla_{G} E\big(U(t)\big)\Big|\Big|\Big|\leqslant e^{-\sigma (t - \hat T)}
    \end{equation}
    and
    \begin{equation}
      E\big(U(t)\big) - E(U^*) = \int_t^{+\infty}\Big|\Big|\Big|\nabla_{G} E\big(U(t)\big)\Big|\Big|\Big|^2 \textup{d}t\leqslant \frac{1}{2\sigma} e^{-2\sigma (t - \hat T)}.
    \end{equation}
    \end{proof}

    We understand that (\ref{Equ: Hess}) has been already applied in \cite{dai2017conjugate,Opt7}. We observe that  $\sigma$ in \eqref{Equ: Hess} is related to the gap between the $(N+1)$-th eigenvalue and the $N$-th eigenvalue of the Kohn-Sham equation.
    %In fact, $\sigma$ should be no more than the eigenvalue gap at $[U^*]$ by definition.
    %As a result, the convergence rate is determined by the gap.

    \section{Temporal discretization}\label{Sec: Exact}
    We may apply various temporal discretization approaches to solve \eqref{Equ: GradientFlowOp}. In this section, we propose and analyze a midpoint point scheme. Our analysis shows that the  midpoint point scheme is quite efficient and recommended.
    \subsection{A midpoint scheme}
    Let $\{t_n:n = 0, 1, 2 \cdots \}\subset[0, +\infty)$ be discrete points such that
    \begin{equation}
      0 = t_0 < t_1 < t_2 < \cdots < t_n < \cdots,
    \end{equation}
    and $\lim\limits_{n\rightarrow+\infty} t_n = +\infty$. Set
    \begin{equation}
      \Delta t_n = t_{n + 1} - t_n,
    \end{equation}
    and consider a midpoint scheme as follows
    \begin{equation}\label{Equ: MidpointOp}
      \frac{U_{n + 1} - U_{n}}{\Delta t_n} = -\nabla_G E(U_{n + 1\slash2}),
    \end{equation}
    where $U_{n + 1\slash2} = (U_{n + 1} + U_{n})\slash{2}$. Equivalently    \begin{equation}\label{Equ: Midpoint}
        \frac{U_{n + 1} - U_{n}}{\Delta t_n} = - \mathcal{A}_{U_{n + 1\slash2}}U_{n + 1\slash2}.
    \end{equation}
 %   In implementation, we apply \eqref{Equ: Midpoint} to solve \eqref{Equ: GradientFlowOp}.
 Our midpoint scheme is an implicit method and we will propose and analyze a practical scheme
    to solve \eqref{Equ: Midpoint} in the next section.

    First, we investigate the existence of the solution of \eqref{Equ: Midpoint} in a  neighborhood of $U^*$, which requires that $\nabla E(U)$ is Lipschitz continuous locally
    \begin{equation*}
      |||\nabla E(U_1) - \nabla E(U_2) ||| \leqslant L_0 ||| U_1 - U_2|||,~\forall U_1, U_2\in B(U^*, \delta_1),
    \end{equation*}
    which is true for LDA when $\rho > 0$. However, it is still open whether $\rho > 0$ \cite{Positivity}.
    \begin{lemma}\label{Lemma: Implicit Function TheoryOp}
       There exist such $\delta_a, \delta_b, \delta^* > 0$ and a unique function $g: B(U^*, \delta_a) \times [-\delta^*, \delta^*]\rightarrow B(U^*, \delta_b)$ which satisfies
      \begin{equation}\label{Equ: Implicit FunctionOp}
        \begin{split}
          g(U, s) - U = -s\nabla_G E\Big(\frac{g(U, s) + U}{2}\Big)
        \end{split}
      \end{equation}
      for some $\delta_a, \delta_b$ and $\delta^* >0$.
    \end{lemma}
    \begin{proof}
     We define $\mathcal{H}$ on $(V_{N_g})^N\times(V_{N_g})^N\times \mathbb{R}$ by
     \begin{equation}
       \mathcal{H}(X, Y, t) := Y - X + t\nabla_G E\Big(\frac{Y + X}{2}\Big).
     \end{equation}
     Obviously, $\mathcal{H}(U^*, U^*, 0) = 0$ and $\frac{\partial}{\partial Y}\mathcal{H}(X, Y, t)$ exists. Since
     \begin{equation}
       \frac{\partial}{\partial Y}\mathcal{H}(U^*, U^*, 0) = I,
     \end{equation}
     we see from implicit function theory that there exists a unique function $g: B(U^*, \delta_a) \times [-\delta^*, \delta^*]\rightarrow B(U^*, \delta_b)$ which satisfies $\mathcal{H}(U, g(U, s), s) = 0$ for some $\delta_a, \delta_b, \delta^* > 0$. Thus we complete the proof.
    \end{proof}

    Due to Lemma \ref{Lemma: Implicit Function TheoryOp}, we see that $U_{n + 1} = g(U_n, \Delta t_n)$ is the solution of \eqref{Equ: MidpointOp}. Then we arrive at the following Algorithm \ref{Alg: Midpoint} and refer to Theorem \ref{Theo: MainConvergeOp} for the  choice of $\delta_T$.
    \begin{algorithm}\label{Alg: Midpoint}
      \caption{A midpoint scheme}
      Given $\varepsilon > 0$, $\delta_T > 0$, initial data~$U_0\in (V_{N_g})^N\bigcap\mathcal{M}^N$, calculate gradient $\nabla_{G} E(U_0)$, let $n = 0$\;
      \While{$|||\nabla_{G} E(U_n)||| > \varepsilon$}{
      Set time step $\Delta t_n \leqslant \delta_T$\;
      Solve
      \[
      \frac{U_{n + 1} - U_{n}}{\Delta t_n} = -\nabla_G E\bigg(\frac{U_{n + 1} + U_{n}}{2}\bigg)
      \]
      to get $U_{n + 1}$\;
      Let $n = n + 1$, calculate gradient $\nabla_{G} E(U_n)$\;
      }
    \end{algorithm}

    We will see from Proposition \ref{Prop: OrthoOp} that the approximations produced by midpoint scheme \eqref{Equ: MidpointOp} are orthogonality preserving, which is significant in electronic structure calculations, for instance. The following lemmas are helpful in our analysis.

    \begin{lemma}
      $(I + s\mathcal{A}_{U})^{-1}$ exists for all $s\in\mathbb{R}$ and $U\in(V_{N_g})^N$.
    \end{lemma}

    \begin{proof}
      Since $\mathcal{A}_{U}$ is skew-symmetric, the corresponding eigenvalues are pure imaginary numbers. As a result, the eigenvalues of $\big(I + s\mathcal{A}_{U}\big)$ belongs to the set
      \begin{equation}
        \{1 + \mu_j\imath:\mu_j\in\mathbb{R}\}
      \end{equation}
    where $\imath$ is the imaginary unit that $\imath^2 = -1$, which implies $\big(I + s\mathcal{A}_{U}\big)$ is invertible.
    \end{proof}

    \begin{lemma}\label{Lemma: OrthoOp}
      If $U\in(V_{N_g})^N\bigcap\mathcal{M}^N$, then
      \begin{equation*}
        \hat U \equiv \big(2(I + s\mathcal{A}_{\tilde U})^{-1}U - U\big)\in (V_{N_g})^N\bigcap\mathcal{M}^N
      \end{equation*}
      for all $s\in\mathbb{R}$ and $\tilde U\in(V_{N_g})^N$.
    \end{lemma}
    \begin{proof}
    A simple calculation shows that
      \begin{equation}
        \begin{split}
          \hat U &= 2(I + s\mathcal{A}_{\tilde U})^{-1}U - U = (I + s\mathcal{A}_{\tilde U})^{-1}\big(2I - (I + s\mathcal{A}_{\tilde U})\big)U\\
                 &= (I + s\mathcal{A}_{\tilde U})^{-1}(I - s\mathcal{A}_{\tilde U})U.
        \end{split}
      \end{equation}
    We have
    \begin{equation}
    \begin{split}
      \langle \hat U^\top  \hat U \rangle &= \langle U^\top (I + s\mathcal{A}_{\tilde U}) (I - s\mathcal{A}_{\tilde U})^{-1} (I + s\mathcal{A}_{\tilde U})^{-1}(I - s\mathcal{A}_{\tilde U})U \rangle = \langle U^\top U \rangle = I_N
    \end{split}
    \end{equation}
    and complete the proof.
    \end{proof}

    \begin{proposition}\label{Prop: OrthoOp}
      If $U_n$ is obtained from Algorithm \ref{Alg: Midpoint}, then $U_n\in$$(V_{N_g})^N\bigcap\mathcal{M}^N$ for all $n\in\mathbb{N}$.
    \end{proposition}
    \begin{proof}
    We split \eqref{Equ: MidpointOp} into two equations
    \begin{equation}\label{Equ: Midpoint IntermediateOp}
      \begin{split}
      \frac{U_{n + 1\slash2} - U_{n}}{{\Delta t_n}\slash{2}} &= - \mathcal{A}_{U_{n + 1\slash2}}U_{n + 1\slash2},\\
      \frac{U_{n + 1} - U_{n + 1\slash2}}{{\Delta t_n}\slash{2}} &= - \mathcal{A}_{U_{n + 1\slash2}}U_{n + 1\slash2},
      \end{split}
    \end{equation}
    and obtain
    \begin{equation}\label{Equ: Midpoint Intermediate1Op}
    \begin{split}
      U_{n + 1\slash2} &= \Big(I + \frac{\Delta t_n}{2}\mathcal{A}_{U_{n + 1\slash2}}\Big)^{-1}U_n,\\
      U_{n + 1} &= 2\Big(I + \frac{\Delta t_n}{2}\mathcal{A}_{U_{n + 1\slash2}}\Big)^{-1}U_n - U_n.
    \end{split}
    \end{equation}
    Therefore, we arrive at the conclusion from Lemma \ref{Lemma: OrthoOp}.
    \end{proof}

    We see from \eqref{Equ: Midpoint IntermediateOp} that the midpoint scheme of gradient flow based method may be reviewed as a  mixed scheme of an implicit Euler method of a temporal step ${\Delta t_n}\slash{2}$ and an explicit Euler method of the temporal step ${\Delta t_n}\slash{2}$ provided an auxiliary point. We will see an crucial difference between our midpoint scheme of gradient flow based method and the retraction optimization method afterwards.

    \begin{lemma}\label{Lemma: Inner PropertyOp}
      If $U\in(V_{N_g})^N\bigcap\mathcal{M}^N$, then spectrum $\sigma\big(\langle {\bar U}^\top {\bar U} \rangle\big)$ of $\langle {\bar U}^\top {\bar U}\rangle$ satisfies
      \begin{equation}
        \sigma\big(\langle {\bar U}^\top {\bar U}\rangle\big)\subset [0, 1],
      \end{equation}
      where
      \begin{equation}
        \bar U\equiv (I + s\mathcal{A}_{\tilde U})^{-1}U
      \end{equation}
      for all $s\in\mathbb{R}$ and $\tilde U\in(V_{N_g})^N$.
    \end{lemma}
    \begin{proof}
    For any eigenvalue $\lambda_j \in \sigma\big(\langle {\bar U}^\top {\bar U} \rangle\big)$, we have
    \begin{equation}\label{Equ: Midpoint Pf1Op}
      \begin{split}
        &0 \leqslant \lambda_j \leqslant \Vert \langle {\bar U}^\top {\bar U} \rangle \Vert_2 = \Vert \langle U^\top \big(I - s^2(\mathcal{A}_{\tilde U})^2\big)^{-1} U\rangle \Vert_2\\
       \leqslant& \big\Vert \big(I - s^2(\mathcal{A}_{\tilde U})^2\big)^{-1}\big\Vert \Vert \langle U^\top U \rangle \Vert_2 = \big\Vert \big(I - s^2(\mathcal{A}_{\tilde U})^2\big)^{-1}\big\Vert.
      \end{split}
    \end{equation}
    Note that $\mathcal{A}_{\tilde U}$ is skew-symmetric, which implies its eigenvalues are pure imaginary numbers. We obtain
    \begin{equation}\label{Equ: Midpoint Pf2Op}
    \begin{split}
      &\Vert \big(I - s^2(\mathcal{A}_{\tilde U})^2\big)^{-1}\Vert\\
      =& \max\big\{(1 + s^2\mu_j^2)^{-1}:\mu_j {\imath} \in \sigma\big(\mathcal{A}_{\tilde U}\big), \mu_j \in \mathbb{R}\big\} \leqslant 1,
    \end{split}
    \end{equation}
    where $\imath$ is the imaginary unit satisfying $\imath^2 = -1$. This completes the proof.
    \end{proof}

    Combining \eqref{Equ: Midpoint Intermediate1Op} and Lemma \ref{Lemma: Inner PropertyOp}, we arrive at
    \begin{proposition}\label{Prop: Inner PropertyOp}
      If $U_n$ is obtained from Algorithm \ref{Alg: Midpoint}, then  spectrum\\ $\sigma\big(\langle {U_{n + 1\slash2}}^\top {U_{n + 1\slash2}} \rangle\big)$ of $\langle {U_{n + 1\slash2}}^\top {U_{n + 1\slash2}} \rangle$ satisfies
      \begin{equation}
        \sigma\big(\langle {U_{n + 1\slash2}}^\top {U_{n + 1\slash2}} \rangle\big)\subset[0, 1].
      \end{equation}
    \end{proposition}
%    \begin{proof}
%        We reach the conclusion by \eqref{Equ: Midpoint Intermediate1Op} and Lemma \ref{Lemma: Inner PropertyOp}.
%    \end{proof}

    Since
    \[\sigma\big(\langle {U}^\top {U}\rangle\big) = \{1\},\]
    for all $U\in \mathcal{M}^N$, we see from Proposition \ref{Prop: Inner PropertyOp} that for the  midpoint scheme of the gradient flow based model, the auxiliary updating points are inside the Stiefel manifold. Nevertheless, we understand from Lemma 3.2 in \cite{dai2017conjugate} that the auxiliary points for the retraction optimization method are outside the Stiefel manifold. In fact, since $U_{n}\in\mathcal{M}^N$ and $\langle {D_{n}}^\top  U_{n} \rangle= 0$, we have
    \[\langle{\tilde U_n}^\top {\tilde U_n}\rangle = I_N + (\Delta t_n)^2 \langle {D_{n}}^\top D_{n}\rangle,\]
    for $\tilde U_{n} = U_{n} + \Delta t_n D_{n}$ and obtain \cite{dai2017conjugate}
    \begin{proposition}\label{Prop: SVDOp}
      Suppose $U_{n}\in\mathcal{M}^N$, $\langle {D_{n}}^\top U_{n}\rangle = 0$, and~$\tilde U_{n} = U_{n} + \Delta t_n D_{n}$~is the auxiliary point of retraction optimization method, then
      \begin{equation}
        \sigma\big( \langle {\tilde U_{n}}^\top {\tilde U_n} \rangle \big) \subset \Big[1, 1 + (\Delta t_n)^2 \Vert \langle {D_{n}}^\top D_{n}\rangle \Vert_2 \Big].
      \end{equation}
    \end{proposition}

    \subsection{Convergence}

    Now we investigate the convergence of the midpoint scheme. First we show that the energy decreases for small time step. In this section, we always assume that $\nabla E$ is local Lipschitz continuous in the neighborhood of a local minimizer $U^*\in(V_{N_g})^N\bigcap\mathcal{M}^N$ as follows
      \begin{equation}\label{Equ: Local LipschitzOp}
      \begin{split}
        ||| \nabla E(U_i) - \nabla E(U_j) ||| \leqslant L ||| U_i - U_j |||, \quad\forall U_i, U_j\in B\big(U^*, \max\{\delta_a, \delta_b\}\big)
      \end{split}
      \end{equation}

    \begin{lemma}\label{Lemma: Energy decreaseOp}
    There holds
      \begin{equation}\label{Equ: GradientGrassLip}
       \begin{split}
         ||| \nabla_{G}E(U_i) - \nabla_{G}E(U_j) ||| \leqslant L_1 ||| U_i - U_j |||, \quad \forall U_i, U_j\in B\big(U^*, \max\{\delta_a, \delta_b\}\big),
       \end{split}
       \end{equation}
     where $L_1 = 2\alpha\big(  2||| \nabla E(U^*) ||| + 2L \max\{\delta_a, \delta_b\} + \alpha L\big)$.
      Moreover, there exists a upper bound $\delta_s$ of $s$ that
      \begin{equation}\label{Equ: Energy DifOp}
      \begin{split}
        E(U) - E\big(g(U, s)\big) \geqslant \frac{s}{4N}\Big|\Big|\Big| \nabla_G E\Big(\frac{g(U,s) + E(U)}{2}\Big)\Big|\Big|\Big|^2,\\ \forall U \in B(U^*, \delta_a)\bigcap\mathcal{M}^N, \forall s \in[0, \delta_s],
      \end{split}
      \end{equation}
      where $\delta_a, \delta_b$ are defined in Lemma \ref{Lemma: Implicit Function TheoryOp}.
    \end{lemma}

    \begin{proof}
    First, we have that $||| \nabla E(U) |||$ is bounded over $B(U^*, \max\{\delta_a, \delta_b\})$ since
    \begin{equation}\label{Equ: localBoundedOp}
    \begin{split}
      & ||| \nabla E(U_i) ||| \leqslant ||| \nabla E(U^*) ||| + ||| \nabla E(U_i) - \nabla E(U^*) ||| \\
      \leqslant& ||| \nabla E(U^*) ||| + L ||| U_i - U^* ||| \leqslant ||| \nabla E(U^*) ||| + L \max\{\delta_a, \delta_b\},
    \end{split}
    \end{equation}
    which together with \eqref{Equ: Local LipschitzOp} and \eqref{Equ: localBoundedOp} leads to
    \begin{equation}\label{Equ: Lip2Op}
    \begin{split}
      &||| \nabla E(U_i) \langle U_i^\top U_i \rangle - \nabla E(U_j) \langle U_j^\top U_j \rangle |||\\
      \leqslant& \big|\big|\big| \nabla E(U_i) \big(\langle U_i^\top U_i \rangle - \langle U_j^\top U_j \rangle \big) \big|\big|\big|
      + ||| \big(\nabla E(U_i) -\nabla E(U_j)\big) \langle U_j^\top U_j\rangle ||| \\
      \leqslant& \big( ||| \nabla E(U^*) ||| + L \max\{\delta_a, \delta_b\}\big) \big(||| U_i ||| + ||| U_j |||\big) ||| U_i - U_j |||\\
      &+ L ||| U_j |||^2 ||| U_i - U_j |||\\
      \leqslant& \alpha \big( 2||| \nabla E(U^*) ||| + 2L \max\{\delta_a, \delta_b\} + \alpha L\big)||| U_i - U_j |||,
    \end{split}
    \end{equation}
    and
    \begin{equation}\label{Equ: Lip3Op}
      \begin{split}
        &||| U_i \langle U_i^\top \nabla E(U_i)\rangle - U_j \langle U_j^\top \nabla E(U_j)\rangle ||| \\
        \leqslant& \big|\big|\big| U_i \big(\langle U_i^\top \nabla E(U_i)\rangle - \langle U_j^\top \nabla E(U_j)\rangle\big) \big|\big|\big|
        + ||| (U_i - U_j) \langle U_j^\top \nabla E(U_j)\rangle ||| \\
        \leqslant& ||| U_i ||| \big(||| \nabla E(U_i) ||| + L||| U_j |||\big) ||| U_i - U_j ||| + ||| U_j|||\cdot||| \nabla E(U_j)|||\cdot ||| U_i - U_j |||\\
        \leqslant& \alpha\big(  2||| \nabla E(U^*) ||| + 2L \max\{\delta_a, \delta_b\} + \alpha L\big)||| U_i - U_j |||,
      \end{split}
    \end{equation}
    where $\alpha = \max\big\{ ||| U ||| : U\in B\big(U^*, \max\{\delta_a, \delta_b\}\big)\big\}$.

    Due to the triangle inequality
    \begin{equation}\label{Equ: Lip1Op}
    \begin{split}
      &||| \nabla E(U_i) - \nabla E(U_j)|||\\
      \leqslant& ||| \nabla E(U_i) \langle U_i^\top U_i \rangle - \nabla E(U_j) \langle U_j^\top U_j \rangle ||| + ||| U_i \langle U_i^\top \nabla E(U_i)\rangle - U_j \langle U_j^\top \nabla E(U_j)\rangle |||,
    \end{split}
    \end{equation}
    we obtain from \eqref{Equ: Lip2Op} and \eqref{Equ: Lip3Op} that
    \begin{equation}
      \begin{split}
        &||| \nabla E(U_i) - \nabla E(U_j)||| \leqslant L_1 ||| U_i - U_j |||.
      \end{split}
    \end{equation}

    Now we are going to prove the remainder. For given $s\in[0, \delta^*]$, Lemma \ref{Lemma: Implicit Function TheoryOp} tells us that $g(U, s)$ exists uniquely. Then we define $S(t) = t g(U, s) + (1 - t) U$ for $t\in[0,1]$, and see that $E\big( S(t) \big)$ is differentiable in (0,1). We understand that there exists a $\xi\in (0,1)$ such that
    \begin{equation}
      \begin{split}
        &E(g(U, s)) - E(U) = E(S(1)) - E(S(0)) = \textup{tr} \big\langle \nabla E\big(S(\xi)\big)^\top \frac{\textup{d}}{\textup{d}t}S(\xi)\big\rangle\\
        =& \textup{tr}\big\langle \nabla E\big(S(\xi)\big)^\top \big(g(U, s) - U\big) \big\rangle
        =-s\;\textup{tr}\Big\langle\nabla E\big(S(\xi)\big)^\top \mathcal{A}_{S(\frac12)}S\big(\frac12\big)\Big\rangle.
        \end{split}
    \end{equation}
    We divide the left part into two terms and obtain
    \begin{equation}\label{Equ: EstimateOp}
      \begin{split}
        E\big(g(U, s)\big) &- E(U) =-s\;\textup{tr}\left\langle\nabla E\Big(S\big(\frac12\big)\Big)^\top \mathcal{A}_{S(\frac12)}S\big(\frac12\big)\right\rangle\\
        &+ s\;\textup{tr}\left\langle \Big(\nabla E\Big(S\big(\frac12\big)\Big) - \nabla E\big(S(\xi)\big)\Big)^\top \mathcal{A}_{S(\frac12)}S\big(\frac12\big)\right\rangle.
      \end{split}
    \end{equation}

    For the first term, we see that
    \begin{equation}
    \begin{split}
      &\textup{tr}\left\langle\nabla E\Big(S\big(\frac12\big)\Big)^\top \mathcal{A}_{S(\frac12)}S\big(\frac12\big)\right\rangle = -\frac12\textup{tr} (\mathcal{A}_{S(\frac12)})^2\\
      =&\frac12\textup{tr} (\mathcal{A}_{S(\frac12)})^*(\mathcal{A}_{S(\frac12)})   =\frac12\big\Vert \mathcal{A}_{S(\frac12)} \big\Vert^2.
    \end{split}
    \end{equation}
    Due to Proposition \ref{Prop: Inner PropertyOp}, we have
    \begin{equation}
      \begin{split}
        &\Big|\Big|\Big| \mathcal{A}_{S(\frac12)}S\big(\frac12\big)\Big|\Big|\Big| \leqslant \Vert \mathcal{A}_{S(\frac12)}\Vert \cdot \Big|\Big|\Big| S\big(\frac12\big) \Big|\Big|\Big| \leqslant \sqrt{N} \big\Vert \mathcal{A}_{S(\frac12)}\big\Vert,
      \end{split}
    \end{equation}
    thus
    \begin{equation}\label{Equ: Tmp1Op}
      \begin{split}
        &\textup{tr}\left\langle\nabla E\Big(S\big(\frac12\big)\Big)^\top \mathcal{A}_{S(\frac12)}S\big(\frac12\big)\right\rangle \geqslant \frac{1}{2N} \Big|\Big|\Big| \mathcal{A}_{S(\frac12)}S\big(\frac12\big)\Big|\Big|\Big|^2.
      \end{split}
    \end{equation}

    For the second term of the last line in \eqref{Equ: EstimateOp}, since
    \begin{equation*}
    ||| S(t) - U^*||| = \big|\big|\big| t\big(g(U, s) - U^*\big) + (1 - t)(U - U^*) \big|\big|\big|\leqslant \max\{\delta_a, \delta_b\}, \forall t \in[0,1],
    \end{equation*}
    by local Lipschitz continuity of $\nabla E$, we get
    \begin{equation}\label{Equ: Tmp2Op}
    \begin{split}
      &\textup{tr}\left\langle\Big(\nabla E\Big(S\big(\frac12\big)\Big) - \nabla E\big(S(\xi)\big)\Big)^\top \mathcal{A}_{S(\frac12)}S\big(\frac12\big)\right\rangle\\
      \leqslant& L \big|\big|\big| S\big(\frac12\big) - S(\xi) \big|\big|\big| \cdot \big|\big|\big| \mathcal{A}_{S(\frac12)}S\big(\frac12\big) \big|\big|\big| \leqslant \frac{sL}{2} \big|\big|\big| \mathcal{A}_{S(\frac12)}S\big(\frac12\big)\big|\big|\big|^2.
    \end{split}
    \end{equation}

    Combining \eqref{Equ: EstimateOp} with \eqref{Equ: Tmp1Op} and \eqref{Equ: Tmp2Op}, we have
    \begin{equation}
      \begin{split}
        &E(U) - E\big(g(U, s)\big) \geqslant s\big(\frac{1}{2N} - \frac{sL}{2}\big)\Big|\Big|\Big| \nabla_G E\Big(\frac{g(U,s) + E(U)}{2}\Big)\Big|\Big|\Big|^2
      \end{split}
    \end{equation}
and reach the conclusion when $\delta_s = \min\big\{{1}\slash(2NL), \delta^*\big\}$.
    \end{proof}

%    \begin{rema}
%      In the proof of Lemma \ref{Lemma: Energy decreaseOp}, $L_1$ is overestimated. Actually, such $L_1$ may be much smaller compared with our rough estimation.
%    \end{rema}

     We define a mapping
      \begin{equation*}
          \hat g: B\big([U^*], \delta_a\big)\times[0, \delta^*] \rightarrow B\big([U^*], \delta_b\big)
      \end{equation*}
    as follows
      \begin{equation*}
        \hat g\big([U], s\big) = \big[g\big(\arg\min\limits_{\tilde U\in [U]} ||| \tilde U - U^* |||, s\big)\big].
      \end{equation*}
    and we always assume that the local minimizer $[U^*]\in\mathcal{G}^N$ is the unique critical point of \eqref{Equ: EnergyGrassmann} in $B\big([U^*], \delta_c\big)$ for some $\delta_c\in (0,\delta_1]$ from now on.
%
%    As we can see, $g(U, s)$ may not belong to $B(U^*, \delta_a)$, which means that we can not repeat this updating directly. However, if we suppose the local minimizer is the unique critical point on the Grassmann manifold, we can make the range of $\hat g$ is a subset of the first coordinate of the domain of $\hat g$ by some restriction on the domain.

    \begin{lemma}\label{Lemma: Self-mappingOp}
    There holds
      \begin{equation*}
        \hat g\Big(B\big([U^*], \delta_e\big)\bigcap\mathcal{L}_{E_e}\times[0, \delta_T]\Big) \subset B\big([U^*], \delta_e\big)\bigcap\mathcal{L}_{E_e}
      \end{equation*}
      for some $\delta_e > 0$, $E_e\in\mathbb{R}$, $\delta_T \in [0, \delta_s]$ where $\delta_s$ is defined in Lemma \ref{Lemma: Energy decreaseOp}.
    \end{lemma}

    \begin{proof}
    We use the notation in Lemma \ref{Lemma: Implicit Function TheoryOp} and Lemma \ref{Lemma: Energy decreaseOp}. Set $\delta_e = \min\big\{\delta_a, \frac12\delta_c\big\}$ and
    \begin{equation*}
      E_{c} = \min \big\{ E(\tilde U) : [\tilde U] \in \overline{B([U^*], \delta_c)\backslash B([U^*], \delta_e)} \big\}.
    \end{equation*}
    We observe that $[\tilde U]\in B\big([U^*], \delta_e\big)$ if $E(\tilde U) \leqslant \frac{E_c + E (U^*)}{2}\equiv E_e$ and $[\tilde U]\in B\big([C^*], \delta_c\big)$.

    For $[U]\in B\big([U^*], \delta_e\big)$ and $s\in[0, \delta_s]$, we observe that there exists a $\tilde U\in [U]$ such that  $||| \tilde U - U^* ||| = \textup{dist}\big([U], [U^*]\big)\leqslant \delta_e$. For simplicity, we still use $U$ to denote $\tilde U$. We obtain from Lemma \ref{Lemma: Energy decreaseOp} that $g(U, s) \in B(U^*, \delta_b)$ and $E\big(g(U, s)\big) \leqslant E(U) \leqslant E_e$ for any fixed $s\in[0,\delta_s]$.

    Due to
    \begin{equation}
      \begin{split}
        &\textup{dist}\big([g(U, s)], [U]\big) \leqslant ||| g(U,s) - U ||| \leqslant s \Big|\Big|\Big| \mathcal{A}_{S(\frac12)}S\big(\frac12\big) \Big|\Big|\Big|,
      \end{split}
    \end{equation}
    and
    \begin{equation}
      \begin{split}
        &\Big|\Big|\Big| \mathcal{A}_{S(\frac12)}S\big(\frac12\big) \Big|\Big|\Big| = \Big|\Big|\Big| \mathcal{A}_{S(\frac12)}S\big(\frac12\big) - \mathcal{A}_{U^*}U^* \Big|\Big|\Big| \leqslant L_1\Big|\Big|\Big| S\big(\frac12\big) - U^* \Big|\Big|\Big|\leqslant L_1 \max\{\delta_a, \delta_b\},
      \end{split}
    \end{equation}
    we obtain
    $$
     \hat g\big([U], s\big) \in B\big([U^*], \delta_c\big), ~~\forall s \in [0,\delta_T],
    $$
    where
    \begin{equation}
    \delta_T =\left\{
    \begin{array}{ll}
       \min\big\{\frac{\delta_c - \delta_e}{L_1 \max\{\delta_a, \delta_b\}}, \delta_s\big\} &    \delta_b > \delta_e,\\
        \delta_s    & \delta_b \leqslant \delta_e.
    \end{array}
    \right.
    \end{equation}
     Since $E\Big(g\big([U], s\big)\Big) \leqslant E_e$, by definition of $E_e$, we have $\hat g\big([U], s\big) \in B\big([U^*], \delta_e\big)$.
    \end{proof}

    \begin{rema}
    Since
    \begin{equation*}
       g(UP, s) = g(U, s)P,~ \forall P\in\mathcal{O}^{N},
    \end{equation*}
    we may directly solve \eqref{Equ: Implicit FunctionOp} to get a representative of $\hat g(U, s)$ with respect to any representative $U$ of $[U]$.
    \end{rema}

    Consequently we arrive at the convergence of the midpoint scheme of the gradient flow based model of Kohn-Sham DFT.

    \begin{theorem}\label{Theo: MainConvergeOp}
       If $[U_{0}]\in B\big([U^*], \delta_e\big)$ and $\sup\{\Delta t_n : n \in \mathbb{N}\} \leqslant \delta_T$, then the sequence $\{U_n\}$ produced by Algorithm \ref{Alg: Midpoint} satisfies
      \begin{eqnarray}
       \lim\limits_{n\rightarrow\infty} |||\nabla_{G} E(U_{n})||| = 0,\\
       \lim\limits_{n\rightarrow\infty} E(U_{n}) = E(U^*),\\
       \lim\limits_{n\rightarrow\infty} \textup{dist}([U_{n}], [U^*]) = 0,\label{Equ: DistanceConv}
      \end{eqnarray}
      where $\delta_e$ and $\delta_T$ are defined in Lemma \ref{Lemma: Self-mappingOp}.

    \end{theorem}

    \begin{proof}
    We see from Lemma \ref{Lemma: Energy decreaseOp} that $E(U_{n + 1}) \leqslant E(U_{n})$. Since $B\big([U^*], \delta_e\big)\bigcap\mathcal{L}_{E_e}$ is compact, we obtain from \ref{Lemma: Self-mappingOp} that $\{ E\big([U_{n}]\big)\}_{n = 0}^{\infty}$ is bounded below. So $\lim\limits_{n\rightarrow\infty} E\big([U_{n}]\big)$ exists. Note that \eqref{Equ: Energy DifOp} implies
      \begin{equation}
      \begin{split}
        &E(U_{n}) - E(U_{n + 1}) \geqslant \frac{\Delta t_n}{4N} |||\nabla_G E(U_{n + 1\slash2}) |||^2,
      \end{split}
    \end{equation}
    we have
    \begin{equation*}
    \begin{split}
      &\sum\limits_{n = 0}^\infty \frac{\Delta t_n}{4N}||| \nabla_G E(U_{n + 1\slash2}) |||^2 \leqslant E(U_{0}) - \lim\limits_{n\rightarrow\infty} E(U_{n}) < +\infty,
    \end{split}
    \end{equation*}
    which together with $\sum\limits_{n = 0}^{\infty} \Delta t_n = +\infty$ leads to
    \begin{equation*}
      \inf\big\{ ||| \nabla_G E(U_{k + 1\slash2}) ||| : k\in\mathbb{N}, k\geqslant n\big\} = 0,~\forall n \in\mathbb{N}.
    \end{equation*}
    Therefore
    \begin{equation}
    \begin{split}
      &\liminf\limits_{n\rightarrow\infty} ||| \nabla_G E(U_{n + 1\slash2}) ||| = 0.
    \end{split}
    \end{equation}
    Consequently, there exists a subsequence $\{U_{n_{k + 1\slash2}}\}_{k = 0}^{\infty}$ such that
    \begin{equation*}
      \lim\limits_{k\rightarrow\infty} |||  U_{n_{k + 1}} - U_{n_{k}} ||| \leqslant \delta_T \lim\limits_{k\rightarrow\infty} ||| \nabla_G E(U_{n_{k + 1\slash2}}) ||| = 0.
    \end{equation*}
    Note that
    \begin{equation*}
    \mathcal{\hat S}\equiv \big\{U\in(V_{N_g})^N:[U]\in B\big([U^*], \delta_e\big)\bigcap \mathcal{L}_{E_e}\big\}
    \end{equation*}
    is compact, we have a subsequence of $\{U_{n_{k}}\}_{k = 0}^{\infty}$, for simplicity, we write as $\{U_{n_{k}}\}_{k = 0}^{\infty}$, satisfying
    \begin{equation}
      \lim\limits_{k\rightarrow\infty} U_{n_k} = \bar U
    \end{equation}
    for some $\bar U \in\mathcal{\hat S}$. Then
    \begin{equation*}
        \lim\limits_{k\rightarrow\infty} U_{n_{k + 1\slash 2}} = \lim\limits_{k\rightarrow\infty} U_{n_k} + \frac{U_{n_{k + 1}} - U_{n_k}}{2} = \bar U.
    \end{equation*}
    According to the Proposition \ref{Prop: GradientFlowOp} and Lemma \ref{Lemma: Energy decreaseOp}, we have
    \begin{equation*}
    \nabla_{G} E(\bar U) = \mathcal{A}_{\bar U}\bar U = 0.
    \end{equation*}
    This means $\liminf\limits_{n\rightarrow\infty} ||| \nabla_{G} E(U_{n})||| = 0$.

    Lemma \ref{Lemma: Self-mappingOp} tells us that $[\bar U]\in B\big([U^*], \delta_e\big)\bigcap\mathcal{L}_{E_e}\subset B\big([U^*], \delta_c\big)$.
    Due to the uniqueness of the critical point in $B\big([C^*], \delta_c\big)$, we have $[\bar U] = [U^*]$ and
    \begin{equation*}
        \lim\limits_{n\rightarrow\infty}E\big(U_{n}\big) = \lim\limits_{k\rightarrow\infty} E\big(U_{n_k}\big) = E(U^*).
    \end{equation*}

    Next we show that $\lim\limits_{n\to\infty}\textup{dist}\big([U_{n}], [U^*]\big) = 0$. If it is not true, then there exists a subsequence $\{U_{n_l}\}_{l = 0}^{\infty}$ and $\check\delta > 0$ that $\textup{dist}\big([U_{n_l}], [U^*]\big)\geqslant \check\delta$. Since $\mathcal{\hat S}$ is compact, there exists a subsequence of $\{U_{n_l}\}_{l = 0}^{\infty}$, for simplicity again written as $\{U_{n_l}\}_{l = 0}^{\infty}$, which satisfies $\lim\limits_{l\rightarrow\infty} U_{n_l} = \check U$ for some $\check U\in\mathcal{\hat S}$. Thus we have
    \begin{equation*}
        E(\check U) = \lim\limits_{l\rightarrow\infty} E(U_{n_l}) = E(U^*).
    \end{equation*}
    Again by the uniqueness of local minimizer in $B\big([U^*], \delta_e\big)$, we obtain $[\check U] = [U^*]$, which contradicts the assumption $\textup{dist}([U_{n_l}], [U^*]) \geqslant \check\delta$.

    Clearly there exists $P_n\in\mathcal{O}^N$ that
    \begin{equation}
      |||U_nP_n - U^*|||= \textup{dist}([U_n], [U^*]),
    \end{equation}
    then
    \begin{equation*}
       \lim\limits_{n\rightarrow\infty} \big|\big|\big| \nabla_{G} E(U_n) \big|\big|\big| =  \lim\limits_{n\rightarrow\infty} \big|\big|\big| \nabla_{G} E(U_nP_n) \big|\big|\big| = \big|\big|\big| \nabla_{G} E(U^*) \big|\big|\big| = 0.
    \end{equation*}
    This completes the proof.
    \end{proof}

    Theorem \ref{Theo: MainConvergeOp} shows that the approximations produced by Algorithm \ref{Alg: Midpoint} converge to the unique local minimizer under some mild assumptions, in which no uniform gap between the required  and nonrequired eigenvalues, or namely \emph{uniformly well posed} (UWP) property in \cite{UWP_Bris, liu2014convergence, liu2015analysis, UWP_Yang}, is needed.
   % Our assumptions are so weak that we can only get convergence result rather than a convergence rate estimation, while the previous work of analysis of convergence rate of eigenvalue model requires much stronger assumption such as UWP.

    %%%%%%%Stronger Assumptions%%%%%%%%
    \subsection{Convergence rate}
    We are able to have some convergence rate of the approximations obtained from Algorithm \ref{Alg: Midpoint}.

    \begin{lemma}\label{Lemma: HessAppr}
      For $U\in B(U^*, \min\{\delta_3, \delta_a\})\bigcap\mathcal{M}^N$ and $\tau\in(0, \delta_T]$, set
      \begin{equation}
        \begin{split}
          U_+ = \Big(I + \frac{\tau}{2}\mathcal{A}_{U_+}\Big)^{-1}U,\\
          U_- = \Big(I - \frac{\tau}{2}\mathcal{A}_{U_-}\Big)^{-1}U.\\
        \end{split}
      \end{equation}
      If \eqref{Equ: Hess} holds true, then there exists some $\delta_{r_1} > 0$ such that
      \begin{equation}
      \begin{split}
        \textup{tr}\Big( \big\langle \big(U_{+} - U_{-}\big)^\top \big(\nabla_G E(U_{+}) - \nabla_G E(U_{-})\big) \big\rangle\Big) \geqslant \frac{\sigma}{2} ||| U_{+} - U_{-}|||^2
      \end{split}
      \end{equation}
      for all $\tau\in(0,\delta_{r_1}]$ and $U\in B(U^*, \min\{\delta_3, \delta_a\})\bigcap\mathcal{M}^N$, where $\delta_{r_1}\in(0,\delta_T]$ is a positive constant, $\delta_T$ is defined in Theorem \ref{Theo: MainConvergeOp} and $\delta_a$ is defined in Lemma \ref{Lemma: Implicit Function TheoryOp}.
    \end{lemma}

    \begin{proof}
      Note that
      \begin{equation}\label{Equ: Explain0}
      \begin{split}
        &\textup{tr}\Big( \big\langle \big(U_{+} - U_{-}\big)^\top \big(\nabla_G E(U_{+}) - \nabla_G E(U_{-})\big) \big\rangle\Big)\\
        =&\frac{\tau}{2}\textup{tr}\Big( \big\langle \big(\nabla_G E(U_{+}) + \nabla_G E(U_{-}) \big)^\top \big(\nabla_G E(U_{+}) - \nabla_G E(U_{-}) \big) \big\rangle \Big).
      \end{split}
      \end{equation}
%      We may review $U_{n + 1\slash2}$ and $U_{n - 1\slash2}$ as functions depending on $U_n$, $\Delta t_n$ and $\Delta t_{n - 1}$.
      Since
      \begin{equation}
        \begin{split}
            &\lim\limits_{\tau \to 0}\frac{\textup{tr}\Big( \big\langle \big( \nabla_G E(U_{+})  \big)^\top \big(\nabla_G E(U_{+}) - \nabla_G E(U_{-}) \big) \big\rangle \Big)}{\tau}\\
            =& \textup{Hess}_G E(U)[\nabla E_G(U), \nabla E_G(U)] \geqslant \sigma ||| \nabla E_G(U) |||^2 = \lim\limits_{\tau \to 0} \frac{\sigma|||U_{+} - U_{-}|||^2}{\tau^2},
        \end{split}
      \end{equation}
      we have
      \begin{equation}\label{Equ: Explain1}
      \begin{split}
        \lim\limits_{\tau\to 0} \tau \frac{\textup{tr} \Big( \big\langle \big( \nabla_G E(U_{+})  \big)^\top \big(\nabla_G E(U_{+}) - \nabla_G E(U_{-}) \big) \big\rangle \Big) } {|||U_{+} - U_{-}|||^2} \geqslant \sigma.
      \end{split}
      \end{equation}
      Similarly,
      \begin{equation}\label{Equ: Explain2}
      \begin{split}
        \lim\limits_{\tau\to 0} \tau \frac{\textup{tr} \Big( \big\langle \big( \nabla_G E(U_{-})  \big)^\top \big(\nabla_G E(U_{+}) - \nabla_G E(U_{-}) \big) \big\rangle \Big) } {|||U_{+} - U_{-}|||^2} \geqslant \sigma.
      \end{split}
      \end{equation}
      Therefore, we see from \eqref{Equ: Explain0}, \eqref{Equ: Explain1} and \eqref{Equ: Explain2} that
      \begin{equation}
      \begin{split}
        \lim\limits_{\tau\to 0} \frac{\textup{tr}\Big( \big\langle \big(U_{+} - U_{-}\big)^\top \big(\nabla_G E(U_{+}) - \nabla_G E(U_{-})\big) \big\rangle\Big)}{|||U_{+} - U_{-}|||^2}\geqslant\sigma.
      \end{split}
      \end{equation}
    Then we know for any $U\in B(U^*, \min\{\delta_3, \delta_a\})\bigcap\mathcal{M}^N$, there exists a $\delta_U > 0$ that
    \begin{equation}
      \begin{split}
        \frac{\textup{tr}\Big( \big\langle \big(U_{+} - U_{-}\big)^\top \big(\nabla_G E(U_{+}) - \nabla_G E(U_{-})\big) \big\rangle\Big)}{|||U_{+} - U_{-}|||^2} >  \frac{\sigma}{2}
      \end{split}
      \end{equation}
    for all $\tau\in(0, \delta_U]$. We denote
    \begin{equation}
    \begin{split}
      \mathcal{C}_U := \bigg\{ V: \frac{\textup{tr}\Big( \big\langle \big(V_{+} - V_{-}\big)^\top \big(\nabla_G E(V_{+}) - \nabla_G E(V_{-})\big) \big\rangle\Big)}{|||V_{+} - V_{-}|||^2}> \frac{\sigma}{2}, \\
      \forall \tau\in(0, \delta_U]\bigg\},
    \end{split}
    \end{equation}
    then we have $\mathcal{C}_U\neq\varnothing$ since $U\in\mathcal{C}_U$. Due to
    \[B(U^*, \min\{\delta_3, \delta_a\})\bigcap\mathcal{M}^N\subset\bigcup\limits_{U\in B(U^*, \min\{\delta_3, \delta_a\})\bigcap\mathcal{M}^N} \mathcal{C}_U\]
    and the compactness of $B(U^*, \min\{\delta_3, \delta_a\})\bigcap\mathcal{M}^N$, we know there exist finite sets $\mathcal{C}_{U_{(l)}}$ that
    \[B(U^*, \min\{\delta_3, \delta_a\})\bigcap\mathcal{M}^N\subset\bigcup\limits_{l = 1}^\ell \mathcal{C}_{U_{(l)}}.\]
    Set
    \[\delta_{r_1} = \min\{\delta_{U_{(1)}}, \delta_{U_{(2)}}, \ldots,\delta_{U_{(\ell)}}, \delta_T \},\]
    and we complete the proof.
    \end{proof}

    \begin{lemma}\label{Lemma: GradientConvEst}
    For $U\in B(U^*, \min\{\delta_a, \delta_b\})\bigcap\mathcal{M}^N$ and $\tau\in(0, \delta_T]$, if we set
      \begin{equation}
        \begin{split}
          U_+ &= \Big(I + \frac{\tau}{2}\mathcal{A}_{U_+}\Big)^{-1}U,\\
          \bar {U}_{+} &= 2U_+ - U,\\
        \end{split}
      \end{equation}
      then there exists some $\delta_{r_2} > 0$ that satisfies
      \begin{equation}
        E(U) - E(\bar{U}_{+}) \leqslant \frac{\tau(L + 3)}{2} ||| \nabla_G E(U_{+})|||^2,
      \end{equation}
      for all $\tau\in(0,\delta_{r_2}]$ and $U\in B(U^*, \min\{\delta_a, \delta_b\})\bigcap\mathcal{M}^N$, where $\delta_{r_2}\in(0,\delta_T]$ is a positive constant, $\delta_T$ is defined in Theorem \ref{Theo: MainConvergeOp} and $\delta_a$, $\delta_b$ and $L$ are defined in Lemma \ref{Lemma: Energy decreaseOp}.
    \end{lemma}
    \begin{proof}
      We see from \eqref{Equ: EstimateOp} and \eqref{Equ: Tmp2Op} that
      \begin{equation}
        E(U) - E(\bar{U}_{+}) \leqslant \tau \textup{tr}\langle \nabla E(U_{+})^\top \nabla_G E(U_{+})\rangle + \frac{\tau L}{2} ||| \nabla_G E(U_{+})|||^2.
      \end{equation}
      Note that
      \begin{equation}
      \begin{split}
        \lim\limits_{\tau \to 0} \frac{\textup{tr}\langle \nabla E(U_{+})^\top \nabla_G E(U_{+})\rangle}{|||\nabla_G E(U_{+})|||^2} =
         \frac{\textup{tr}\langle \nabla E(U)^\top \nabla_G E(U)\rangle}{|||\nabla_G E(U)|||^2} = 1.
      \end{split}
      \end{equation}
    Then we see that for any $U\in B(U^*, \min\{\delta_a, \delta_b\})\bigcap\mathcal{M}^N$, there exists a $\hat\delta_U > 0$ that
    \begin{equation}
      \begin{split}
       \frac{\textup{tr}\langle \nabla E(U_{+})^\top \nabla_G E(U_{+})\rangle}{|||\nabla_G E(U_{+})|||^2} < \frac{3}{2}
      \end{split}
      \end{equation}
    for all $\tau\in(0, \hat\delta_U]$. We denote
    \begin{equation}
    \begin{split}
      \hat{\mathcal{C}}_U := \bigg\{ V: \frac{\textup{tr}\langle \nabla E(V_{+})^\top \nabla_G E(V_{+})\rangle}{|||\nabla_G E(V_{+})|||^2} <  \frac{3}{2}, \forall \tau\in(0, \hat\delta_U]\bigg\},
    \end{split}
    \end{equation}
    then we have $\hat{\mathcal{C}}_U\neq\varnothing$ since $U\in\hat{\mathcal{C}}_U$. Due to
    \[B(U^*, \min\{\delta_a, \delta_b\})\bigcap\mathcal{M}^N\subset\bigcup\limits_{U\in B(U^*, \min\{\delta_a, \delta_b\})\bigcap\mathcal{M}^N} \hat{\mathcal{C}}_U\]
    and the compactness of $B(U^*, \min\{\delta_a, \delta_b\})\bigcap\mathcal{M}^N$, we note that there exist finite sets $\hat{\mathcal{C}}_{U_{(l)}}$ that
    \[B(U^*, \min\{\delta_a, \delta_b\})\bigcap\mathcal{M}^N\subset\bigcup\limits_{l = 1}^{\hat\ell} \hat{\mathcal{C}}_{U_{(l)}}.\]
    Set
    \[\delta_{r_2} = \min\{\hat\delta_{U_{(1)}}, \hat\delta_{U_{(2)}}, \ldots,\hat\delta_{U_{(\hat\ell)}}, \delta_T \},\]
    and we arrive at the conclusion.
    \end{proof}

    \begin{theorem}\label{Theo: ConvRate}
      Suppose Hessian coercivity holds true as \eqref{Equ: Hess}. If $[U_{0}]\in B\big([U^*], \delta_e\big)$ and $\Delta t_n = \tau \leqslant \delta_{r_1}, \forall n\geqslant N_0$, then the sequence $\{U_n\}$ produced by Algorithm \ref{Alg: Midpoint} satisfies
      \begin{equation}
      \begin{split}
        |||\nabla_G E(U_{n})||| \leqslant \Big(1 + \frac{L_1\tau}{2}\Big) \Big(\frac{4 + \tau^2L_1^2 - 2\sigma \tau}{4 + \tau^2L_1^2 + 2\sigma \tau}\Big)^{(n - N_0 + 1)\slash 2} |||\nabla_G E(U_{N_0 - 1\slash2})|||, \\
        \forall n \geqslant N_0,
      \end{split}
      \end{equation}
      where $N_0$ is a positive integer, $\delta_e$ and $\delta_T$ are defined in Lemma \ref{Lemma: Self-mappingOp}, $L_1$ is defined in \eqref{Equ: GradientGrassLip} and $\delta_{r_1}$ is defined in Lemma \ref{Lemma: HessAppr}.

      Moreover, if $\Delta t_n = \tau \leqslant \min\{\delta_{r_1}, \delta_{r_2}\}$, $\forall n\geqslant N_1$, then
      \begin{equation}
      \begin{split}
      &E(U_n) - E(U^*)\\
      \leqslant& \frac{(L + 3)(4 + \tau^2L_1^2 + 2\sigma \tau)}{8\sigma}  \Big(\frac{4 + \tau^2L_1^2 - 2\sigma \tau}{4 + \tau^2L_1^2 + 2\sigma \tau}\Big)^{n - N_1 + 1} |||\nabla_G E(U_{N_1 - 1\slash2})|||^2, \forall n \geqslant N_1,
      \end{split}
      \end{equation}
      where $N_1\geqslant N_0$ is a positive integer, $L$ is defined in \eqref{Equ: Local LipschitzOp} and $\delta_{r_2}$ is defined in Lemma \ref{Lemma: GradientConvEst}.
    \end{theorem}

    \begin{proof}
    Due to \eqref{Equ: DistanceConv}, there exist $N_0\in\mathbb{N}$ that $U_n\in B(U^*, \min\{\delta_3, \delta_a\})\bigcap\mathcal{M}^N$, $\forall n \geqslant N_0$ and $N_1\geqslant N_0$ that $U_n\in B(U^*, \min\{\delta_3, \delta_a, \delta_b\})\bigcap\mathcal{M}^N$, $\forall n \geqslant N_1$.

    We observe that
    \begin{equation}\label{Equ: Rate1}
    \begin{split}
      &\tau||\nabla_G E(U_{n + 1\slash2})|||^2 - \tau|||\nabla_G E(U_{n - 1\slash2})|||^2\\
      =& \tau \textup{tr}\Big( \big\langle \big( \nabla_G E(U_{n + 1\slash2}) + \nabla_G E(U_{n - 1\slash2}) \big)^\top\big(\nabla_G E(U_{n + 1\slash2}) - \nabla_G E(U_{n - 1\slash2}) \big) \big\rangle \Big)\\
      =& -2\textup{tr}\Big(\big\langle \big(U_{n + 1\slash2} - U_{n - 1\slash2}\big)^\top \big( \nabla_G E(U_{n + 1\slash2}) - \nabla_G E(U_{n - 1\slash2}) \big) \big\rangle \Big).
    \end{split}
    \end{equation}
    And the parallelogram identity yields
    \begin{equation}
    \begin{split}
      &4||| U_{n + 1\slash2} - U_{n - 1\slash2} |||^2 = \tau^2 |||  \nabla_G E(U_{n + 1\slash2}) +  \nabla_G E(U_{n - 1\slash2}) |||^2\\
      =& 2\tau^2||| \nabla_G E(U_{n + 1\slash2})|||^2 + 2\tau^2||| \nabla_G E(U_{n + 1\slash2})|||^2\\
       & -\tau^2||| \nabla_G E(U_{n + 1\slash2}) - \nabla_G E(U_{n - 1\slash2}) |||^2,
    \end{split}
    \end{equation}
    which together with
    \begin{equation}
      \begin{split}
        &|||\nabla_G E(U_{n + 1\slash2}) - \nabla_G E(U_{n - 1\slash2}) |||^2 \leqslant L_1^2 |||U_{n + 1\slash2} - U_{n - 1\slash2} |||^2
      \end{split}
    \end{equation}
    leads to
    \begin{equation}\label{Equ: Rate2}
    \begin{split}
      &||| U_{n + 1\slash2} - U_{n - 1\slash2} |||^2 \geqslant \frac{2\tau^2}{4 + \tau^2L_1^2} \Big( |||\nabla_G E(U_{n + 1\slash2})|||^2 + |||\nabla_G E(U_{n - 1\slash2})|||^2 \Big).
    \end{split}
    \end{equation}

    Thus we obtain from Lemma \ref{Lemma: HessAppr}, \eqref{Equ: Rate2} and \eqref{Equ: Rate1} that
    \begin{equation}
    \begin{split}
      &\tau|||\nabla_G E(U_{n + 1\slash2})|||^2 - \tau|||\nabla_G E(U_{n - 1\slash2})|||^2\\
      \leqslant& - \frac{2\sigma\tau^2}{4 + \tau^2L_1^2} \Big( |||\nabla_G E(U_{n + 1\slash2})|||^2 + |||\nabla_G E(U_{n - 1\slash2})|||^2 \Big), ~\forall n \geqslant N_0.
    \end{split}
    \end{equation}
    Namely, we have
    \begin{equation}
    \begin{split}
       \Big(1 + \frac{2\sigma \tau}{4 + \tau^2L_1^2} \Big) |||\nabla_G E(U_{n + 1\slash2})|||^2 \leqslant& \Big( 1 - \frac{2\sigma \tau}{4 + \tau^2L_1^2} \Big) |||\nabla_G E(U_{n - 1\slash2})|||^2, \forall n \geqslant N_0,
    \end{split}
    \end{equation}
    or
    \begin{equation}
      |||\nabla_G E(U_{n + 1\slash2})||| \leqslant \Big(\frac{4 + \tau^2L_1^2 - 2\sigma \tau}{4 + \tau^2L_1^2 + 2\sigma \tau}\Big)^{1\slash2} |||\nabla_G E(U_{n - 1\slash2})|||, \forall n \geqslant N_0.
    \end{equation}
    Therefore,
    \begin{equation}\label{Equ: GradientConvEst}
      |||\nabla_G E(U_{n + 1\slash2})||| \leqslant \Big(\frac{4 + \tau^2L_1^2 - 2\sigma \tau}{4 + \tau^2L_1^2 + 2\sigma \tau}\Big)^{(n - N_0 + 1)\slash 2} |||\nabla_G E(U_{N_0 - 1\slash2})|||, \forall n \geqslant N_0
    \end{equation}
    and
    \begin{equation}
    \begin{split}
      &|||\nabla_G E(U_{n})||| \leqslant |||\nabla_G E(U_{n + 1\slash2})||| + |||\nabla_G E(U_{n}) - \nabla_G E(U_{n + 1\slash2}) |||\\
      \leqslant& |||\nabla_G E(U_{n + 1\slash2})||| + L_1 ||| U_{n} - U_{n + 1\slash2} |||\\
      \leqslant& \Big(1 + \frac{L_1\tau}{2}\Big)|||\nabla_G E(U_{n + 1\slash2})|||\\
      \leqslant& \Big(1 + \frac{L_1\tau}{2}\Big) \Big(\frac{4 + \tau^2L_1^2 - 2\sigma \tau}{4 + \tau^2L_1^2 + 2\sigma \tau}\Big)^{(n - N_0 + 1)\slash 2} |||\nabla_G E(U_{N_0 - 1\slash2})|||, \forall n \geqslant N_0.
    \end{split}
    \end{equation}

    %%%%%%%%%%%%%%%%%%%%%%%%%%%%%%
    Finally, we obtain from Lemma \ref{Lemma: GradientConvEst} that
    \begin{equation}
    \begin{split}
      &E(U_n) - E(U_{n + 1}) \leqslant \frac{\tau (L + 3)}{2} ||| \nabla_G E(U_{n + 1\slash2}) |||^2 \\
      \leqslant&  \frac{\tau (L + 3)}{2}  \Big(\frac{4 + \tau^2L_1^2 - 2\sigma \tau}{4 + \tau^2L_1^2 + 2\sigma \tau}\Big)^{n - N_1 + 1} |||\nabla_G E(U_{N_1 - 1\slash2})|||^2, \forall n \geqslant N_1.
    \end{split}
    \end{equation}
    Consequently,
    \begin{equation}
    \begin{split}
      &E(U_n) - E(U^*)\\
      \leqslant& \frac{(L + 3)(4 + \tau^2L_1^2 + 2\sigma \tau)}{8\sigma}  \Big(\frac{4 + \tau^2L_1^2 - 2\sigma \tau}{4 + \tau^2L_1^2 + 2\sigma \tau}\Big)^{n - N_1 + 1} |||\nabla_G E(U_{N_1 - 1\slash2})|||^2, \forall n \geqslant N_1.
    \end{split}
    \end{equation}
    This completes the proof.
    \end{proof}

    \begin{rema}
%    We can estimate the optimal convergence rate of Algorithm \ref{Alg: Midpoint}.
    Note that
    \begin{equation}
    \begin{split}
      &\frac{4 + \tau^2L_1^2 - 2\sigma \tau}{4 + \tau^2L_1^2 + 2\sigma \tau}
       = 1 - \frac{4\sigma \tau}{4 + \tau^2L_1^2 + 2\sigma \tau}\\
      =&1 - \frac{4\sigma}{4\slash \tau + \tau L_1^2 + 2\sigma }
        \geqslant 1 - \frac{4\sigma}{4L_1 + 2\sigma }
    \end{split}
    \end{equation}
    where the equality holds if and only if $\tau = {2}\slash{L_1}$. As a result,  Algorithm \ref{Alg: Midpoint} possesses the optimal convergence rate if
    \begin{equation}\label{Equ: OptimalRate}
      \tau = \left\{
      \begin{array}{ll}
        \min\{\delta_{r_1}, \delta_{r_2}\},   &\quad \displaystyle\frac{2}{L_1} > \min\{\delta_{r_1}, \delta_{r_2}\},\\
        \displaystyle\frac{2}{L_1},    &\quad \displaystyle\frac{2}{L_1} \leqslant \min\{\delta_{r_1}, \delta_{r_2}\}.
      \end{array}
      \right.
    \end{equation}
    Moreover, if $U_{k + 1\slash2} \neq U_{k - 1\slash2}$ for some $k \geqslant N_0$, then $L_1\geqslant {\sigma}\slash{2}$.
%    because
%    \begin{equation}
%    \begin{split}
%      &\frac{\sigma}{2} |||U_{k + 1\slash2} -  U_{k - 1\slash2}|||^2 \\
%      \leqslant& \textup{tr}\Big( \big\langle \big(U_{n + 1\slash2} - U_{n - 1\slash2}\big)^\top \big(\nabla_G E(U_{n + 1\slash2}) - \nabla_G E(U_{n - 1\slash2})\big) \big\rangle\Big)\\
%      \leqslant& L_1|||U_{k + 1\slash2} -  U_{k - 1\slash2}|||^2.
%    \end{split}
%    \end{equation}
%    Then we observe that there exists a lower bound for convergence rate by
    Notice that
    \begin{equation}
      \frac{4 + \tau^2L_1^2 - 2\sigma \tau}{4 + \tau^2L_1^2 + 2\sigma \tau}\geqslant 1 - \frac{4\sigma}{4L_1 + 2\sigma }\geqslant 0
    \end{equation}
    where the last equality holds when $L_1 = {\sigma}\slash{2}$. Then we see that convergence rate can approach $0$ given proper assumptions in theory.
    \end{rema}

    \section{An orthogonality preserving iteration}
    We understand that the convergence of SCF iteration of nonlinear eigenvalue models can neither be predicted by theory nor by numerics for those systems in large scale with small energy gap. In this section, we propose and analyze an orthogonality preserving iteration scheme based on the gradient flow based model, which is indeed a practical version of the midpoint scheme proposed in section \ref{Sec: Exact}. In implementation of Algorithm \ref{Alg: Midpoint}, we are not able to get the exact  $U_{n+1}$ of \eqref{Equ: MidpointOp}. Some approximation should be taken into account in solving \eqref{Equ: MidpointOp}, which then produces the orthogonality preserving iteration scheme that will be proved to be convergent.

    \subsection{An iteration}
    With the gradient flow based approach, in this subsection, we are able to design a convergent orthogonality preserving iteration scheme for solving the Kohn-Sham equation.  We recall and split   midpoint scheme (\ref{Equ: MidpointOp}) into two equations
    \begin{equation}\label{Equ: MidpointSplitOp}
      \begin{split}
      \frac{U_{n + 1\slash2} - U_{n}}{{\Delta t_n}\slash{2}} &= - \nabla_{G} E(U_{n + 1\slash2}),\\
      \frac{U_{n + 1} - U_{n + 1\slash2}}{{\Delta t_n}\slash{2}} &= - \nabla_{G} E(U_{n + 1\slash2}),
      \end{split}
    \end{equation}
    and provide partition
    \begin{equation}
      0 = t_0 < t_1 < t_2 < \cdots < t_n < \cdots,
    \end{equation}
    where $\lim\limits_{n\rightarrow+\infty} t_n = +\infty$ and $\Delta t_n = t_{n + 1} - t_n$.

    We may solve the first equation of \eqref{Equ: MidpointSplitOp} approximatively and then update the approximation using $U_{n + 1} = 2U_{n + 1\slash2} - U_n$. Consequently, we obtain Algorithm \ref{Alg: Self-consistent}.
  %  \begin{equation}
%      \begin{split}
%        &U_{n + 1\slash2}^{(0)} = U_n,\\
%        &\textup{Solve~}U_{n + 1\slash2}^{(k)}\text{~by}\\
%        &\qquad\frac{U_{n + 1\slash2}^{(k)} - U_{n}}{{\Delta t_n}\slash{2}} = - \mathcal{A}_{U_{n + 1\slash2}^{(k - 1)}} U_{n + 1\slash2}^{(k)}, ~k = 1, 2, \ldots, p_n,\\
%        &U_{n + 1} = 2U_{n + 1\slash2}^{(p_n)} - U_n,
%      \end{split}
%    \end{equation}
%    or equivalently
%    \begin{equation}\label{Equ: inexactMidpoint IntermediateOp}
%      \begin{split}
%        &U_{n + 1\slash2}^{(0)} = U_n,\\
%        &U_{n + 1\slash2}^{(k)} = \Big( I + \frac{\Delta t_n}{2} \mathcal{A}_{U_{n + 1\slash2}^{(k-1)}} \Big)^{-1}U_{n}, ~k = 1, 2, \ldots, p_n,\\
%        &U_{n + 1} = 2U_{n + 1\slash2}^{(p_n)} - U_n.
%      \end{split}
%    \end{equation}

    \begin{algorithm}\label{Alg: Self-consistent}
    \caption{An orthogonality preserving iteration}
      Given $\varepsilon > 0$, $\tilde\delta_T > 0$, initial data $U_0\in(V_{N_g})\bigcap\mathcal{M}^N$, calculate gradient $\nabla_{G} E(U_0)$, let $n = 0$\;
      \While{$||| \nabla_{G} E(U_n) ||| > \varepsilon$}{
      Set time step size $\Delta t_n \leqslant \tilde\delta_T$ and iteration times $p_n\in\mathbb{N}_+$\;
      $U_{n + 1\slash2}^{(0)} = U_{n}$\;
      \For{$k = 1, \ldots, p_n$}{
      \begin{eqnarray}\label{Equ: inexactMidpoint IntermediateOp}
      U_{n + 1\slash2}^{(k)} = \Big( I + \displaystyle\frac{\Delta t_n}{2} \mathcal{A}_{U_{n + 1\slash2}^{(k - 1)}} \Big)^{-1}U_{n};\quad\quad\quad\quad\quad\quad\quad\quad\quad\quad\quad\quad\quad\quad\quad\quad~
      \end{eqnarray}
      }
      $U_{n + 1} = 2U_{n + 1\slash2}^{(p_n)} - U_n$\;
      Let $n = n + 1$, calculate gradient $\nabla_{G} E(U_n)$\;
      }
    \end{algorithm}
    \begin{rema}
    Although Algorithm \ref{Alg: Self-consistent} involves time step $\Delta t_n$, we can regard the time step as a parameter and then Algorithm \ref{Alg: Self-consistent} becomes a nonlinear operator iteration.
    \end{rema}

 %   The midpoint scheme in finite iterations that we choose is as follows in Algorithm \ref{Alg: Self-consistent}.
    We refer to Theorem \ref{Theo: inexactMainConvergeOp} for the choice of $\tilde\delta_T$ in Algorithm \ref{Alg: Self-consistent}. Due to the low-rank structure in $ \Big( I +  s \mathcal{A}_{U} \Big)^{-1}$, we may apply Sherman-Morrison-Woodbury formula \cite{dai2017conjugate, Str_Pre2} to obtain
    \begin{equation}
    \begin{split}
      \Big( I +  s \mathcal{A}_{U} \Big)^{-1}\tilde U =~& \tilde U + s[\nabla E(U)\quad U]\\
      \cdot\Bigg(I_{2N} + s \bigg[&
      \begin{array}{cc}
            \langle U^\top \nabla E(U) \rangle                     & -\langle U^\top U\rangle\\
            \langle\big(\nabla E(U)\big)^\top \nabla E(U) \rangle  & -\langle U^\top \nabla E(U)\rangle
      \end{array}
      \bigg]\Bigg)^{-1}
      \bigg[
      \begin{array}{c}
        \langle U^\top \tilde U\rangle\\
        \big\langle \big(\nabla E(U)\big)^\top \tilde U\big\rangle
      \end{array}
      \bigg].
    \end{split}
    \end{equation}
    We observe that the computational complexity from $U_n$ to $U_{n + 1}$ of Algorithm \ref{Alg: Self-consistent} is mainly determined by $\langle U^\top U\rangle$, $\langle U^\top \nabla E(U) \rangle$ and $\big\langle\big(\nabla E(U)\big)^\top \nabla E(U) \big\rangle$. If $\nabla E$ is a dense operator, the computational complexity of Algorithm \ref{Alg: Self-consistent} is $\mathcal{O}(NN_g^2)$; otherwise, if $\nabla E$ is sparse, generated by finite element bases for example, the computational complexity can be reduced to $\mathcal{O}(N^2N_g)$.

    Similar to section \ref{Sec: Exact}, we have
    \begin{proposition}
    If $U_n$ is obtained from Algorithm \ref{Alg: Self-consistent}, then $U_n\in(V_{N_g})^N\bigcap\mathcal{M}^N$ for all $n\in\mathbb{N}$.
    \end{proposition}

    By the mathematical induction, we obtain that the auxiliary updating points are inside the Stiefel manifold, too.
    \begin{proposition}\label{Prop: inexactInnerProperty}
        If $U_n$ is obtained from Algorithm \ref{Alg: Self-consistent}, then spectrum\\ $\sigma\big(\langle {U_{n + 1\slash2}^{(k)}}^\top {U_{n + 1\slash2}^{(k)}}\rangle\big)$ of $\langle {U_{n + 1\slash2}^{(k)}}^\top {U_{n + 1\slash2}^{(k)}}\rangle$ satisfies
      \begin{equation}
        \sigma\big(\langle {U_{n + 1\slash2}^{(k)}}^\top {U_{n + 1\slash2}^{(k)}}\rangle \big)\subset[0, 1],
      \end{equation}
      for any $p_n\in\mathbb{N}_+$ and $k = 1, 2, \ldots, p_n$.
    \end{proposition}

    \subsection{Convergence}

    Now we prove the convergence of the orthogonality preserving iteration scheme. First we prove a useful lemma.

    \begin{lemma}\label{Lemma: inexactMidAppOp}
    If~$U_{n + 1\slash2}^{(k)}$~is defined in Algorithm \ref{Alg: Self-consistent} for any~$p\in\mathbb{N}_+$ and
       \begin{equation*}
       \begin{split}
         ||\mathcal{A}_{U_i} - \mathcal{A}_{U_j} || \leqslant \hat L ||| U_i - U_j ||| \qquad \forall U_i, U_j\in B(U_{N + 1\slash2}, \delta_r),
       \end{split}
       \end{equation*}
    then there exists a upper bound $\delta_z$ for $\Delta t_n$ that
    \begin{equation}
    ||| U_{n + 1\slash2}^{(k)} - U_{n + 1\slash2}||| \leqslant C\Delta t_n
    \end{equation}
    and
    \[U_{n + 1\slash2}^{(k)} \in B(U_{n + 1\slash2}, \delta_r)\]
    for all $\Delta t_n \in [0, \delta_z]$ and $k = 1, 2, \ldots, p$, where $U_{n + 1\slash2}$ is the solution of \eqref{Equ: Midpoint IntermediateOp} and $C$ is a constant.
    \end{lemma}

    \begin{proof}
    We prove the lemma by mathematical induction. Set
    \begin{equation}
       \delta_z =
       \left\{
       \begin{array}{l}
         \min\Big\{\frac{2}{\hat L\sqrt{N}}, \frac{2\delta_r}{||| \nabla_G E(U_{n + 1\slash2}) |||}, \delta^*\Big\}, \quad ||| \nabla_G E(U_{n + 1\slash2}) ||| > 0,\\
         \min\Big\{\frac{2}{\hat L\sqrt{N}}, \delta^*\Big\}, \quad ||| \nabla_G E(U_{n + 1\slash2}) ||| = 0.\\
       \end{array}
       \right.
     \end{equation} Clearly, the claim holds when $k = 0$ because
    \begin{equation}
      ||| U_{n + 1\slash2}^{(0)} - U_{n + 1\slash2} ||| = ||| U_n - U_{n + 1\slash2}||| = \frac{||| \nabla_G E(U_{n + 1\slash2}) |||}{2}  \Delta t_n.
    \end{equation}

    Suppose the claim holds for $k - 1$. Since $U_{n + 1\slash2}$ is the solution of \eqref{Equ: Midpoint IntermediateOp}, we have
    \begin{equation*}
    \begin{split}
      &U_{n + 1\slash2}^{(k)} - U_{n + 1\slash2}\\
      =&\Big( I + \frac{\Delta t_n}{2} \mathcal{A}_{U_{n + 1\slash2}^{(k - 1)}} \Big)^{-1}U_{n} - \Big( I + \frac{\Delta t_n}{2} \mathcal{A}_{U_{n + 1\slash2}} \Big)^{-1}U_{n}\\
      =&\frac{\Delta t_n}{2}\Big( I + \frac{\Delta t_n}{2} \mathcal{A}_{U_{n + 1\slash2}^{(k - 1)}} \Big)^{-1} \Big(\mathcal{A}_{U_{n + 1\slash2}} - \mathcal{A}_{U_{n + 1\slash2}^{(k - 1)}}\Big) \Big( I + \frac{\Delta t_n}{2} \mathcal{A}_{U_{n + 1\slash2}} \Big)^{-1}U_{n}.
    \end{split}
    \end{equation*}
    Note that $\mathcal{A}_{U}$ is skew-symmetric. We obtain
    \begin{equation*}\label{Equ: inexactMidpoint Pf2Op}
    \begin{split}
      &\bigg\Vert \Big(I + \frac{\Delta t_n}{2}\mathcal{A}_{U} \Big)^{-1}\bigg\Vert = \bigg\Vert \Big(I - \frac{1}{4}(\Delta t_n \mathcal{A}_{U})^2\Big)^{-1}\bigg\Vert^{\frac12} \leqslant 1, \quad\forall U\in (V_{N_g})^N
    \end{split}
    \end{equation*}
    and hence
    \begin{equation*}
    \begin{split}
      &||| U_{n + 1\slash2}^{(k)} - U_{n + 1\slash2} |||\\
      \leqslant& \frac{\Delta t_n}{2} \bigg\Vert\Big(I + \frac{\Delta t_n}{2} \mathcal{A}_{U_{n + 1\slash2}^{(k - 1)}} \Big)^{-1}\bigg\Vert \bigg\Vert\mathcal{A}_{U_{n + 1\slash2}} - \mathcal{A}_{U_{n + 1\slash2}^{(k - 1)}}\bigg\Vert  \bigg\Vert\Big(I + \frac{\Delta t_n}{2} \mathcal{A}_{U_{n + 1\slash2}}  \Big)^{-1}\bigg\Vert ||| U_n |||\\
      \leqslant&\frac{\Delta t_n\hat L\sqrt{N}}{2} ||| U_{n + 1\slash2}^{(k - 1)} - U_{n + 1\slash2} ||| \leqslant \frac{\hat L\sqrt{N}\delta_r}{2} \Delta t_n.
    \end{split}
    \end{equation*}
    By $\Delta t_n \in [0, \delta_z]$, we see $U_{n + 1\slash2}^{(k)} \in B(U_{n + 1\slash2}, \delta_r)$ and the claim also holds for $k$. Thus we confirm that $U_{n + 1\slash2}^{(k)} \in B(U_{n + 1\slash2}, \delta_r)$ and
    \begin{equation*}
     ||| U_{n + 1\slash2}^{(k)} - U_{n + 1\slash2} ||| \leqslant C\Delta t_n
    \end{equation*}
    for all $k = 0, 1, \ldots, p$, where $C$ is some constant.
    \end{proof}

    Similar to the  midpoint scheme in Lemma \ref{Lemma: Energy decreaseOp}, we prove the local energy descending property for Algorithm \ref{Alg: Self-consistent}. We introduce a mapping $h_p$ from $(U, s)\in(V_{N_g})^N\times\mathbb{R}$ to $h_p(U, s)\in(V_{N_g})^N$ as follows:
    \begin{equation}
    h_p(U, s) = 2\bar U^{(p)} - U,
    \end{equation}
    where $\bar U^{(p)}$ is recursively defined by
    \begin{equation}
      \begin{split}
        \bar U^{(k)} &= \Big( I + \frac{s}{2} \mathcal{A}_{\bar U^{(k - 1)}} \Big)^{-1} U,~k = p, p - 1, \ldots, 1,\\
        \bar U^{(0)} &= U.
      \end{split}
    \end{equation}
    In this section, we always assume that $\nabla E$ is local Lipschitz continuous in the neighborhood of a local minimizer $U^*\in(V_{N_g})\bigcap\mathcal{M}^N$:
    \begin{equation}\label{Equ: inexactLocalLipschitzOp}
      \begin{split}
        ||| \nabla E(U_i) - \nabla E(U_j))||| \leqslant \tilde L ||| U_i - U_j |||, \forall U_i, U_j\in B(U^*, \tilde\delta_L),
      \end{split}
    \end{equation}
    where $\tilde\delta_L > \max\{\delta_a, \delta_b\}$.
    \begin{lemma}\label{Lemma: inexactEnergyDecreaseOp}
    There holds
       \begin{equation}\label{Equ: Local Lipschitz2Op}
       \begin{split}
         ||| \nabla_G E(U_i) - \nabla_G E(U_j)||| \leqslant \tilde L_1 ||| U_i - U_j|||, \quad\Vert \mathcal{A}_{U_i} - \mathcal{A}_{U_j}\Vert \leqslant \tilde L_1 ||| U_i - U_j|||, \\
         \forall U_i, U_j\in B(U^*, \tilde\delta_L).
       \end{split}
       \end{equation}
    Moreover, there exists a upper bound $\tilde\delta_s$ for $s$ that
        \begin{equation}
        \begin{split}
          E(U) - E\big(h_p(U, s)\big) \geqslant \frac{s}{4N}\Big|\Big|\Big| \nabla_G E\Big(\frac{h_p(U,s) + E(U)}{2}\Big)\Big|\Big|\Big|^2,\\
          \forall U \in B(U^*, \delta_a)\bigcap\mathcal{M}^N, \forall s \in[0, \tilde\delta_s].
        \end{split}
        \end{equation}
    Meanwhile $h_p(U, s)\in B(U^*, \tilde\delta_L)$.
    \end{lemma}

    \begin{proof}
    We only need to prove that $h_p(U, s)\in B(U^*, \tilde\delta_L)$. Set
    \begin{equation*}
        S(t) = tg(U, s) + (1 - t)U,~t\in[0, 1].
    \end{equation*}
    We obtain from Lemma \ref{Lemma: Energy decreaseOp} that for $U\in B(U^*, \delta_a)$ and $s\in[0,\delta_s]$, there holds
    \begin{equation*}
    \begin{split}
      \Big|\Big|\Big| S\big(\frac12\big) - U^* \Big|\Big|\Big| \leqslant& \frac12 ||| g(U, s) - U^*||| + \frac12||| U - U^*||| \leqslant \frac12\delta_b + \frac12\delta_a,
    \end{split}
    \end{equation*}
    which implies
    \[
    S\big(\frac12\big)\in B(U^*, \frac12\delta_b + \frac12\delta_a)\subset B(U^*, \tilde\delta_L).\]
    Note that Lemma \ref{Lemma: inexactMidAppOp} implies
    \[\bar U^{(p)}(s)\in B\Big(S\big(\frac12\big), \frac12(\tilde\delta_L - \delta_b) \Big) \subset B(U^*, \tilde\delta_L)\]
    provided
    \begin{equation*}
      s\in
      \left\{
      \begin{array}{l}
        \bigg[0, \min\Big\{\frac{2}{\hat L_1\sqrt{N}}, \frac{\tilde\delta_L - \delta_b}{||| \nabla_G E(S(\frac12)) |||}, \delta^*\Big\}\bigg], \quad \big|\big|\big| \nabla_G E\big(S\big(\frac12\big)\big) \big|\big|\big| > 0,\\[0.3cm]
        \bigg[0, \min\Big\{\frac{2}{\tilde L_1\sqrt{N}}, \delta^*\Big\}\bigg], \quad \big|\big|\big| \nabla_G E\big(S\big(\frac12\big)\big) \big|\big|\big| = 0.
      \end{array}
      \right.
    \end{equation*}
    While there holds
    \begin{equation}
        \big|\big|\big| \nabla_G E\big(S\big(\frac12\big)\big) \big|\big|\big| =  \big|\big|\big| \nabla_G E\big(S\big(\frac12\big)\big) - \nabla_G E\big(U^*\big)\big|\big|\big| \leqslant \tilde L_1||| S\big(\frac12\big) - U^* ||| \leqslant \frac12\tilde L_1(\delta_a + \delta_b),
    \end{equation}
    we have
    \[\bar U^{(p)}(s)\in B\Big(S\big(\frac12\big), \frac12(\tilde\delta_L - \delta_b) \Big)\]
    as long as
    \begin{equation*}
      s\in \bigg[0, \min\Big\{\frac{2}{\tilde L_1\sqrt{N}}, \frac{2(\tilde\delta_L - \delta_b)}{\tilde L_1(\delta_a + \delta_b)}, \delta^*\Big\}\bigg].
    \end{equation*}
    Therefore, we get
    \begin{equation*}
    \begin{split}
      &||| h_p(U, s) - U^*||| \leqslant ||| h_p(U, s) - g(U, s)||| + ||| g(U, s) - U^*|||\\
    \leqslant& 2\big|\big|\big| \bar U^{(p)} - S\big(\frac12\big)\big|\big|\big| + \delta_b \leqslant \tilde\delta_L.
    \end{split}
    \end{equation*}

    Similarly, we have
    \begin{equation}
      \begin{split}
        &E(U) - E(h_p(U, s)) \geqslant s\Big(\frac{1}{2N} - \frac{s\tilde L}{2}\Big)\Big|\Big|\Big| \nabla_G E\Big(\frac{h_p(U,s) + E(U)}{2}\Big)\Big|\Big|\Big|^2.
      \end{split}
    \end{equation}
    All the above results hold when $s\in[0,\tilde\delta_s]$, where
    \begin{equation*}
    \tilde\delta_s = \min\Big\{\frac{2}{\tilde L_1\sqrt{N}}, \frac{2(\tilde\delta_L - \delta_b)}{\tilde L_1(\delta_a + \delta_b)}, \frac1{2N\tilde L},\delta^*\Big\}.
    \end{equation*}

    \end{proof}

    We can define a mapping
    \begin{equation*}
          \hat h_p: B\big([U^*], \delta_a\big)\times[0, \tilde\delta_s] \rightarrow B\big([U^*], \tilde\delta_L\big)
    \end{equation*}
    such that
    \begin{equation*}
      \hat h_p\big([U], s\big) = \Big[h_p\big(\arg\min\limits_{\tilde U\in [U]} ||| \tilde U - U^* |||, s\big)\Big].
    \end{equation*}

    \begin{lemma}\label{Lemma: inexactSelf-mappingOp}
      There holds
      \begin{equation*}
          \hat h_p\Big(B\big([U^*], \tilde\delta_e\big)\bigcap\mathcal{L}_{\tilde E_e}\times[0, \tilde\delta_T] \Big)\subset B([U^*], \tilde\delta_e)\bigcap\mathcal{L}_{\tilde E_e}
      \end{equation*}
      for some $\tilde\delta_e > 0$, $\tilde E_e\in\mathbb{R}$, $\tilde\delta_T \in [0, \tilde\delta_s]$ where $\tilde\delta_s$ is defined in Lemma \ref{Lemma: inexactEnergyDecreaseOp}.
    \end{lemma}

    Then comparing with the midpoint scheme case in Theorem \ref{Theo: MainConvergeOp}, we arrive at the following convergence result. Since the proof is similar, we omit the details.

    \begin{theorem}\label{Theo: inexactMainConvergeOp}
       If $[U_{0}]\in B\big([U^*], \tilde\delta_e\big)$ and $\sup\{\Delta t_n : n \in \mathbb{N}\} \leqslant \tilde\delta_T$, then for any $p_n\in\mathbb{N}_+$, the sequence $\{U_n\}$ produced by Algorithm \ref{Alg: Self-consistent} satisfies
      \begin{eqnarray*}
        \lim\limits_{n\rightarrow\infty} ||| \nabla_{G} E(U_{n}) ||| = 0,\\
        \lim\limits_{n\rightarrow\infty} E(U_{n}) = E(U^*),\\
        \lim\limits_{n\rightarrow\infty} \textup{dist}\big([U_{n}], [U^*]\big) = 0,
      \end{eqnarray*}
      where $\tilde\delta_e$, $\tilde\delta_T$ are defined in Lemma \ref{Lemma: inexactSelf-mappingOp}.
    \end{theorem}

    Finally, we turn to the convergence rate of the approximations produced by Algorithm \ref{Alg: Self-consistent}.
     \begin{lemma}\label{Lemma: inexactHessAppr}
      For $U\in B(U^*, \min\{\delta_3, \delta_a\})\bigcap\mathcal{M}^N$ and $\tau\in(0, \tilde\delta_T]$, set
      \begin{equation}
        \begin{split}
          U_+ = \Big(I + \frac{\tau}{2}\mathcal{A}_{U_+}\Big)^{-1}U,\\
          U_- = \Big(I - \frac{\tau}{2}\mathcal{A}_{U_-}\Big)^{-1}U.\\
        \end{split}
      \end{equation}
      If \eqref{Equ: Hess} holds true, then there exists some $\tilde\delta_{r_1} > 0$ such that
      \begin{equation}
      \begin{split}
        \textup{tr}\Big( \big\langle \big(U_{+} - U_{-}\big)^\top \big(\nabla_G E(U_{+}) - \nabla_G E(U_{-})\big) \big\rangle\Big) \geqslant \frac{\sigma}{2} ||| U_{+} - U_{-}|||^2
      \end{split}
      \end{equation}
      for all $\tau\in(0,\tilde\delta_{r_1}]$ and $U\in B(U^*, \min\{\delta_3, \delta_a\})\bigcap\mathcal{M}^N$, where $\tilde\delta_{r_1}\in(0,\tilde\delta_T]$ is a positive constant, $\tilde\delta_T$ is defined in Theorem \ref{Theo: inexactMainConvergeOp} and $\delta_a$ is defined in Lemma \ref{Lemma: Implicit Function TheoryOp}.
    \end{lemma}

    \begin{lemma}\label{Lemma: inexactGradientConvEst}
    For $U\in B(U^*, \tilde\delta_L)\bigcap\mathcal{M}^N$ and $\tau\in(0, \tilde\delta_T]$, if
      \begin{equation}
        \begin{split}
          U_+ &= \Big(I + \frac{\tau}{2}\mathcal{A}_{U_+}\Big)^{-1}U,\\
          \bar{U}_{+} &= 2U_+ - U,\\
        \end{split}
      \end{equation}
      then there exists some $\tilde\delta_{r_2} > 0$ that satisfies
      \begin{equation}
        E(U) - E(\bar{U}_{+}) \leqslant \frac{\tau(\tilde L + 3)}{2} ||| \nabla_G E(U_{+})|||^2,
      \end{equation}
      for all $\tau\in(0,\tilde\delta_{r_2}]$ and $U\in B(U^*, \tilde\delta_L\})\bigcap\mathcal{M}^N$, where $\tilde\delta_{r_2}\in(0,\tilde\delta_T]$ is a positive constant, $\tilde\delta_T$ is defined in Theorem \ref{Theo: inexactMainConvergeOp} and $\tilde\delta_L$ and $\tilde L$ are defined in Lemma \ref{Lemma: inexactEnergyDecreaseOp}.
    \end{lemma}

    \begin{theorem}\label{Theo: inexactConvRate}
      Suppose Hessian coercivity holds true as \eqref{Equ: Hess}. If $[U_{0}]\in B\big([U^*], \tilde\delta_e\big)$ and $\Delta t_n = \tau \leqslant \tilde\delta_{r_1}, \forall n\geqslant \tilde N_0$, then the sequence $\{U_n\}$ produced by Algorithm \ref{Alg: Self-consistent} satisfies
      \begin{equation}
      \begin{split}
        |||\nabla_G E(U_{n})||| \leqslant \Big(1 + \frac{\tilde L_1\tau}{2}\Big)\Big(\frac{4 + \tau^2\tilde L_1^2 - 2\sigma \tau}{4 + \tau^2\tilde L_1^2 + 2\sigma \tau}\Big)^{(n - \tilde N_0 + 1)\slash 2} |||\nabla_G E(U_{\tilde N_0 - 1\slash2})||| , \\
        \forall n \geqslant \tilde N_0,
      \end{split}
      \end{equation}
      where $\tilde N_0$ is a positive integer, $\tilde \delta_e$ and $\tilde \delta_T$ are defined in Lemma \ref{Lemma: inexactSelf-mappingOp}, $\tilde L_1$ is defined in \eqref{Equ: Local Lipschitz2Op} and $\tilde \delta_{r_1}$ is defined in Lemma \ref{Lemma: inexactHessAppr}.

      Moreover, if $\Delta t_n = \tau \leqslant \min\{\tilde \delta_{r_1}, \tilde \delta_{r_2}\}$, $\forall n\geqslant \tilde N_1$, then
      \begin{equation}
      \begin{split}
      &E(U_n) - E(U^*)\\
      \leqslant& \frac{(\tilde L + 3)(4 + \tau^2\tilde L_1^2 + 2\sigma \tau)}{8\sigma} |||\nabla_G E(U_{N_1 - 1\slash2})|||^2 \Big(\frac{4 + \tau^2\tilde L_1^2 - 2\sigma \tau}{4 + \tau^2\tilde L_1^2 + 2\sigma \tau}\Big)^{n - \tilde N_1 + 1}, \forall n \geqslant \tilde N_1,
      \end{split}
      \end{equation}
      where $\tilde N_1\geqslant \tilde N_0$ is a positive integer, $\tilde L$ is defined in \eqref{Equ: inexactLocalLipschitzOp} and $\tilde \delta_{r_2}$ is defined in Lemma \ref{Lemma: inexactGradientConvEst}.
    \end{theorem}

    \begin{rema}
    Similarly, Algorithm \ref{Alg: Self-consistent} reaches the optimal convergence rate when
    \begin{equation}
      \tau = \left\{
      \begin{array}{ll}
        \min\{\delta_{r_1}, \delta_{r_2}\},   &\quad \displaystyle\frac{2}{\tilde L_1} > \min\{\delta_{r_1}, \delta_{r_2}\},\\
        \displaystyle\frac{2}{\tilde L_1},    &\quad \displaystyle\frac{2}{\tilde L_1} \leqslant \min\{\delta_{r_1}, \delta_{r_2}\}.
      \end{array}
      \right.
    \end{equation}
    Furthermore, if $U_{k + 1\slash2} \neq U_{k - 1\slash2}$ for some $k \geqslant \tilde N_0$, then we have $\tilde L_1\geqslant{\sigma}\slash{2}$ and convergence rate of the approximations produced by Algorithm \ref{Alg: Self-consistent} can approach $0$ given proper assumptions in theory.
    \end{rema}

    Compared with Algorithm \ref{Alg: Midpoint}, Algorithm \ref{Alg: Self-consistent} is computable. In particular, Algorithm \ref{Alg: Self-consistent} does not require a large band gap and  Theorem \ref{Theo: inexactConvRate} tells the convergence rate of the orthogonality preserving iterations.
    %As a practical algorithm, the step size of our algorithm does not have a lower bound in each updating iteration while optimization methods with backtracking technique does, which indicates that our algorithm is applied for some self-adaptive techniques on the step size better.

    \section{Numerical experiments}
    Our code of the orthogonality preserving iterations of the gradient flow based model is developed based on by PHG toolbox\cite{PHG}. We adopt quadratic finite elements in the spacial discretization. For the exchange-correlation potential, we choose the local density approximation(LDA) in \cite{LDA}:
    \begin{equation}
      v_{xc}(\rho) = \varepsilon_{xc}(\rho) + \rho\frac{\delta\varepsilon_{xc}(\rho)}{\delta\rho},
    \end{equation}
    where $\varepsilon_{xc}(\rho) = \varepsilon_x(\rho) + \varepsilon_c(\rho)$ with
    \begin{equation}
      \varepsilon_x(\rho) = -\frac34 (\frac{3}{\pi})^{1\slash3}\rho^{1\slash3}
    \end{equation}
    and
    \begin{equation}
      \varepsilon_c(\rho) =
      \left\{
      \begin{array}{l}
        -0.1423/(1 + 1.0529\sqrt{r_s} + 0.3334r_s) \:\textup{if}\: r_s \geqslant 1,\\
        0.0311\ln r_s -0.048 + 0.0020 r_s \ln r_s -0.0116 r_s \:\textup{if}\: r_s < 1,
      \end{array}
      \right.
    \end{equation}
    here $r_s = \big({3}\slash(4\pi\rho)\big)^{1\slash 3}$. We see from Theorem \ref{Theo: inexactMainConvergeOp} that the approximations produced by Algorithm \ref{Alg: Self-consistent} is convergent given proper $\tilde\delta_t$ and $\tilde\delta_T$. In implementation of Algorithm \ref{Alg: Self-consistent}, we apply some self-adapted time step sizes and some acceleration techniques.

    We give four examples whose molecular structures can be found in Figure \ref{Fig: Structure}.

    \begin{figure}
    \centering
      \includegraphics[width = 0.49\textwidth]{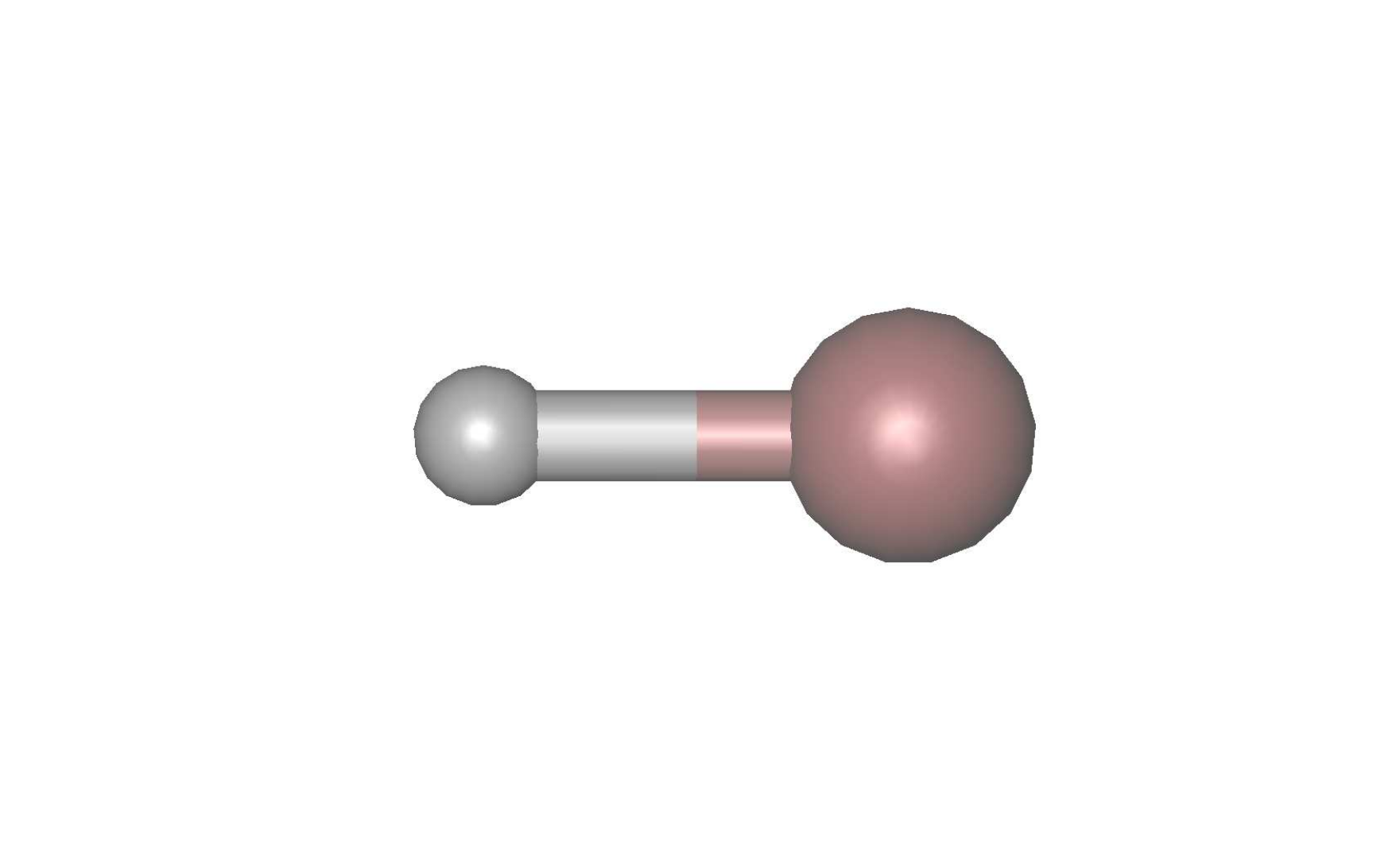}
      \includegraphics[width = 0.49\textwidth]{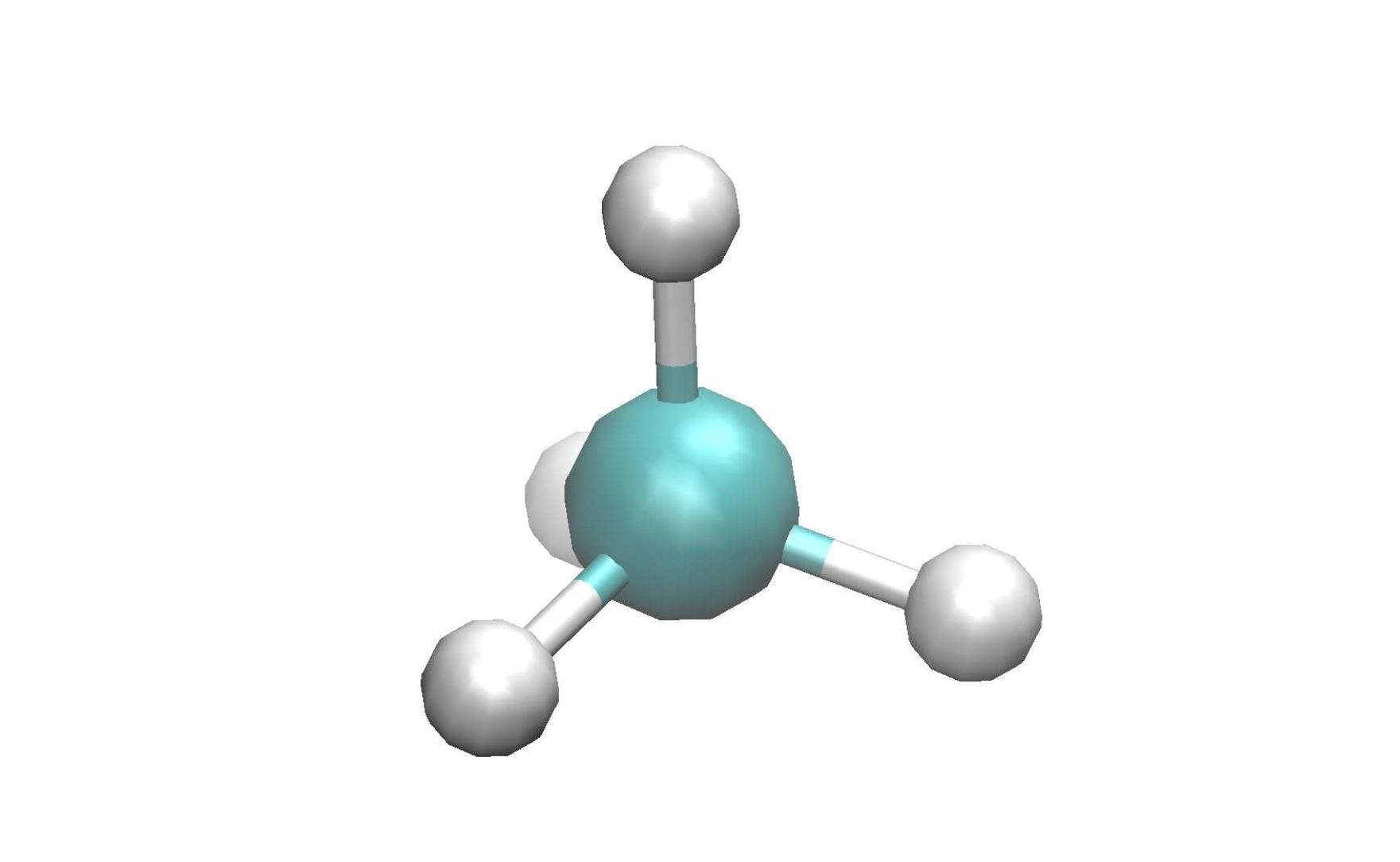}
      \includegraphics[width = 0.49\textwidth]{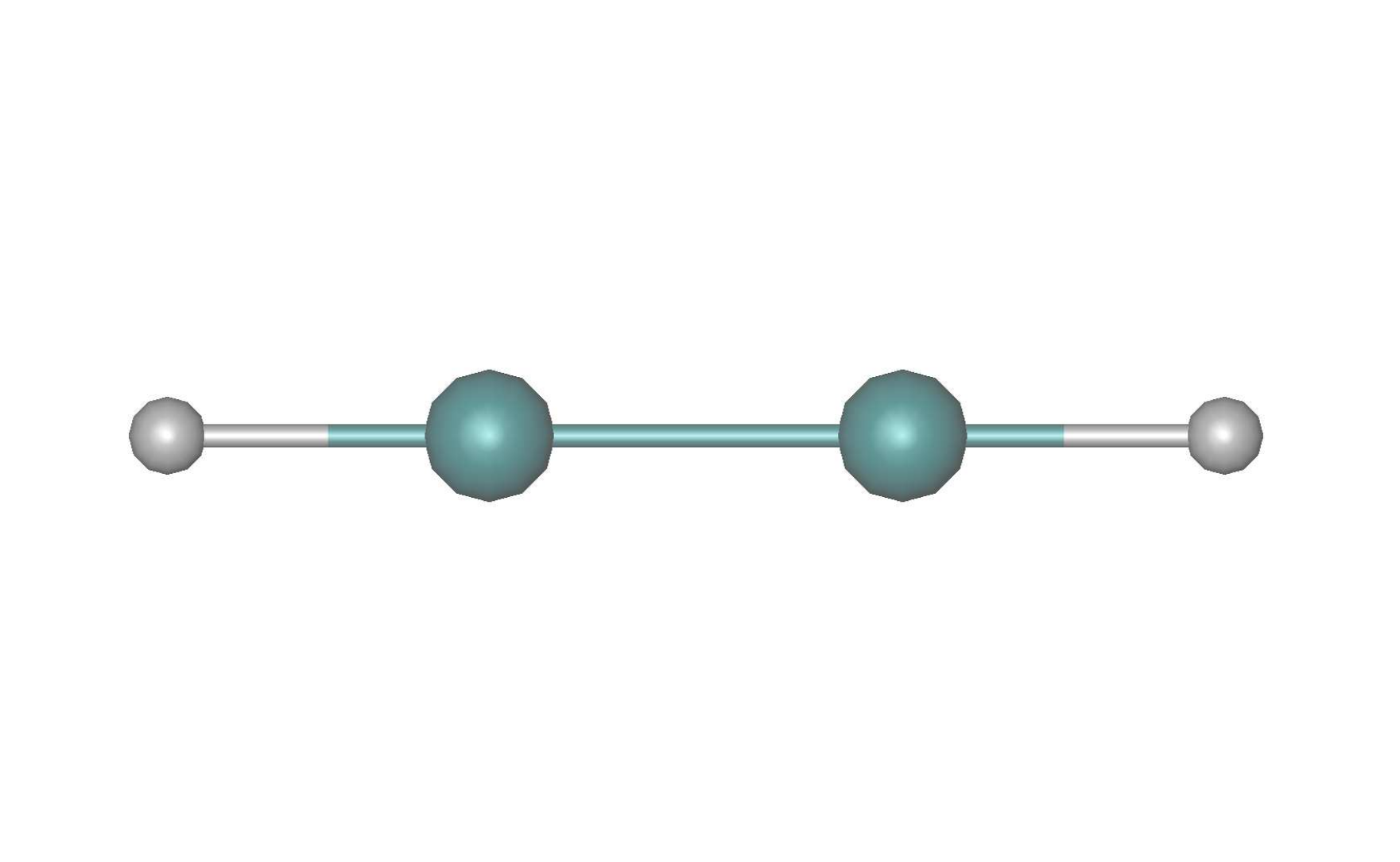}
      \includegraphics[width = 0.49\textwidth]{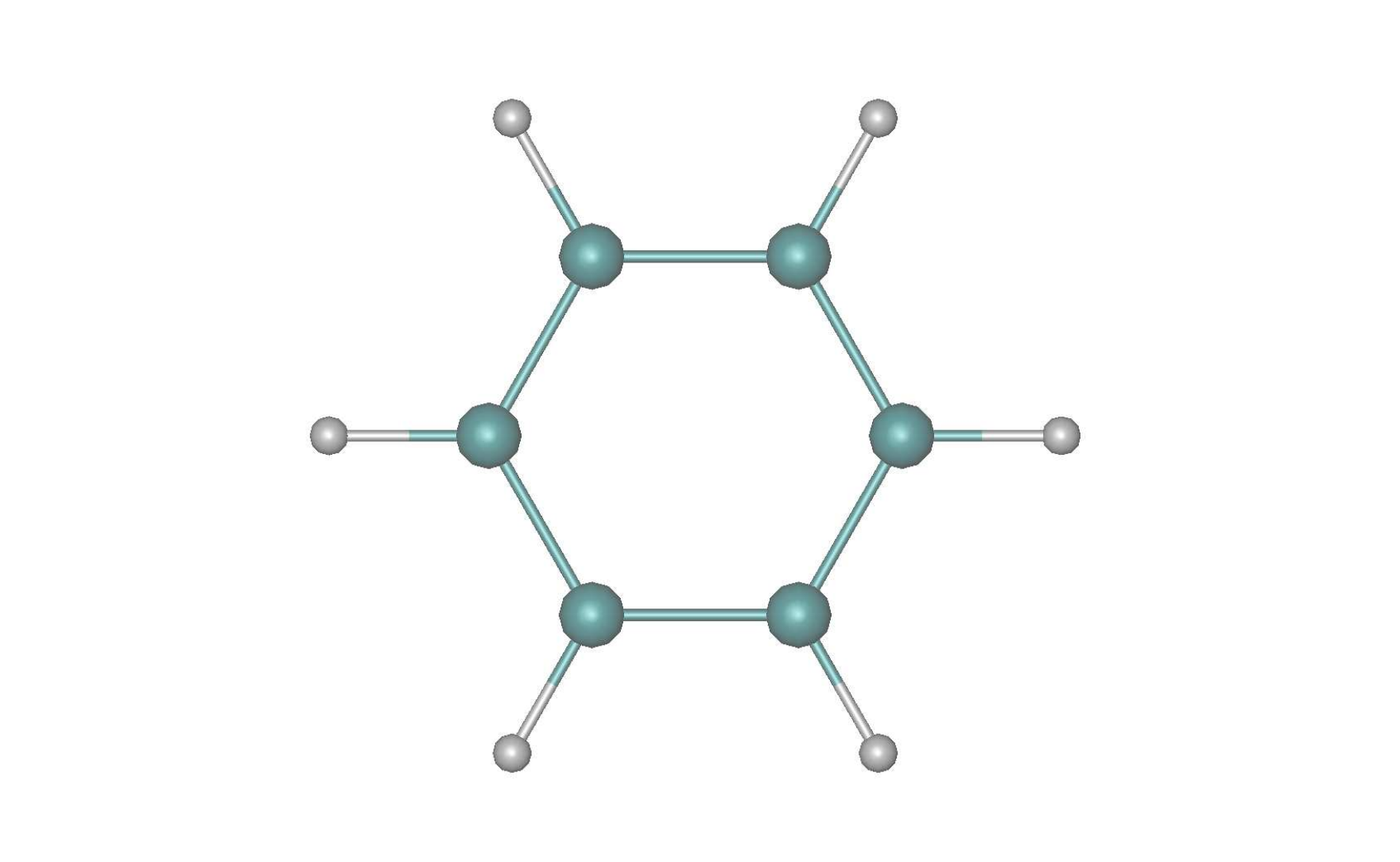}
      \caption{Molecular structures(Row 1 Column 1: LiH; Row 1 Column 2: CH\textsubscript{4}; Row 2 Column 1: C\textsubscript{2}H\textsubscript{2}; Row 2 Column 2: C\textsubscript{6}H\textsubscript{6}).}\label{Fig: Structure}
    \end{figure}

    \emph{Example} 1. Consider the gradient flow model for lithium hydride(LiH) with orbits number $N = 2$ on a fixed tetrahedral finite element mesh over $[-56, 55]\times [-54, 53]\times[-54, 53] \subset\mathbb{R}^3$ from an adaptive refinement finite element method\cite{AFEM} with degrees of freedom $N_g = 10971$(see Figure \ref{Fig: LiHPlus}). We see from Figure \ref{Fig: LiHPlus} that the approximations of electron density between the two nuclei converge. Figure \ref{Fig: LiH} shows the energy and the gradient convergence curve. We see that the energy approximations converge monotonically and the approximations of the gradient oscillate to zero.

    Moreover, the approximated energy of the ground state of LiH we obtain is \\$-7.990787295248$ a.u., which closes to the experimental value $-8.0705$ a.u. in \cite{LiH1} and also consistent with the numerical result $-8.044572$ a.u. \cite{dai2008three} and other numerical results in \cite{LiH1, LiH2}. The minor ground state energy difference results from spacial discretization, boundary condition approximation and precision of LDA model of exchange-correlation term.

    \begin{figure}
    \centering
      \includegraphics[width = 0.49\textwidth]{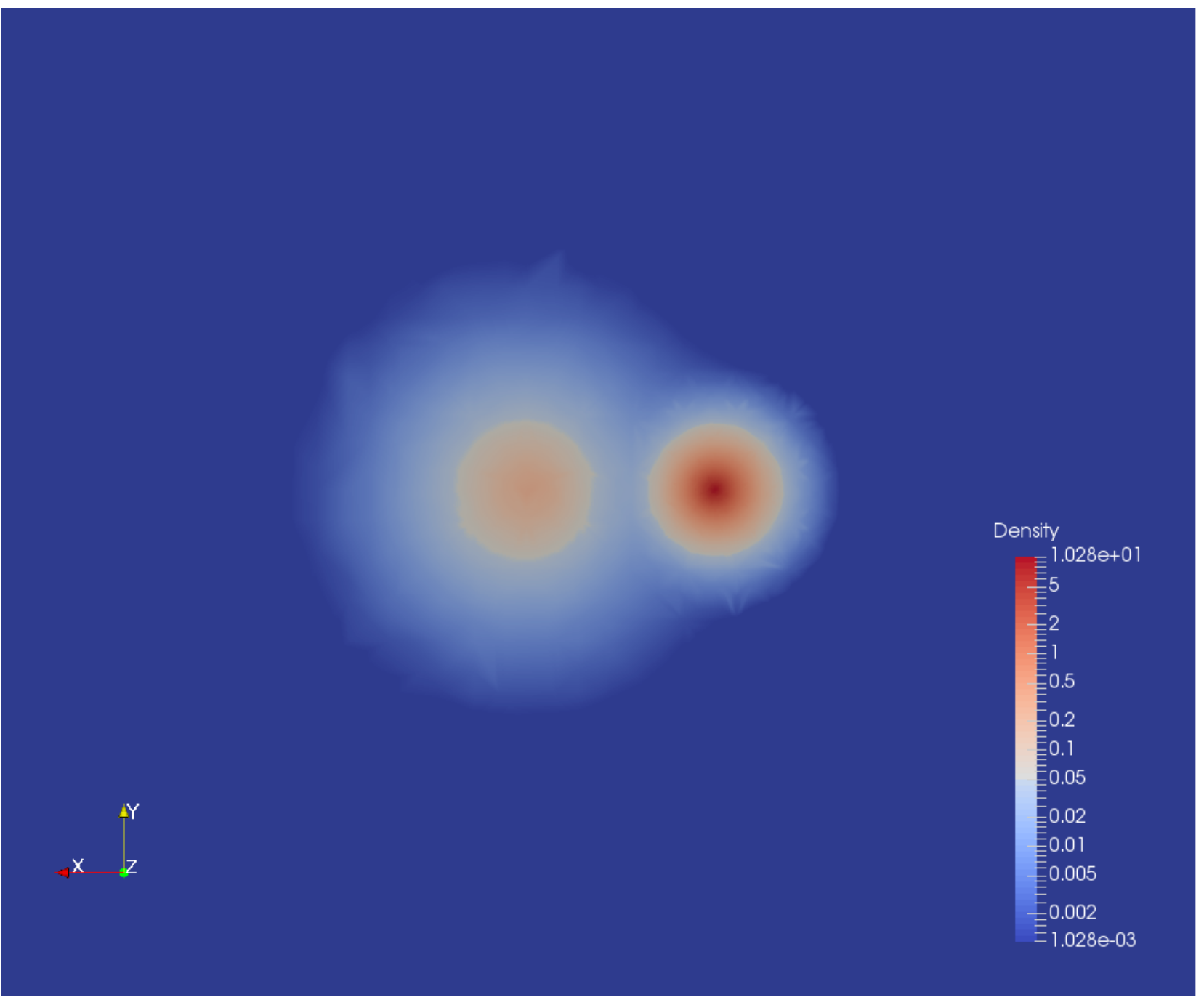}
      \includegraphics[width = 0.49\textwidth]{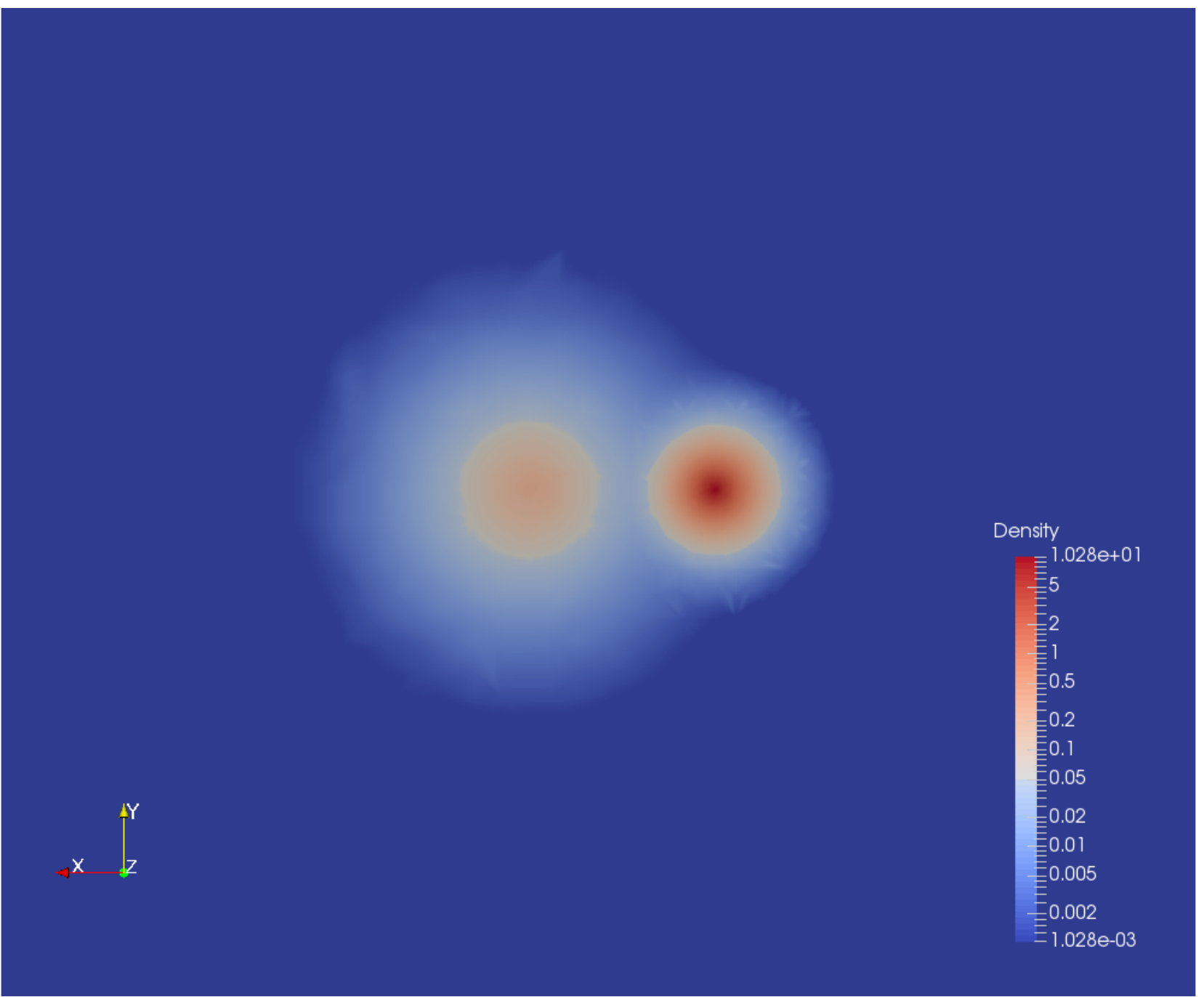}
      \includegraphics[width = 0.49\textwidth]{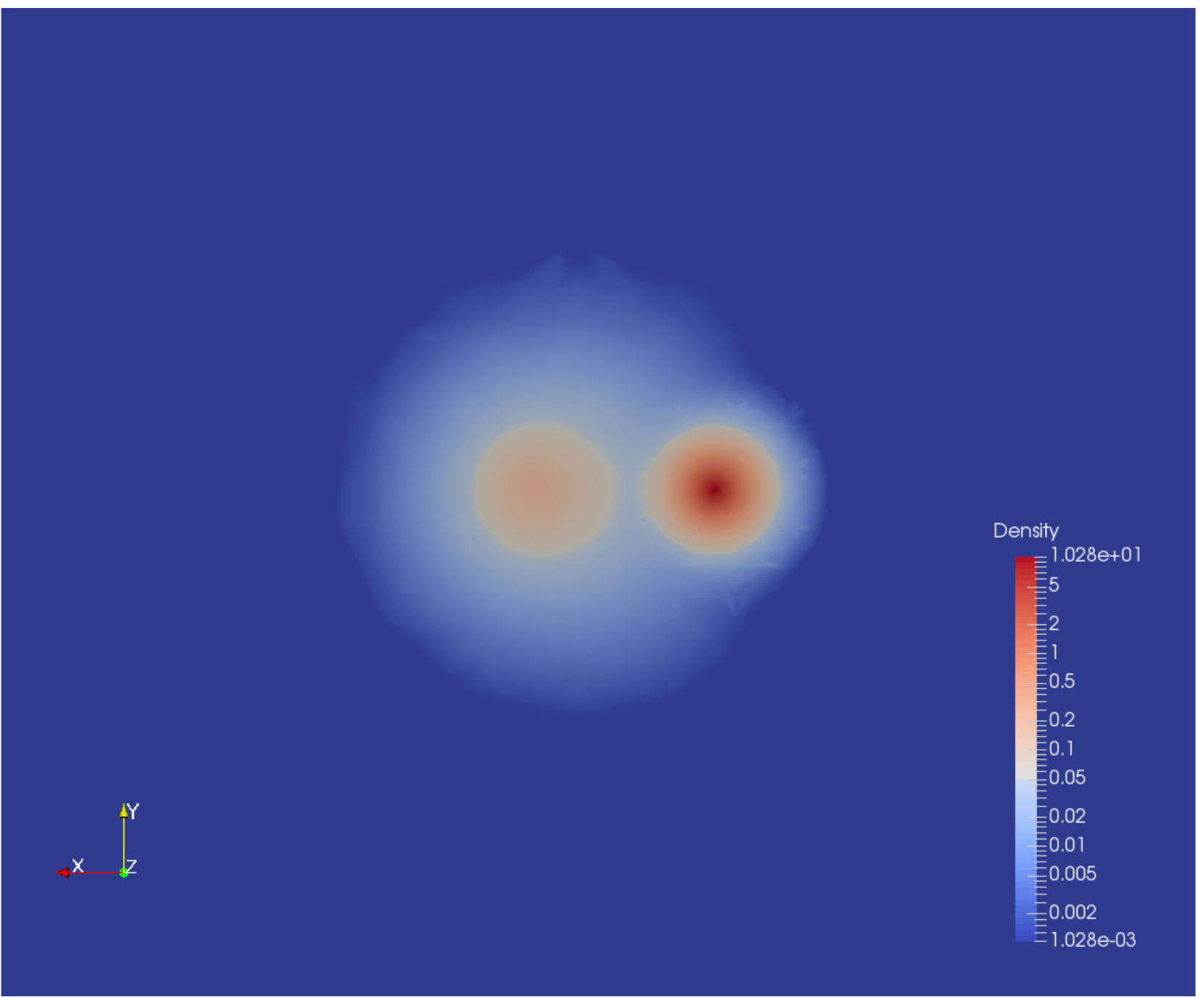}
      \includegraphics[width = 0.49\textwidth]{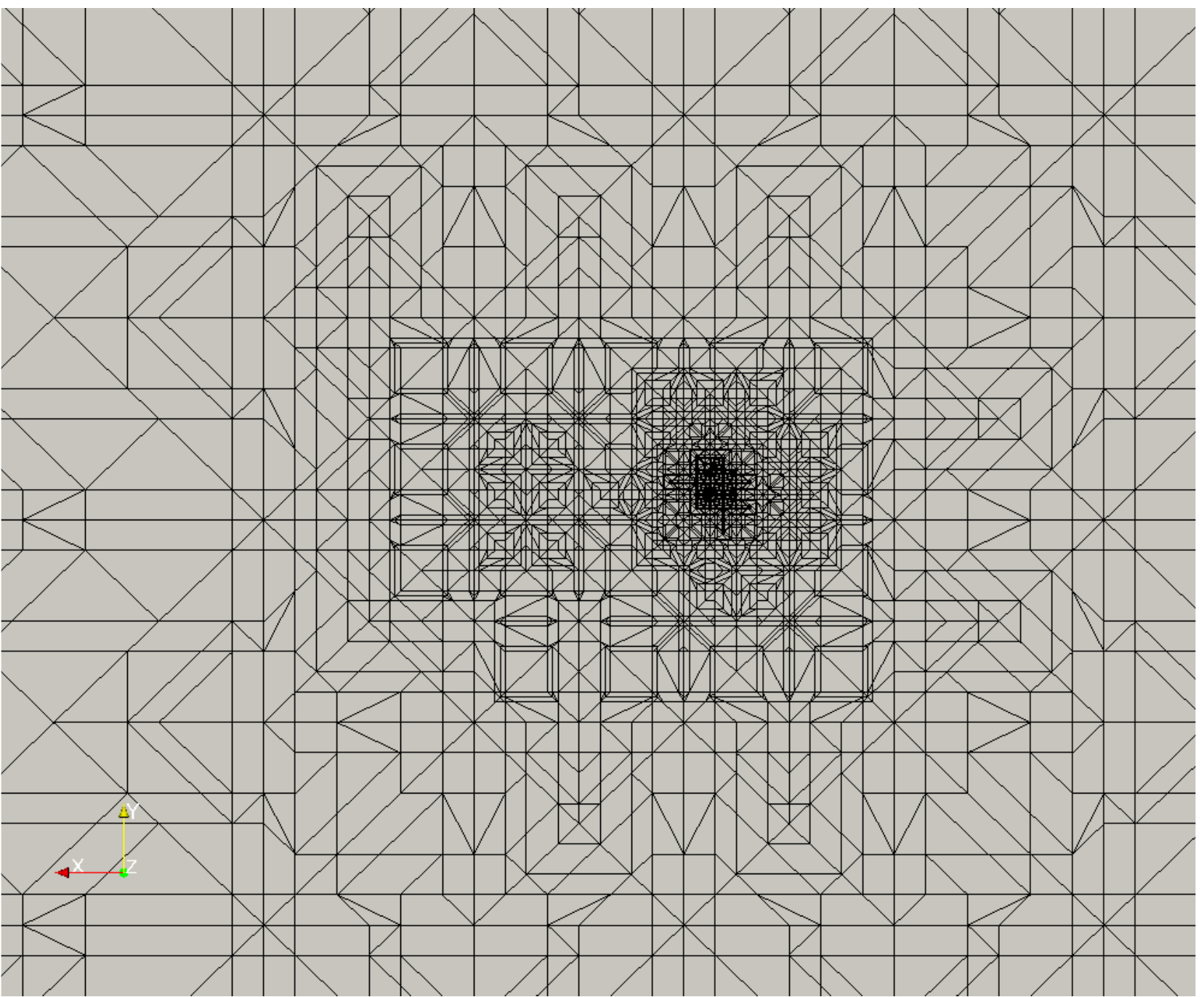}
      \caption{LiH: Output density(Row 1 Column 1: 0th iteration; Row 1 Column 2: 40th iteration; Row 2 Column 1: 4420th iteration;) and input grid(Row 2 Column 2).}\label{Fig: LiHPlus}
    \end{figure}

    \begin{figure}
    \centering
      \includegraphics[width = 0.49\textwidth]{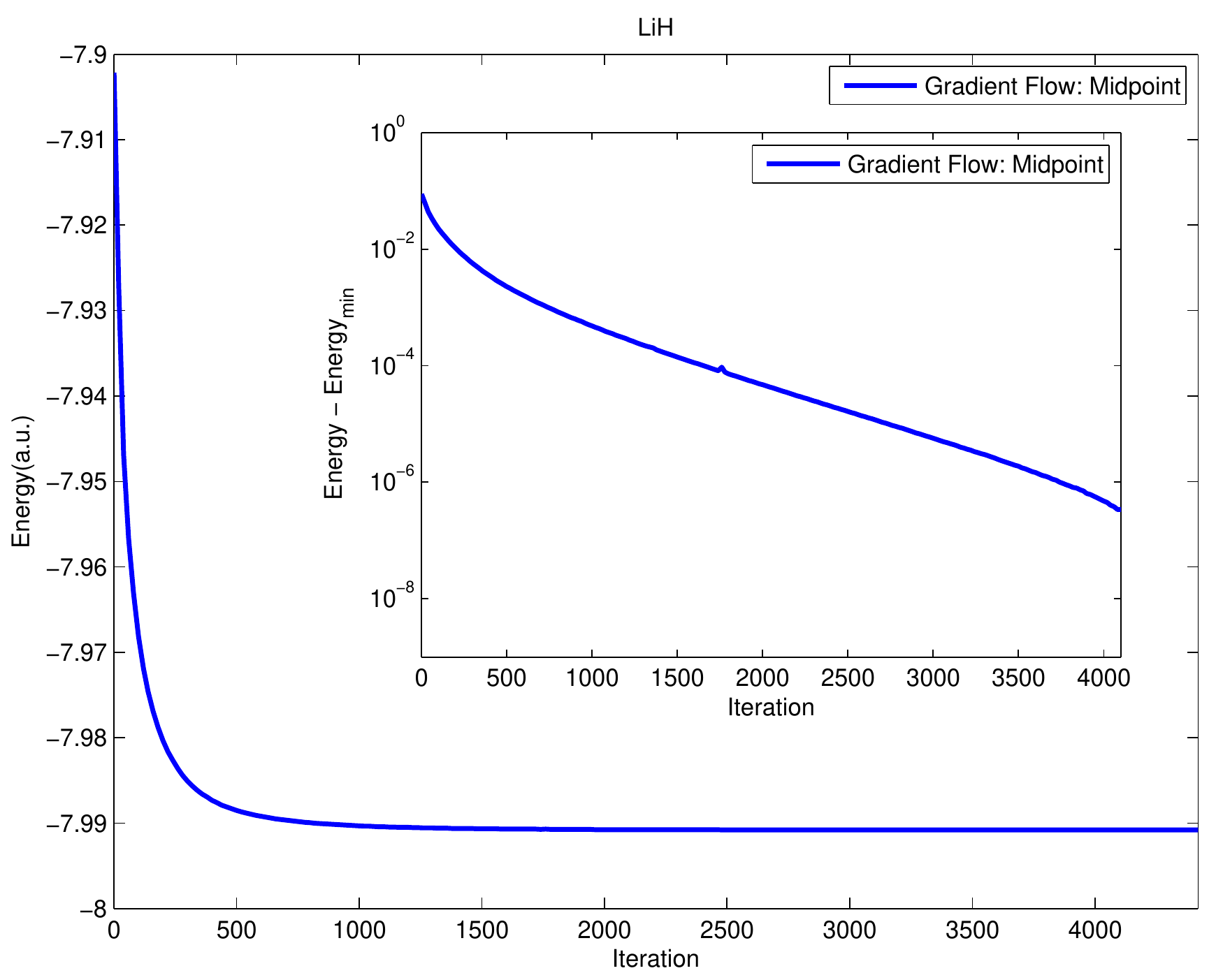}
      \includegraphics[width = 0.49\textwidth]{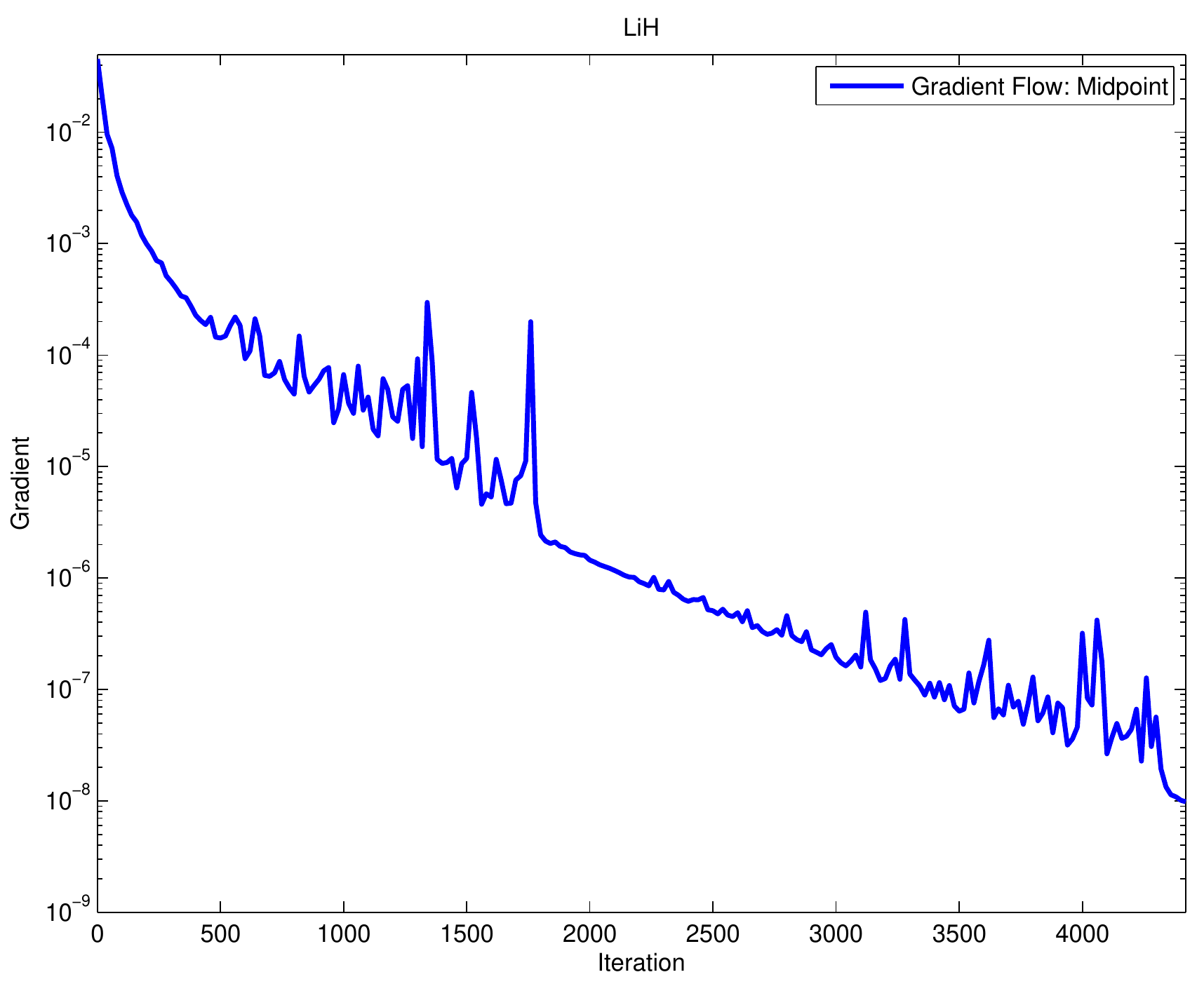}
      \caption{Convergence curves for energy(left) and gradient(right) for LiH.}\label{Fig: LiH}
    \end{figure}

    \emph{Example} 2. For methane(CH\textsubscript{4}) whose orbits number $N = 5$, we compute the gradient flow model on a fixed tetrahedral finite element mesh on $[-56, 55]\times [-54, 53]\times[-54, 53]\subset\mathbb{R}^3$ from an adaptive refinement finite element method\cite{AFEM} with degrees of freedom $N_g = 17267$(see Figure \ref{Fig: CH4Plus}). We see from Figure \ref{Fig: CH4Plus} that the approximations of electron density converge to a regular tetrahedron shape. We learn form Figure \ref{Fig: CH4} that both the approximations of energy and the gradient converge well.

    \begin{figure}
    \centering
      \includegraphics[width = 0.49\textwidth]{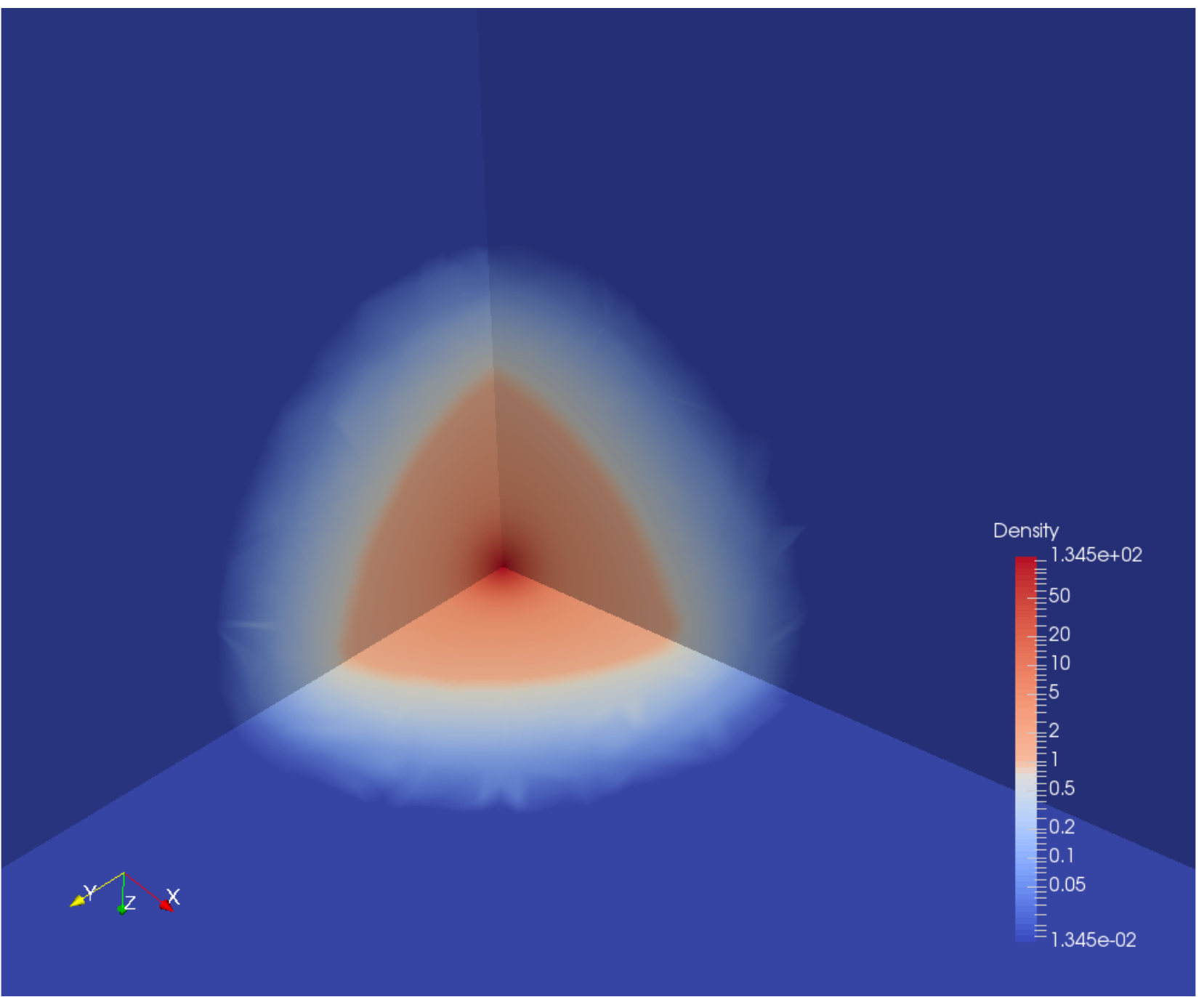}
      \includegraphics[width = 0.49\textwidth]{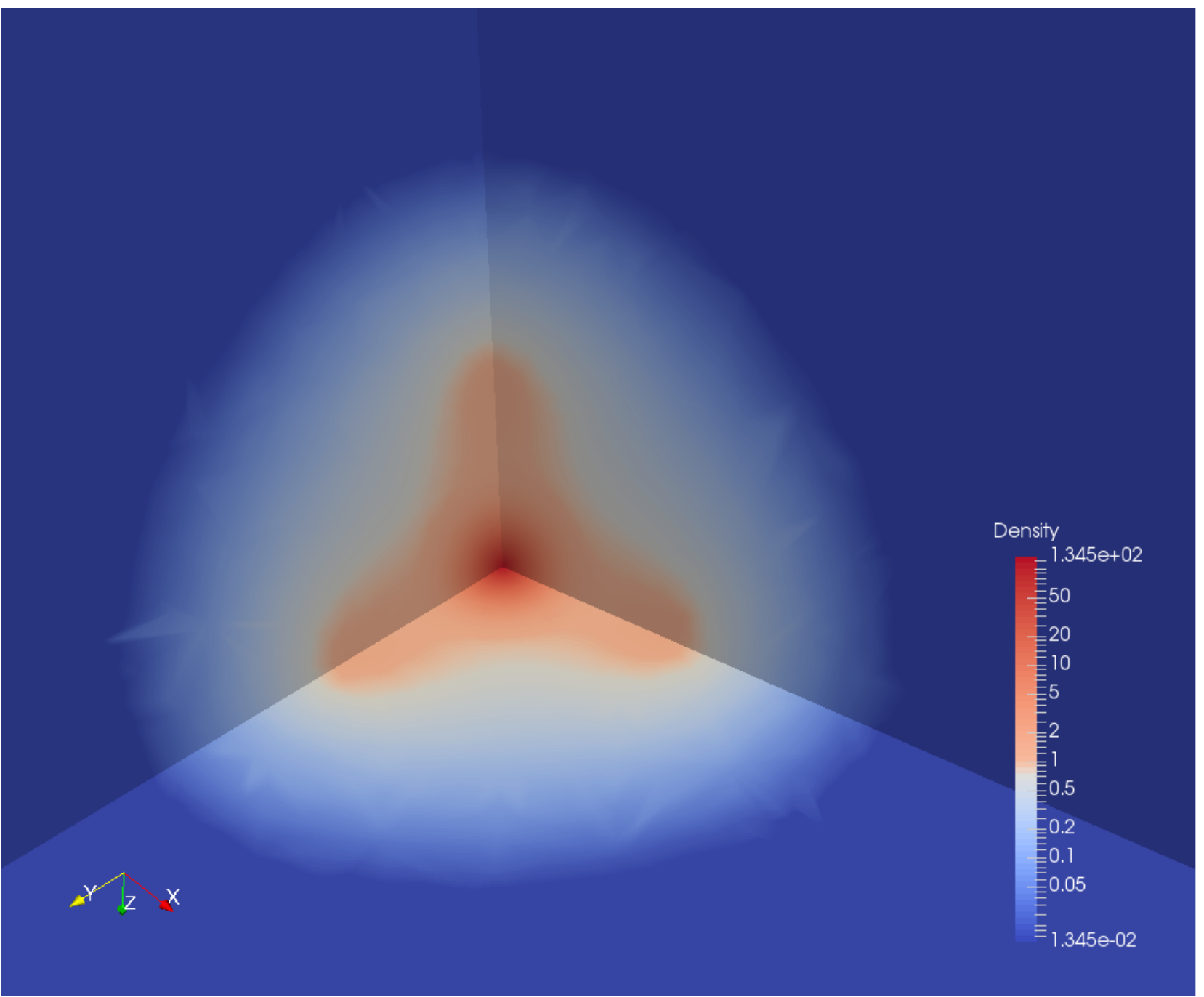}
      \includegraphics[width = 0.49\textwidth]{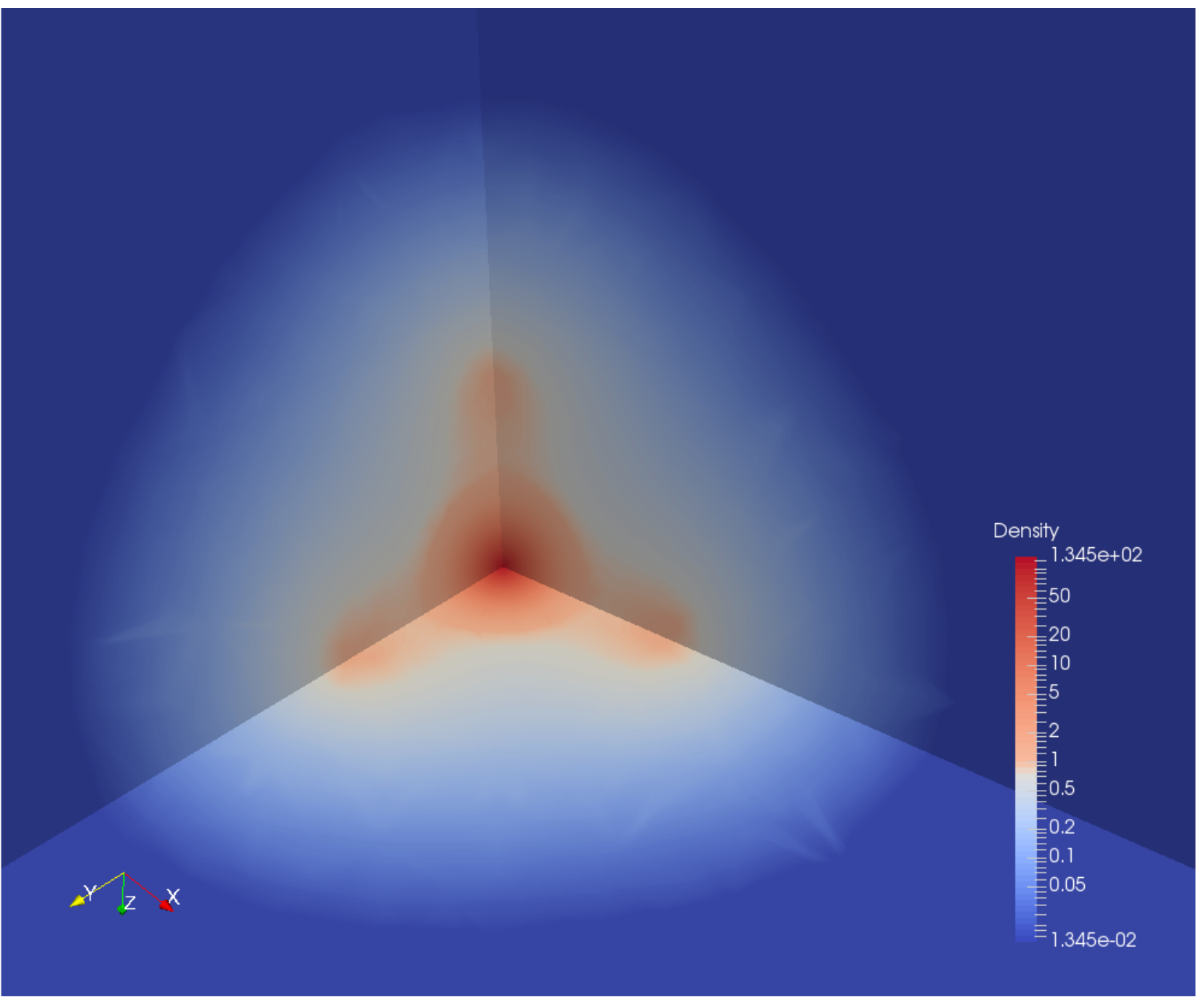}
      \includegraphics[width = 0.49\textwidth]{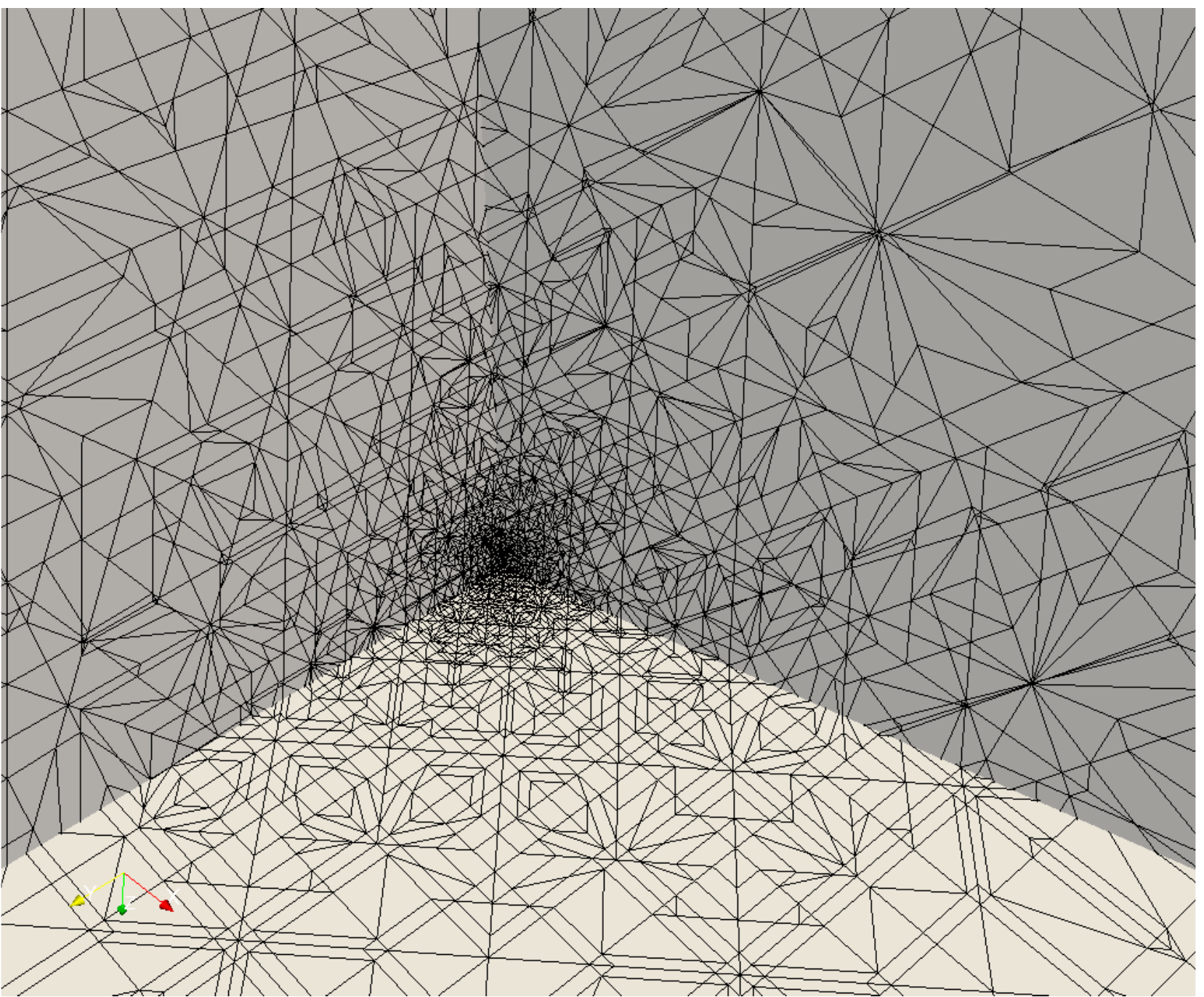}
      \caption{CH\textsubscript{4}: Output density(Row 1 Column 1: 0th iteration; Row 1 Column 2: 640th iteration; Row 2 Column 1: 10580th iteration;) and input grid(Row 2 Column 2).}\label{Fig: CH4Plus}
    \end{figure}

    \begin{figure}
    \centering
      \includegraphics[width = 0.49\textwidth]{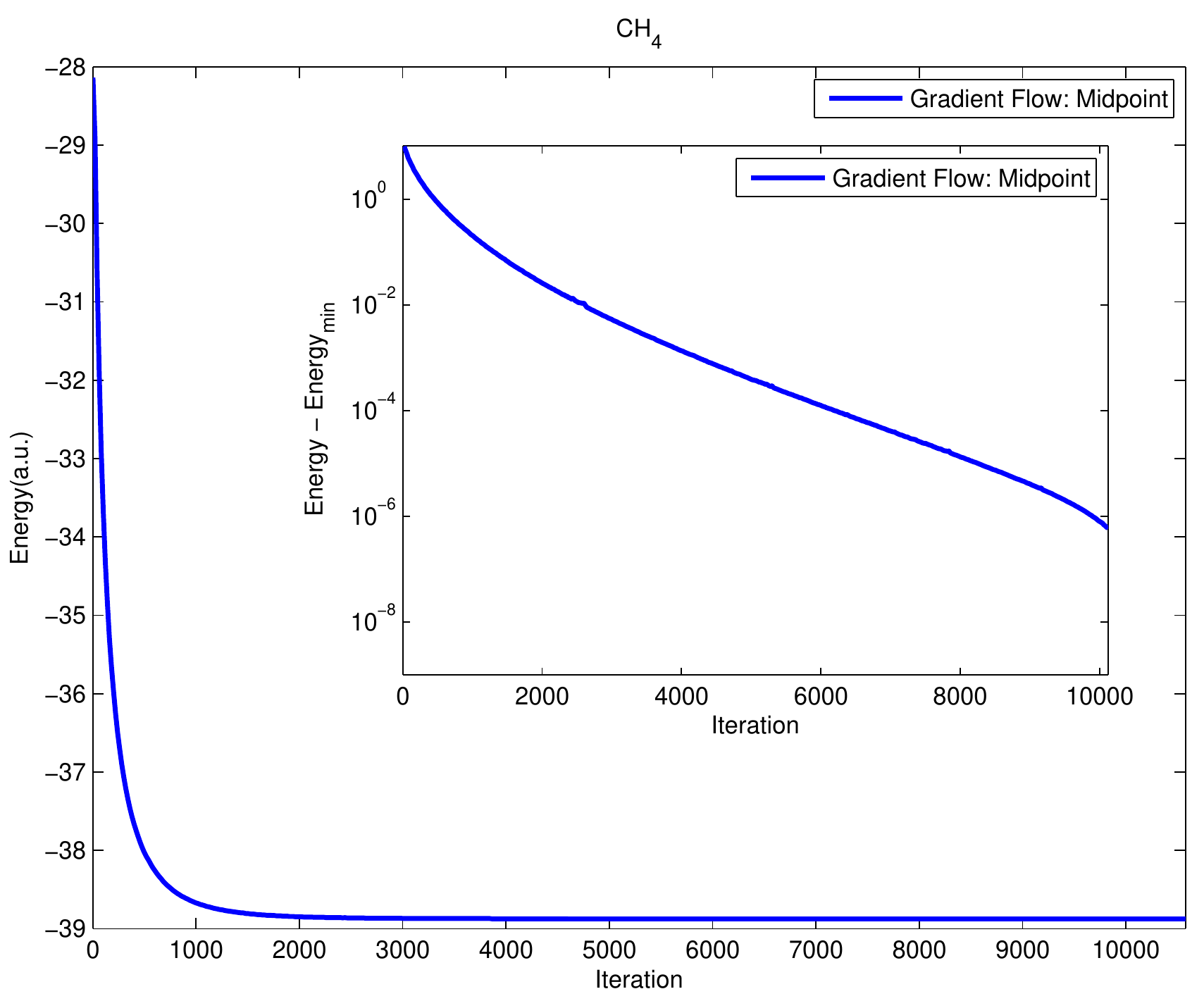}
      \includegraphics[width = 0.49\textwidth]{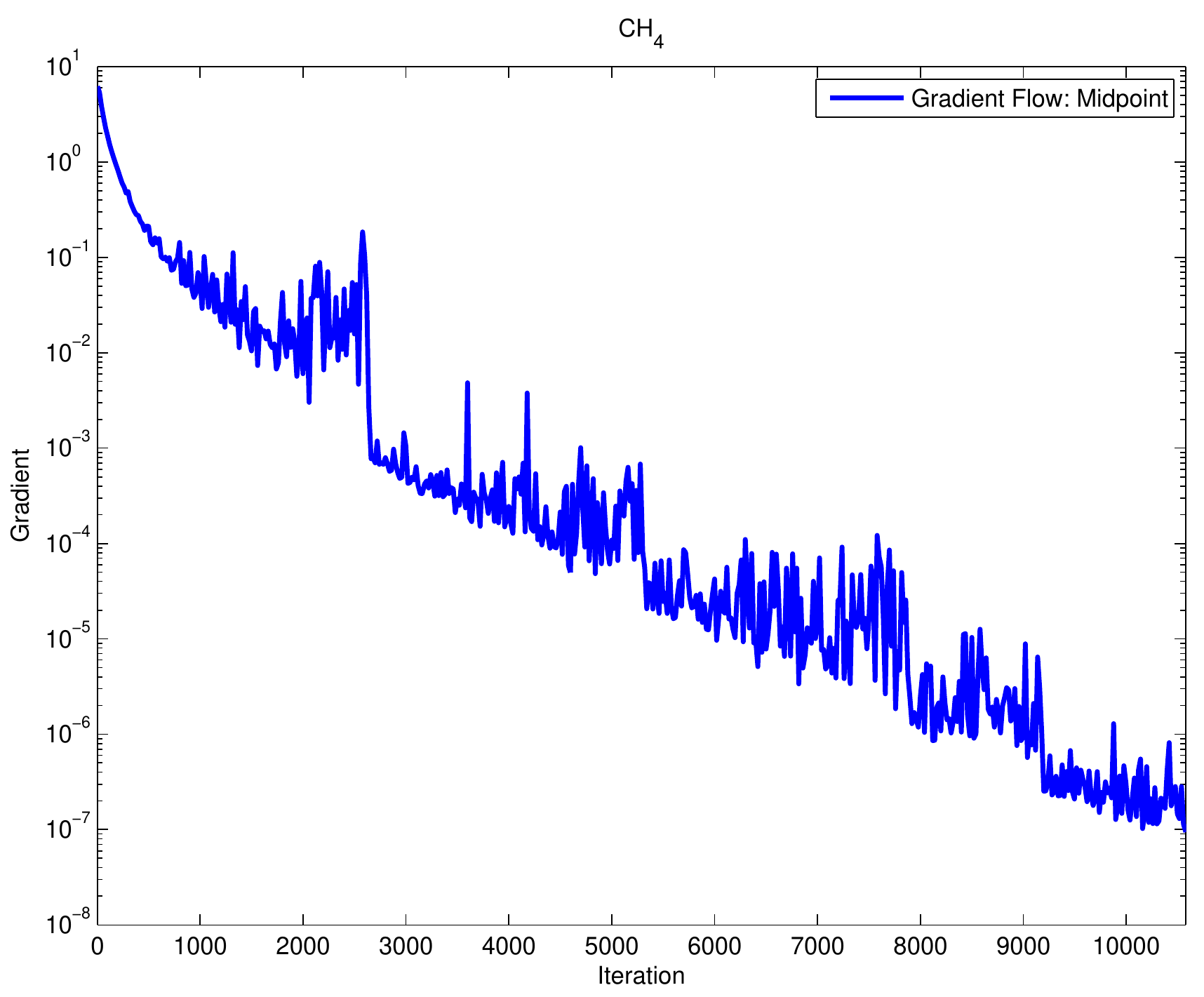}
      \caption{Convergence curves for energy(left) and gradient(right) for CH\textsubscript{4}.}\label{Fig: CH4}
    \end{figure}

    \emph{Example} 3. We choose a fixed tetrahedral finite element mesh on $[-56, 55]\times [-54, 53]\times[-54, 53]\subset\mathbb{R}^3$ from an adaptive refinement finite element method\cite{AFEM} with degrees of freedom $N_g = 16531$ and apply the gradient flow based model to compute the ground state of ethyne(C\textsubscript{2}H\textsubscript{2}) with orbits number $N = 7$(see Figure \ref{Fig: C2H2Plus}). We observe from Figure \ref{Fig: C2H2Plus} that the approximations of electron density converge. And similar to the examples above, the convergence curve of the approximated energy and the approximated gradient behaves as expected in Figure \ref{Fig: C2H2}.

    \begin{figure}
    \centering
      \includegraphics[width = 0.49\textwidth]{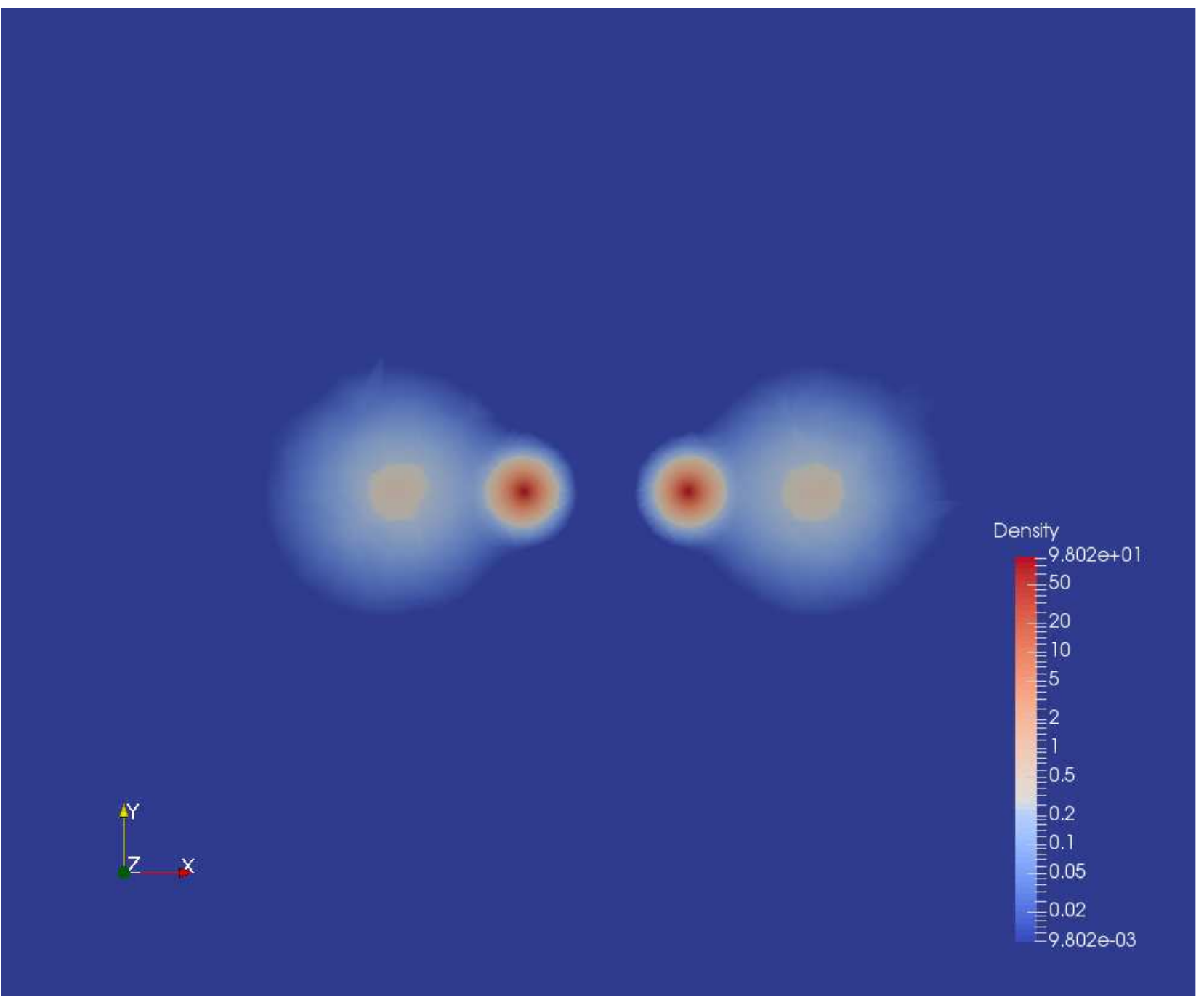}
      \includegraphics[width = 0.49\textwidth]{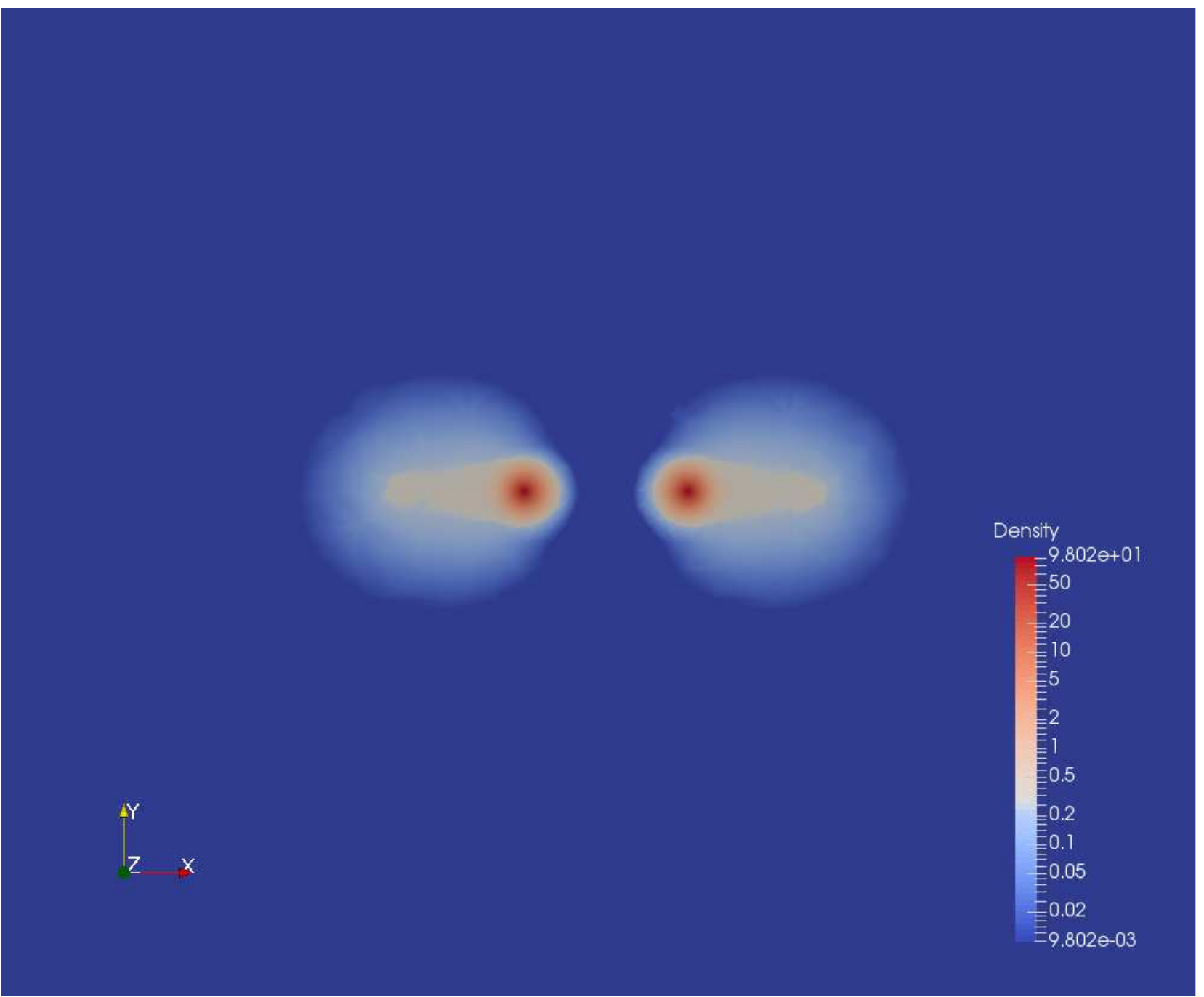}
      \includegraphics[width = 0.49\textwidth]{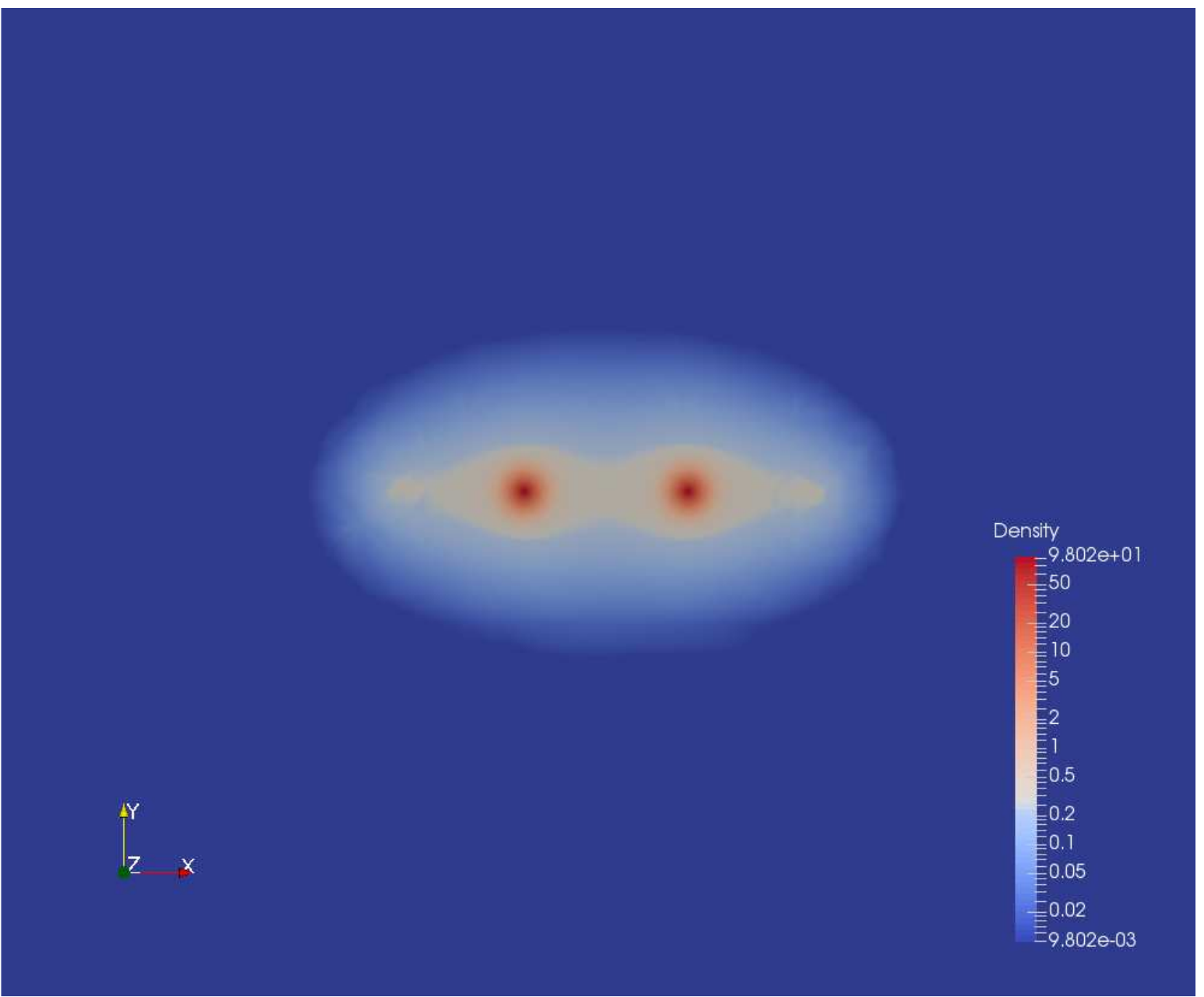}
      \includegraphics[width = 0.49\textwidth]{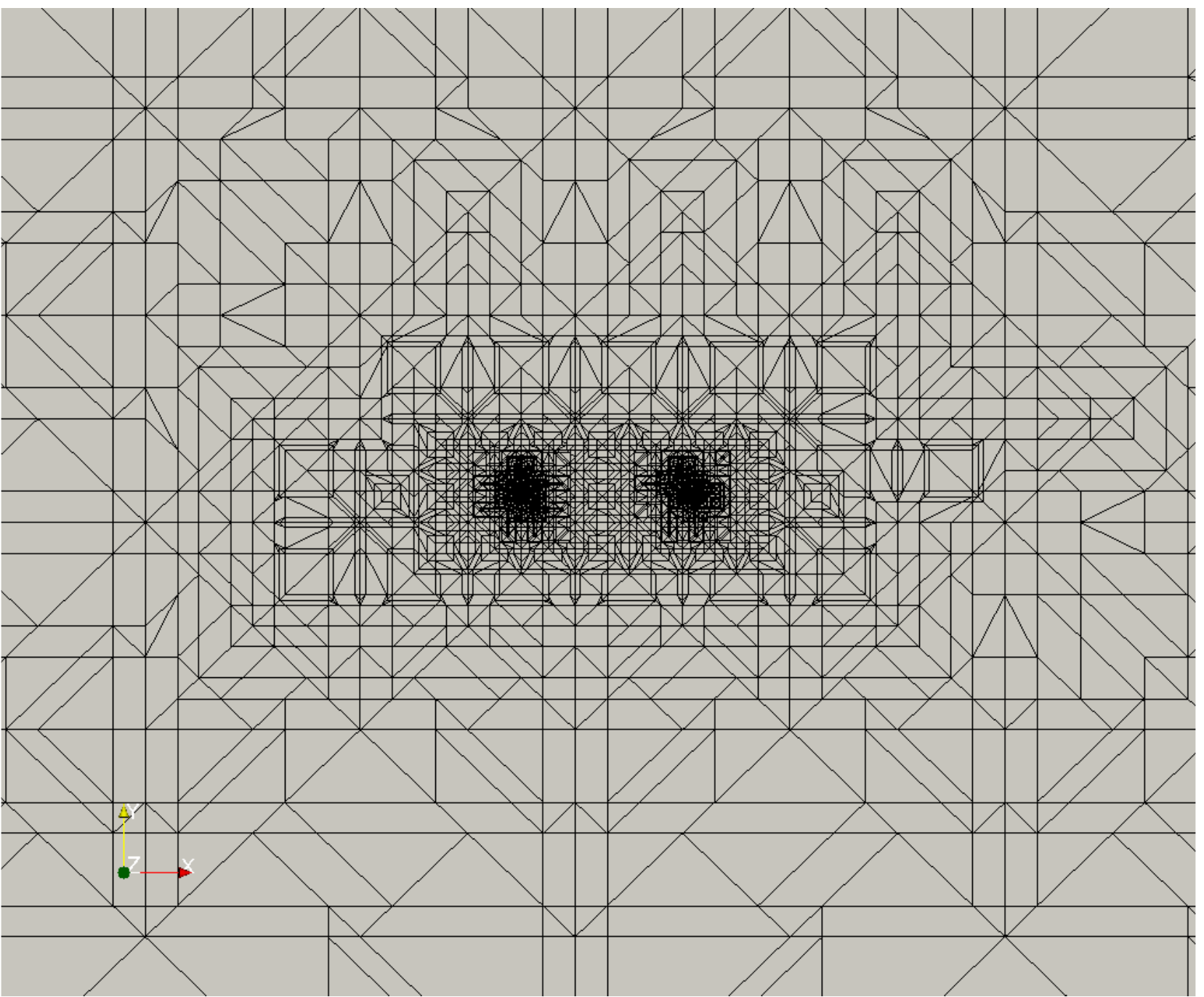}
      \caption{C\textsubscript{2}H\textsubscript{2}: Output density(Row 1 Column 1: 0th iteration; Row 1 Column 2: 580th iteration; Row 2 Column 1: 15160th iteration;) and input grid(Row 2 Column 2).}\label{Fig: C2H2Plus}
    \end{figure}

    \begin{figure}
    \centering
      \includegraphics[width = 0.49\textwidth]{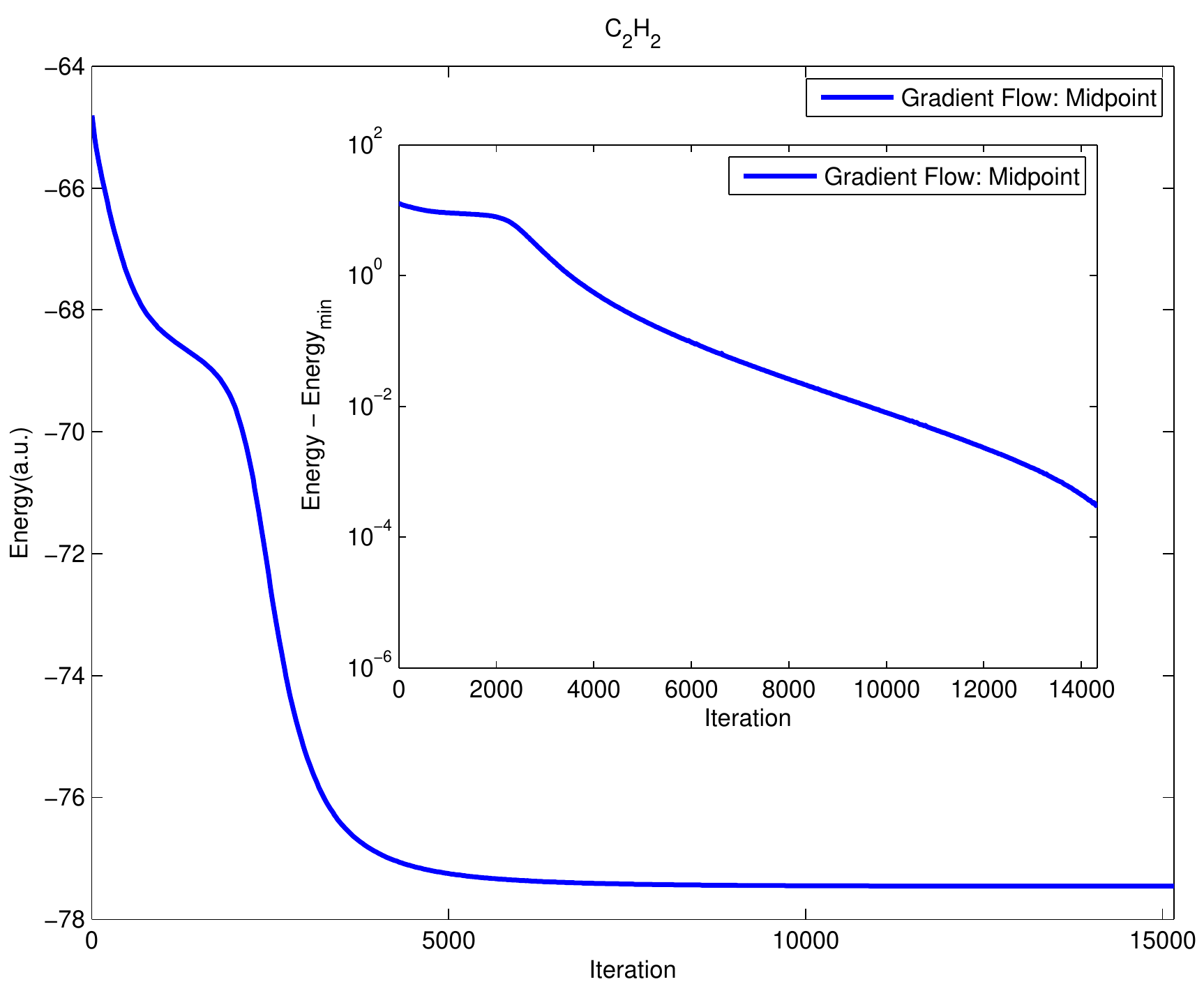}
      \includegraphics[width = 0.49\textwidth]{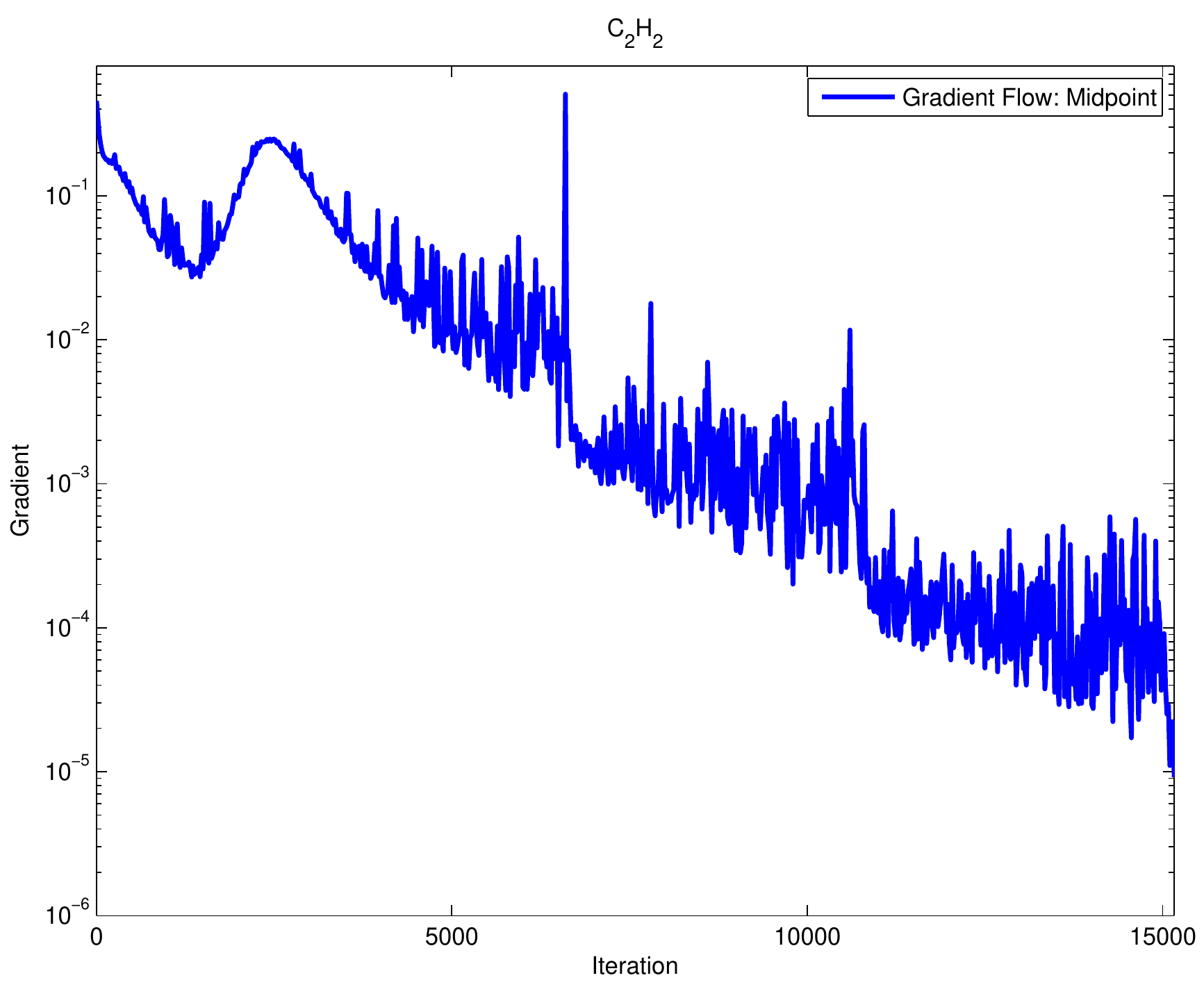}
      \caption{Convergence curves for energy(left) and gradient(right) for C\textsubscript{2}H\textsubscript{2}.}\label{Fig: C2H2}
    \end{figure}

    \emph{Example} 4. We apply the gradient flow based model to compute the ground state of benzene(C\textsubscript{6}H\textsubscript{6}) with orbits number $N = 21$ on a fixed tetrahedral finite element mesh on $[-56, 55]\times [-54, 53]\times[-54, 53]\subset\mathbb{R}^3$ generated by an adaptive refinement finite element method\cite{AFEM} with degrees of freedom $N_g = 20541$(see Figure \ref{Fig: C6H6Plus}). We see from Figure \ref{Fig: C6H6Plus} that the approximations of electron density are convergent. We understand from Figure \ref{Fig: C6H6} that the approximations of energy converges monotonically and the lower limit of the norm of the gradient approximations converge to zero.

    \begin{figure}
    \centering
      \includegraphics[width = 0.49\textwidth]{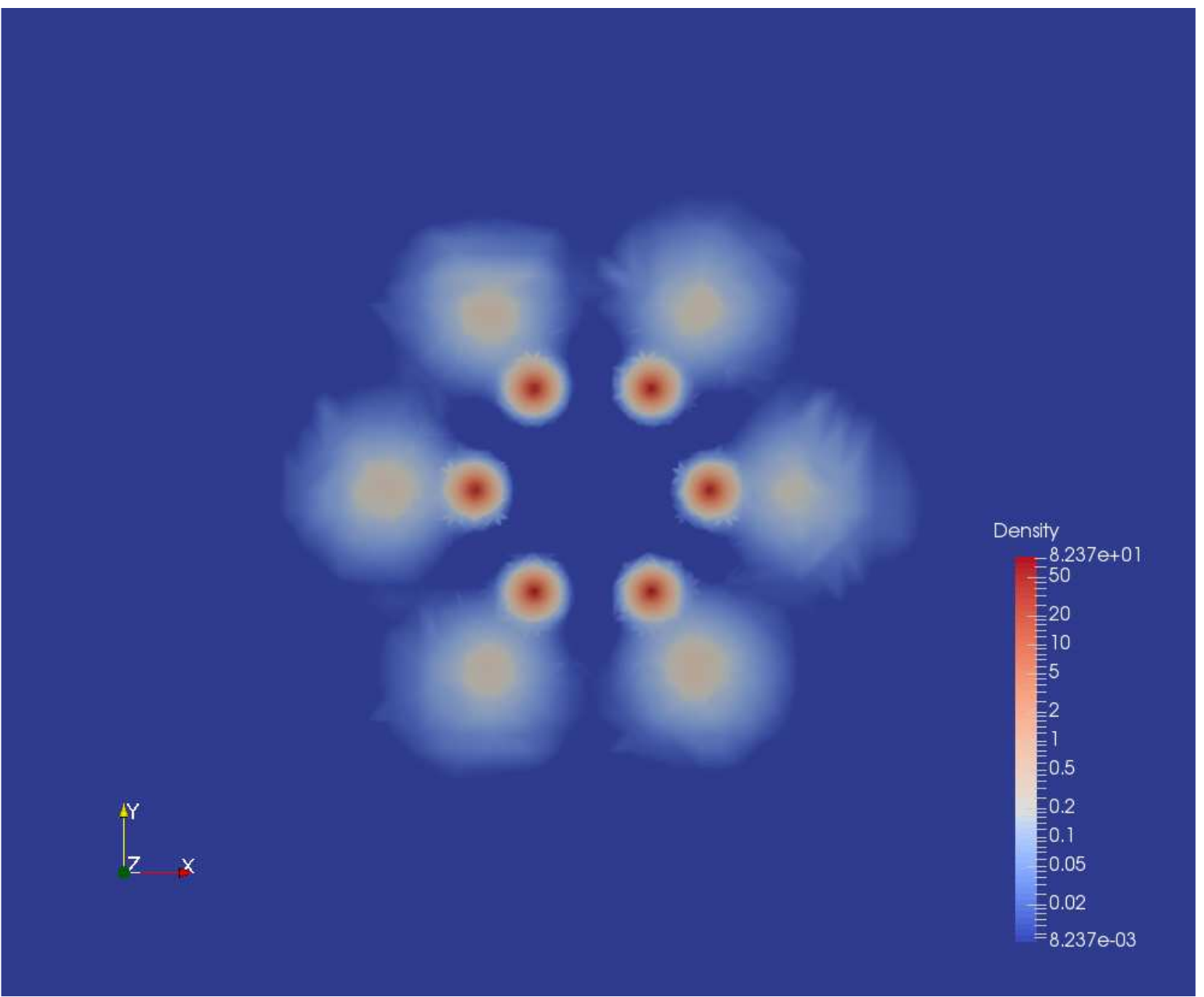}
      \includegraphics[width = 0.49\textwidth]{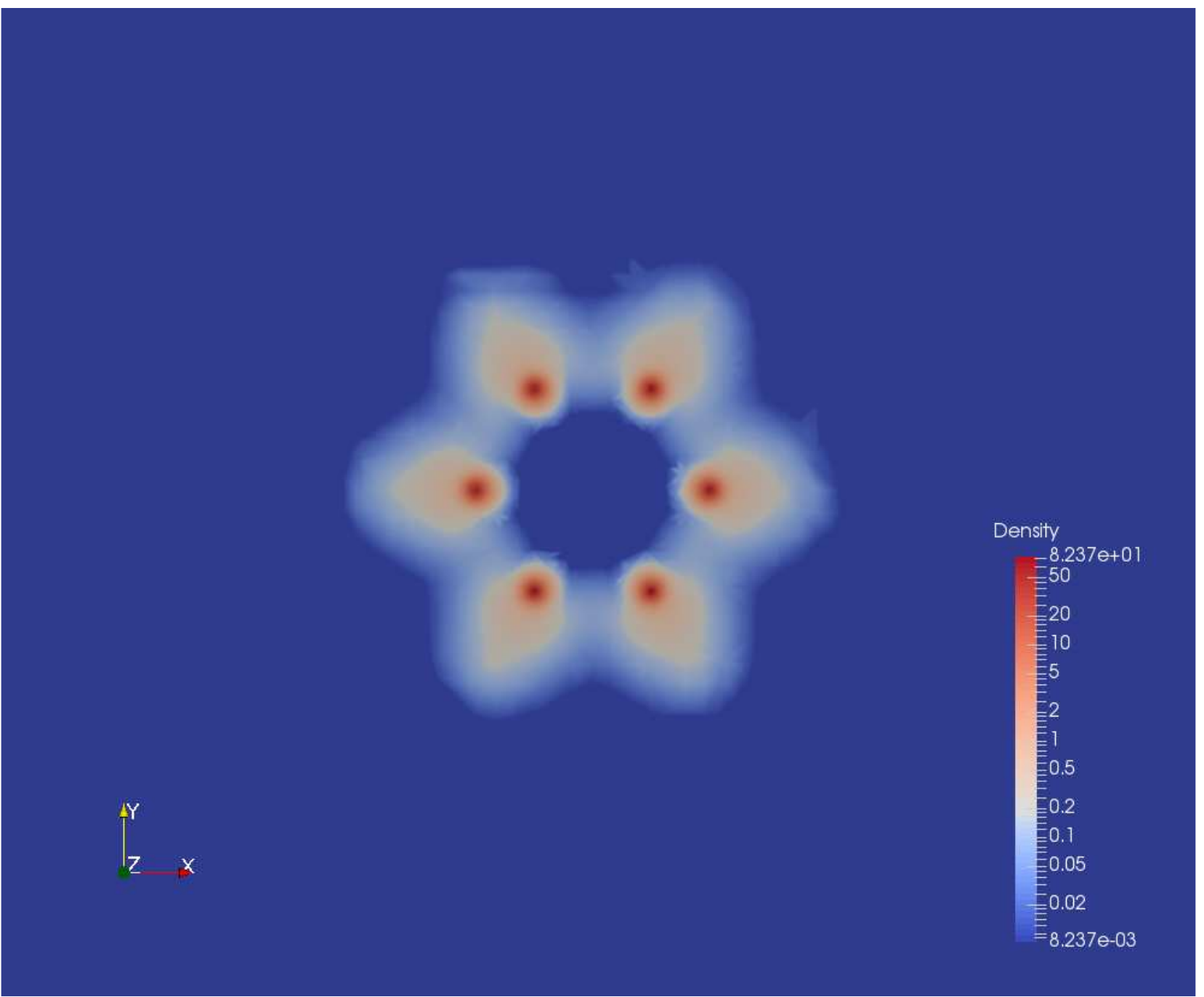}
      \includegraphics[width = 0.49\textwidth]{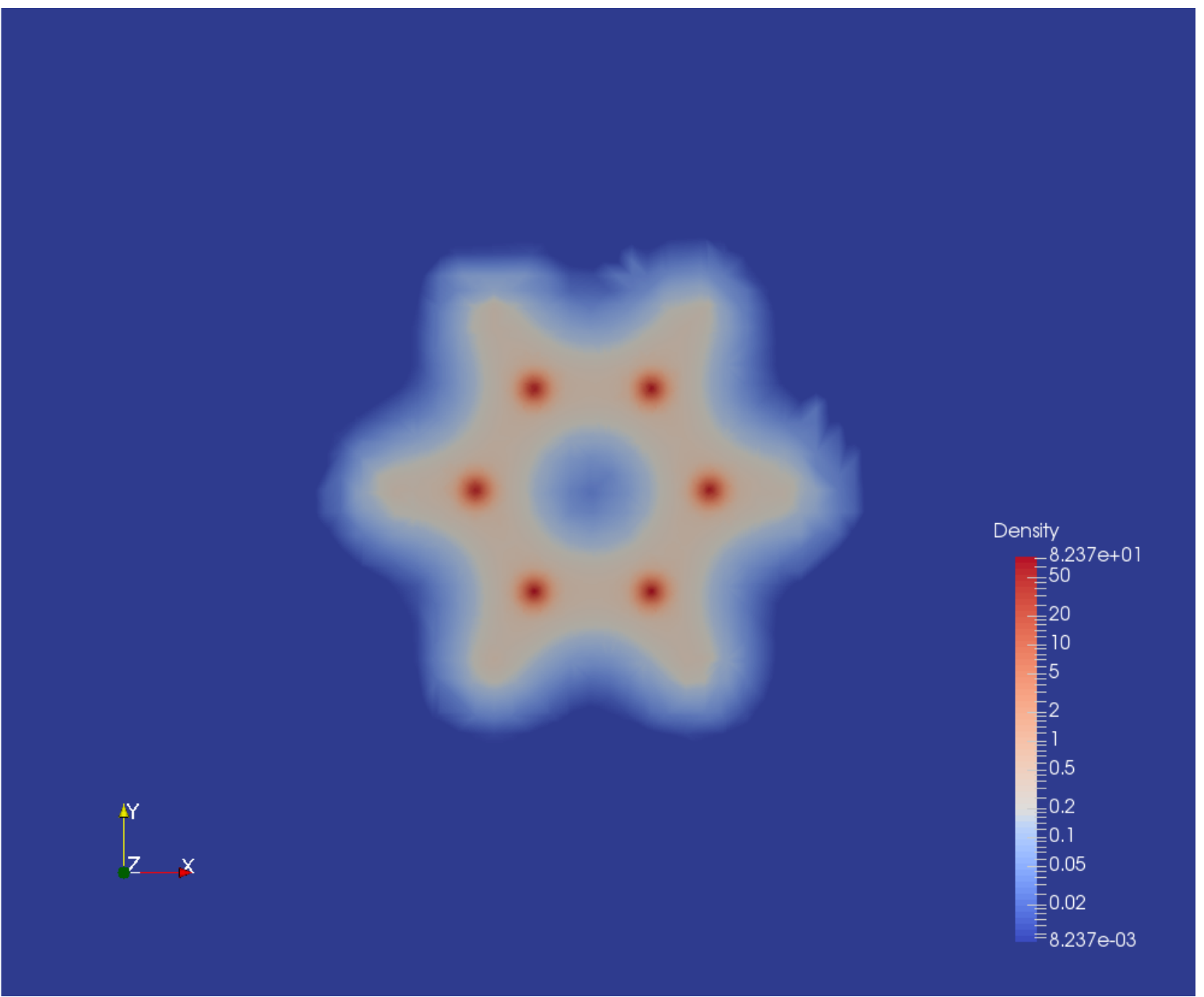}
      \includegraphics[width = 0.49\textwidth]{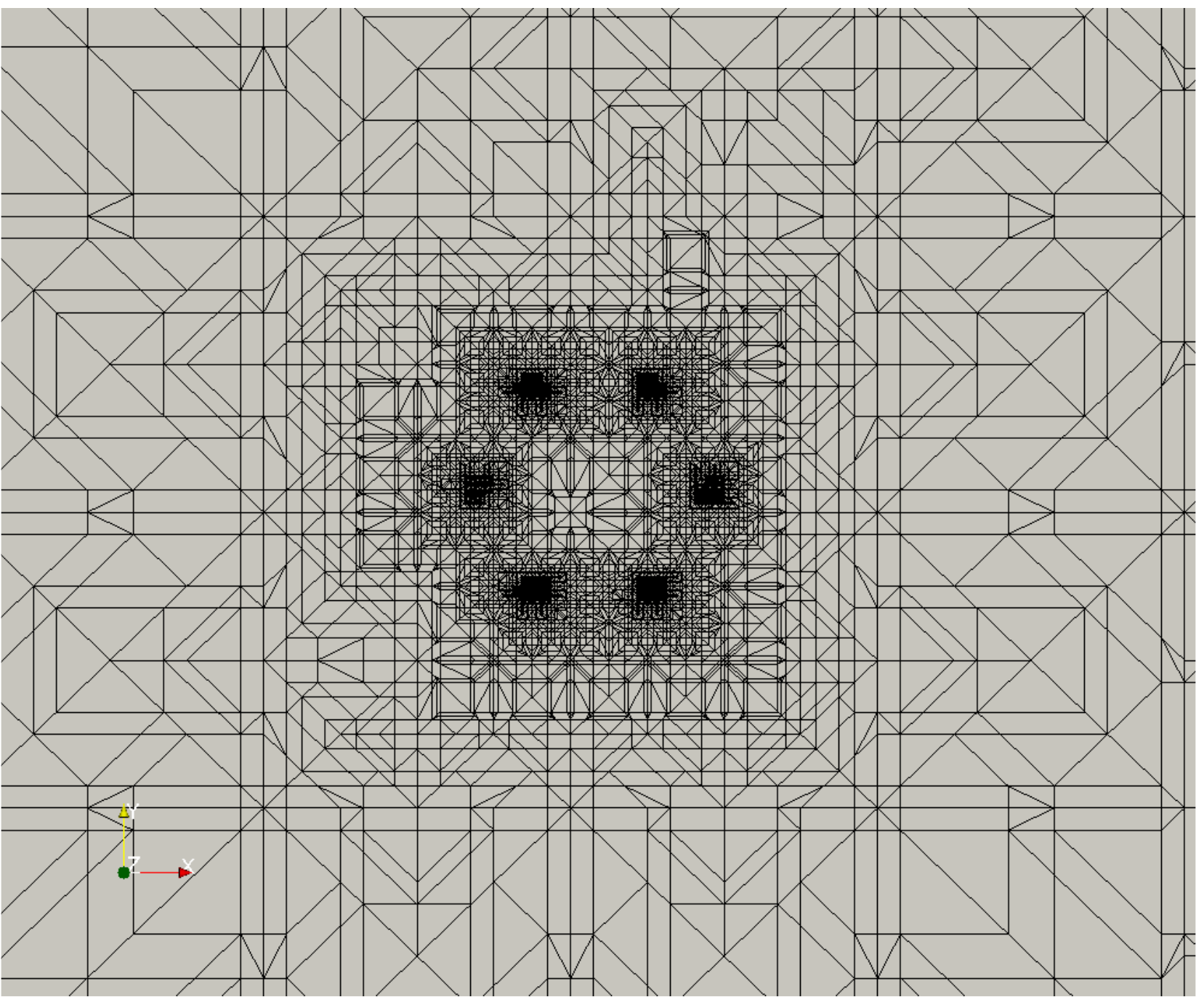}
      \caption{C\textsubscript{6}H\textsubscript{6}: Output density(Row 1 Column 1: 0th iteration; Row 1 Column 2: 300th iteration; Row 2 Column 1: 7460th iteration;) and input grid(Row 2 Column 2).}\label{Fig: C6H6Plus}
    \end{figure}

    \begin{figure}
    \centering
      \includegraphics[width = 0.49\textwidth]{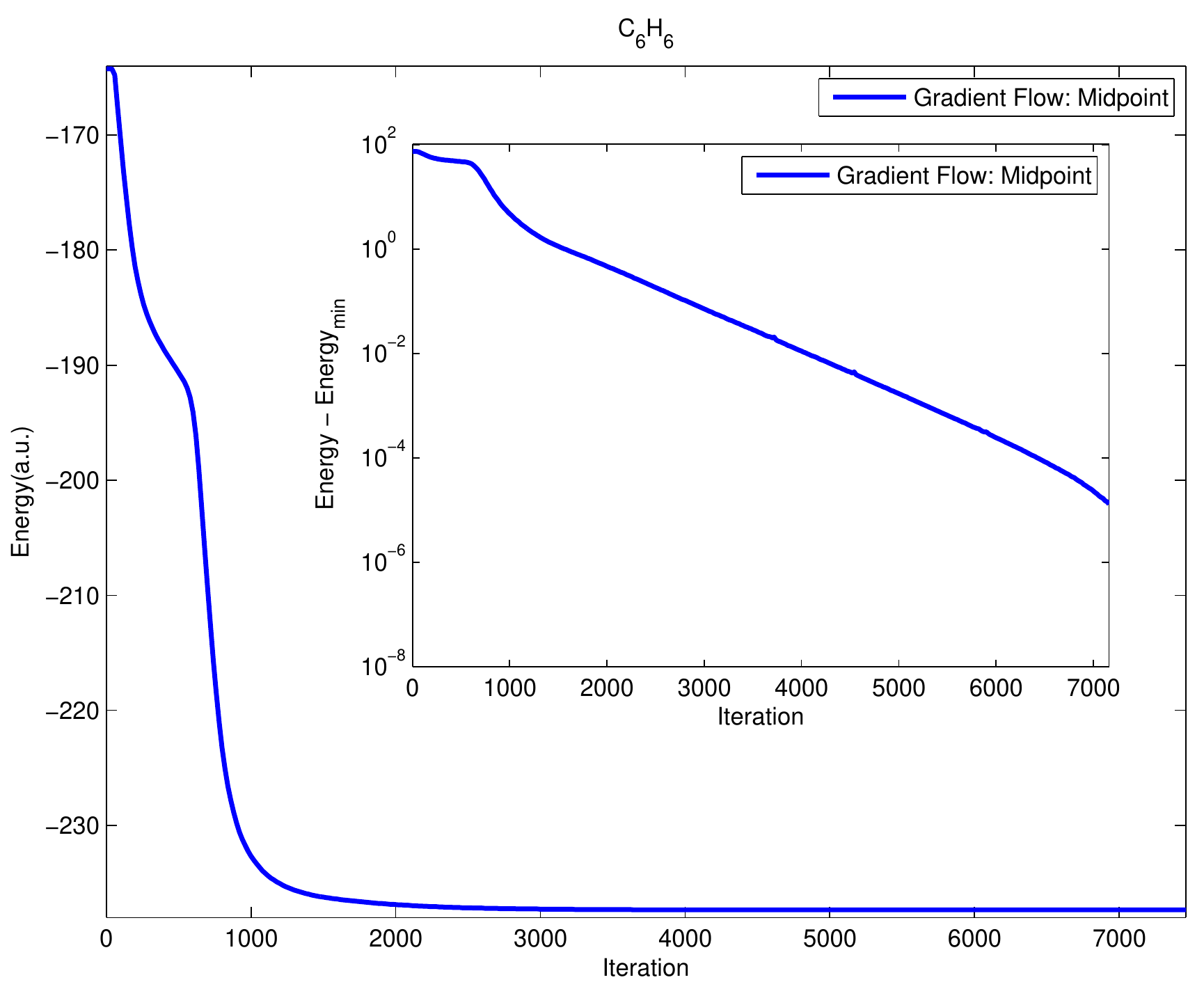}
      \includegraphics[width = 0.49\textwidth]{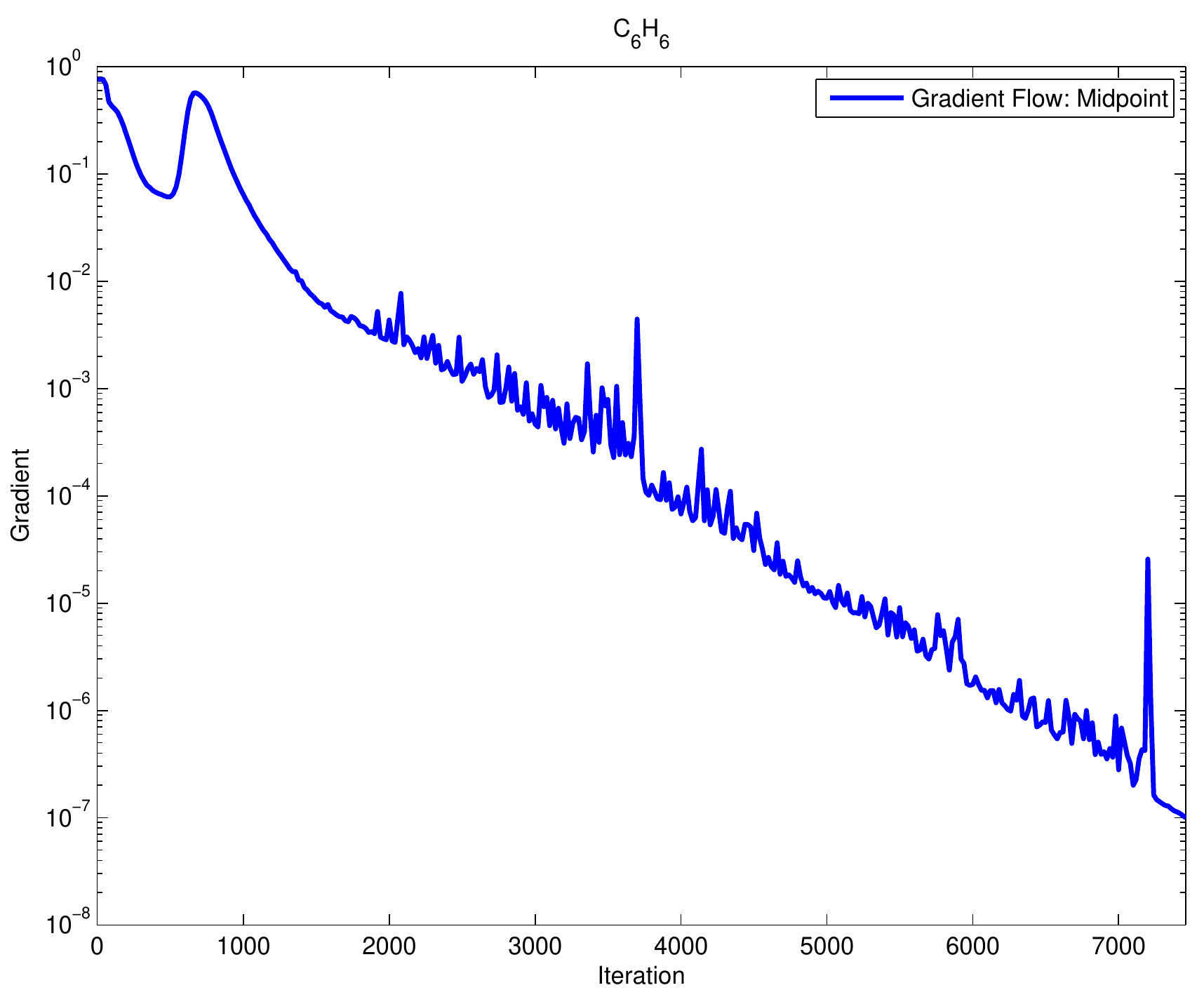}
      \caption{Convergence curves for energy(left) and gradient(right) for C\textsubscript{6}H\textsubscript{6}.}\label{Fig: C6H6}
    \end{figure}

     \emph{Examples} 1-4 indicate that our orthogonality preserving iterations of the gradient flow based model (Algorithm \ref{Alg: Self-consistent}) work well in ground state calculations.

    \section{Concluding remarks}

    In this paper, we have proposed and analyzed a gradient flow based model of Kohn-Sham DFT, which is an alternative way to solve Kohn-Sham DFT apart from the existing eigenvalue model with SCF iterations and the energy minimization model with optimization approaches. First we have established a continuous dynamical system based on the extended  gradient flow and proven that the solution remains on the Stiefel manifold, and then we have proven the local convergence of the dynamical system. Apart from that, local convergence rate can be further estimated if the Hessian is coercive locally. Second, we have come up with a midpoint scheme to discretize the dynamical system in the temporal direction and proven that it preserves orthogonality. We should mention that the auxiliary updating points of the midpoint scheme distribute inside the Stiefel manifold while those of retraction optimization methods distribute outside the Stiefel manifold. Compared with manifold path optimization methods diminishing energy locally \cite{Str_Pre2}, our midpoint scheme is a global approximation of the gradient on the step size interval. We also have proven the local convergence and estimated the convergence rate of the midpoint scheme under mild assumptions. In particular, based on the midpoint scheme, we have then proposed and analyzed an orthogonality preserving iteration scheme for the Kohn-Sham model and proven that the scheme is convergent under mild assumptions and the corresponding convergence rate can be estimated. Without annoying orthogonality preserving strategy and backtracking in optimization model and divergence of small gap systems in SCF iterations of nonlinear eigenvalue model, the gradient flow based model of Kohn-Sham DFT is promising.  It is worthwhile to look into the relationship between our orthogonality preserving scheme from the gradient flow based model and the conventional self-consistent field iteration from the nonlinear eigenvalue model. Moreover, our gradient flow based model can be extended to other models in electronic structure calculations such as Hartree-Fock type models. In this paper, we have mainly discussed the midpoint scheme to discretize the gradient flow based model. We may study other orthogonality preserving discretizations in temporal, such as the leapfrog scheme. Finally, we should mention that it is very useful if the convergence of the approximations of the gradient flow based model can be speed up, which is indeed our on-going work.

    \bibliographystyle{siam}
    \bibliography{reference}
\end{document}